\documentclass[a4paper, USenglish, cleveref]{lipics-v2021}
\pdfoutput=1

\pdfoutput=1 
\hideLIPIcs  

\title{Graph-Based Product Form}

\author{C\'eline Comte\footnotemark[1]}%
{LAAS--CNRS, Universit\'e de Toulouse, CNRS, Toulouse, France}%
{celine.comte@cnrs.fr}%
{https://orcid.org/0009-0005-9413-7124}%
{}

\author{Isaac Grosof\footnote{Authors are joint first authors, and are written in alphabetical order.}}%
{%
	Industrial Engineering and Management Science, Northwestern University, Evanston, IL, USA \and
	Electrical and Computer Engineering, University of Illinois, Urbana-Champaign, Urbana, IL, USA%
}%
{izzy.grosof@northwestern.edu}%
{https://orcid.org/0000-0001-6205-8652}
{Supported by AFOSR Grant FA9550-24-1-0002 and by a startup package from Northwestern University} 

\authorrunning{C. Comte and I. Grosof}

\Copyright{C. Comte and I. Grosof}



\nolinenumbers 

\usepackage{algorithm}%
\usepackage{algorithmicx}%
\usepackage{algpseudocode}%

\usepackage{dsfont}  
\newcommand{\indicator}[1]{ \mathds{1} [ #1 ] }

\usepackage{tikz}
\usetikzlibrary{calc}
\usetikzlibrary{arrows.meta}  
\usetikzlibrary{shapes,positioning}

\tikzset{
	state/.style={draw, circle},
	phantom/.style={state, fill=gray!50},
	newedge/.style={very thick},
}

\newcommand{\batchfirstnodes}{
	\foreach \i in {0, 1, ..., 7} {
		\node[state] (\i) at (\i, 0) {$\i$};
	}
	
	\foreach \i in {1, 2, ..., 7} {
		\node[phantom] (bar\i) at (\i, -1) {$\bar \i$};
	}
	
	\node (8) at (8, 0) {$\cdots$};
	\node (bar8) at (8, -1) {$\cdots$};
}

\newcommand{\batchfirstcliques}{
	\def\margin{.15cm}
	
	\draw[dashed, fill=blue!30, fill opacity=.3]
	($(0.north west)+(-\margin, \margin)$)
	-- ($(1.north east)+(\margin, \margin)$)
	-- ($(bar1.south east)+(\margin, -\margin)$)
	-- ($(bar1.south west)+(0, -\margin)$)
	-- ($(0.south west)+(-\margin, 0)$)
	-- ($(0.north west)+(-\margin, \margin)$);
	
	\draw[dashed, fill=blue!30, fill opacity=.3]
	($(2.north west)+(-\margin, \margin)$)
	-- ($(2.north east)+(\margin, \margin)$)
	-- ($(bar2.south east)+(\margin, -\margin)$)
	-- ($(bar2.south west)+(-\margin, -\margin)$)
	-- ($(2.north west)+(-\margin, \margin)$);
	
	\draw[dashed, fill=blue!30, fill opacity=.3]
	($(3.north west)+(-\margin, \margin)$)
	-- ($(4.north east)+(\margin, \margin)$)
	-- ($(bar4.south east)+(\margin, -\margin)$)
	-- ($(bar3.south west)+(-\margin, -\margin)$)
	-- ($(3.north west)+(-\margin, \margin)$);
	
	\draw[dashed, fill=blue!30, fill opacity=.3]
	($(5.north west)+(-\margin, \margin)$)
	-- ($(5.north east)+(\margin, \margin)$)
	-- ($(bar5.south east)+(\margin, -\margin)$)
	-- ($(bar5.south west)+(-\margin, -\margin)$)
	-- ($(5.north west)+(-\margin, \margin)$);
	
	\draw[dashed, fill=blue!30, fill opacity=.3]
	($(6.north west)+(-\margin, \margin)$)
	-- ($(7.north east)+(\margin, \margin)$)
	-- ($(bar7.south east)+(\margin, -\margin)$)
	-- ($(bar6.south west)+(-\margin, -\margin)$)
	-- ($(6.north west)+(-\margin, \margin)$);
}

\usepackage{soul}

\newcommand{\remove}[1]{\textcolor{red!30}{#1}}
\newcommand{\add}[1]{\answerstyle{#1}}
\newcommand{\replace}[2]{\remove{#1}\add{#2}}
\renewcommand{\remove}[1]{}
\renewcommand{\add}[1]{#1}
\renewcommand{\replace}[2]{#2}

\newcommand\qa{q^{\mathsf{a}}}
\newcommand\qb{q^{\mathsf{b}}}
\newcommand\pia{\pi^{\mathsf{a}}}
\newcommand\pib{\pi^{\mathsf{b}}}
\newcommand\pa{p^{\mathsf{a}}}
\newcommand\pb{p^{\mathsf{b}}}

\begin{document}

\maketitle

\begin{abstract}
	Product-form stationary distributions in Markov chains have been a foundational advance and driving force in our understanding of stochastic systems. In this paper, we introduce a new product-form relationship that we call ``\replace{graph-based product-form}{graph-based product form}''. As our first main contribution, we prove that two states of the Markov chain are in graph-based product form if and only if the following two equivalent conditions are satisfied: (i) a cut-based condition, reminiscent of classical results on product-form queueing systems, and (ii) a novel characterization that we call joint-ancestor freeness. The latter characterization allows us in particular to introduce a graph-traversal algorithm that checks product-form relationships for all pairs of states, with time complexity $O(|V|^2 |E|)$, if the Markov chain has a finite transition graph $G = (V, E)$. We then generalize graph-based product form to encompass more complex relationships, which we call ``\replace{higher-level product-form}{higher-level product form}'', and we again show these can be identified via a graph-traversal algorithm when the Markov chain has a finite state space. Lastly, we identify several examples from queueing theory that satisfy this product-form relationship.
\end{abstract}

\section{Introduction} \label{sec:intro}
Important classes of queueing systems and stochastic networks have been shown to have a so-called \emph{product-form stationary distribution}, where the stationary probability of a given state has a simple multiplicative relationship to the stationary probability of other nearby states.
The product-form property allows the stationary distribution of these systems to be cleanly and precisely characterized in closed-form, which is not possible for many other queueing systems. Such characterization has been instrumental in numerically evaluating performance \cite{B73,R80} and analyzing scaling regimes \cite{K21,C24-1,C24-2}, and it was more recently applied in the context of reinforcement learning \cite{A24,C24}.

Important queueing systems and stochastic networks that have been proven to exhibit product-form behavior include Jackson networks~\cite{J57}, BCMP networks~\cite{BCMP75}, Whittle networks~\cite{serfozo}, and networks of order-independent queues~\cite{BKK95} (also see~\cite{K11}).
Similar structures have appeared in other fields, such as statistical physics, with the zero-range process~\cite{A82}.
Discovering these categories of product-form systems and the underlying properties that give rise to their product-form behavior has represented a foundational advance and driving force in our understanding of stochastic systems.

Product-form results are often tied to time reversibility or quasi-reversibility properties. These can be established through the detailed balance property given the stationary distribution,
as well as by applying Kolmogorov's criterion on the transition rates of the Markov chain \cite{kelly}. In the simplest case of birth-and-death processes, both product-form and reversibility are implied by the transition diagram of the Markov chain.

However, there are important queueing systems that exhibit product-form behavior which cannot be explained under any existing product-form framework.
A motivating example for this paper is the multiserver-job saturated system with two job classes, which Grosof et al.~\cite{GHS23} demonstrated to have a product-form stationary distribution.
We explore this system via our novel framework in \Cref{sec:msj-example}.

We introduce a new kind of product-form Markov \replace{chains}{chain},
\emph{graph-based} product form.
In these Markov chains, product-form arises purely from the connectivity structure of the transition graph,
or in other words from the set of transitions with nonzero probability, for discrete-time Markov chains, abbreviated DTMCs; or nonzero rate, for continuous-time Markov chains, abbreviated CTMCs.
If a Markov chain has graph-based product form,
that product form holds regardless of its transition probabilities or rates,
under a given connectivity structure.
This is in contrast to most prior classes of product-form relationships,
where tweaking a single transition probability or rate would remove the product-form property.

In this paper, we characterize which directed graphs hold the correct structure to give rise to graph-based product-form Markov chains.
Our characterization is built up from a product-form relationship between states
(i.e., nodes in the transition diagram),
which exists when the ratio of the stationary probabilities for two nodes forms a simple multiplicative relationship, arising from the graph structure.
If those relationships span the graph, then the whole Markov chain has graph-based product form.
We therefore focus on characterizing which graphs give rise to product-form relationships between
a given pair of nodes.

In our main result, \Cref{theo:s-product-form},
we give two \add{equivalent} necessary and sufficient conditions
under which such a product-form relationship exists:
a cut-based characterization,
reminiscent to classical conditions for product-form,
and a novel characterization which we call \emph{joint-ancestor freeness}. More specifically, focusing on two particular nodes or states~$i$ and~$j$:
\begin{itemize}
	\item \emph{Cut-based characterization}:
	For the first condition,
	we show that if there exists a cut (i.e.\ a partition of the nodes into two sets),
	where $i$ is on one side of the cut and $j$ is on the other,
	and where the only edges that cross the cut have either $i$ or $j$ as their source nodes,
	then $i$ and $j$ have a product-form relationship.
	We call such a cut an \emph{$i, j$-sourced cut}.
	Unfortunately, directly searching for such cuts is inefficient and impractical,
	as there are exponentially many cuts in the graph.
	\item \emph{Joint-ancestor freeness}:
	We show that the existence of such a cut is equivalent to a second, simpler-to-check property,
	which we call joint-ancestor freeness.
	We refer to a node $k$ as a joint ancestor of $i$ and $j$ if there exists a path from $k$ to $i$ which does not go through $j$, and a path from $k$ to $j$ which does not go through $i$.
	We show that the existence of an $i,j$-sourced cut is equivalent to $i$ and $j$ having no joint ancestors,
	which is efficient to directly search for.
\end{itemize}
Finally, we show in \Cref{theo:no-product-form} that this relationship is bidirectional:
If there is no $i,j$-sourced cut, or equivalently if there is a joint ancestor $k$,
then nodes $i$ and $j$ will not have a straightforward product-form relationship.

Even in graphs where the most straightforward product-form relationships do not connect every pair of nodes,
a less-direct kind of product-form relationship can still exist,
which we call ``\replace{higher-level product-form}{higher-level product form}''.
We call the above product-form relationships ``first-level product-form'',
and we show that a weaker, but still noteworthy, kind of product-form relationship,
``second-level product-form'',
exists whenever there exists a cut with multiple nodes as sources on one or both sides of the cut,
such that the sources on each side are connected by first-level product-form relationships.
Further levels can be defined recursively.
We study this higher-level product-form \replace{behavior}{relationship} in \Cref{sec:second-level,sec:second-level-product},
with \Cref{sec:batch1} as a motivating example.

\subsection{Contributions}

In \Cref{sec:product-form-def}, we define the novel concept
of \replace{graph-based product-form}{graph-based product form}.
In \Cref{sec:s-product-form}, and specifically \Cref{theo:s-product-form},
we prove that \replace{graph-based product-form}{graph-based product form}
between two nodes $i$ and $j$
is equivalent to two graph-based properties:
The existence of an $i,j$-sourced cut,
and the absence of a joint ancestor of $i$ and $j$.
In \Cref{sec:cut-graph,sec:ps-product-form}, we introduce the cut graph and its connection to \replace{graph-based product-form}{graph-based product form} spanning an entire Markov chain.
In \Cref{sec:second-level}, we explore and characterize \replace{higher-level product-form}{higher-level product form} relationships, which we show correspond to higher-level cuts.
In \Cref{sec:s-product-form-examples,sec:examples}, we give a variety of examples of graphs which do or do not have \replace{graph-based product-form}{graph-based product form}, and use them to illustrate our characterization.

\subsection{Prior work} \label{sec:prior-work}

There is a massive literature that focuses on deriving
the stationary distribution (or a stationary measure)
of Markov chains with countable
(finite or infinite) state spaces.
In reviewing this literature, we focus on results that
either provide a closed-form expression for the stationary measures
or make structural assumptions on the Markov chain,
or both.

\paragraph*{Reversibility, quasi-reversibility, and partial balance}

A long series of works has derived product-form stationary distributions
by focusing on Markov chains
where a stronger form of the balance equations holds,
thus balancing the probability flow between a state and each set
in a partition of its neighbors.
This is often equivalent to
properties of the time-reversed process \cite{serfozo}.
For example, the Kolmogorov criterion \cite[Theorem~2.8]{serfozo}
is a necessary and sufficient condition for reversibility,
which as a by-product yields a closed-form expression for the stationary distribution
as a product of transition rates.
Another example is quasi-reversibility,
as described in \cite[Chapter~3]{kelly}.
Among these works, many have focused on Markov chains
exhibiting a specific transition diagram,
e.g., multi-class queueing systems with arrivals and departures occurring one at a time,
and have identified necessary and sufficient conditions
on the transition rates
that yield a product-form stationary distribution.
This approach has therefore produced many models
applicable to queueing theory and statistical physics.
Reversible models and their variants involving internal routing include
the celebrated Jackson networks \cite{J57},
the zero-range process \cite{A82},
and Whittle networks \cite{serfozo}.
Quasi-reversibility has also given rise to multiple models,
including order-independent queues \cite{BKK95,BK96}
and pass-and-swap queues \cite{CD21}; see \cite{GR20} for a recent survey.
Other examples of queueing models
that satisfy partial balance equations are
token-based order-independent queues \cite{ABDV22}
and certain saturated multiserver-job queues \cite{RM17,GHS23}.

\paragraph*{Graph-based product form}

To the best of our knowledge,
very few papers exploit the structure
of a Markov chain's transition diagram
(rather than its transition \emph{rates})
to guarantee the existence of a product-form stationary distribution.
One example is \cite{F87}, which introduces
\textit{single-input super-state decomposable Markov chains}:
the Markov chain's state space is partitioned into a finite number of sets,
called superstates,
such that all edges into a superset
have the same node as endpoint.
(All finite-state-space Markov chains satisfy this condition
when the partition is formed by singletons.)
Under this assumption, the process of deriving the Markov chain's stationary distribution
can be divided into two steps,
one that solves the stationary distribution
of a Markov chain defined over the superstates,
and another that solves the stationary distribution
of a Markov chain restricted to each superstate.
While the superstate decomposition has a loose resemblance to our cuts,
there is no deeper similarity between the methods.
In particular,
our approach is nontrivial both for finite and infinite Markov chains.
The superstate decomposition approach can be seen as a different approach to deriving product-forms.

Closer to our work,
\add{Grosof et al.~\cite{GHS23}} consider\remove{s} a multiserver-job (MSJ) model
described by a CTMC
and show that it has a product-form stationary distribution
irrespective of the transition rates.
This result is proven in more detail in a technical report \cite{GHS20}.
This example, which inspired the present work,
is discussed in detail in \Cref{sec:msj-example}.

\paragraph*{\add{Matrix-geometric methods}}

\add{Another well-known family of methods that exploits structural properties
	of Markov chains is matrix-geometric methods~\cite{H93,L99,P11},
	developed for the analysis of quasi-birth-and-death (QBD) processes.
	A Markov chain is called a QBD process if its states
	can be partitioned into disjoint superstates indexed by $0, 1, 2, \ldots$,
	such that transitions occur either within a superstate (inner transitions),
	from a superstate $i \in \{0, 1, 2, \ldots\}$ to superstate $i+1$ (upward transition),
	or from a superstate $i \in \{1, 2, \ldots\}$ to superstate $i-1$ (downward transition).
	The graph-based product form we consider does have intersection with QBD processes,
	in the sense that there exist QBD processes that can be analyzed
	with the prism of \replace{graph-based product-form}{graph-based product form};
	see \Cref{ex:qbd,sec:msj-example,sec:batch1}.
	However, QBD processes and graph-based product form are two different notions:
	there are QBD processes that do not have a graph-based product form,
	and there are Markov chains that exhibit graph-based product form
	but are not QBD processes; see \Cref{ex:one-way-cycle,ex:one_way_cycle,ex:tree},
	\Cref{sec:batch-2}, and Appendix~\ref{app:clique}.
	Furthermore, as we can see in \Cref{sec:batch1},
	there are QBD processes for which the framework of graph-based product form
	allows us to derive the stationary distribution more directly,
	without resorting to matrix-geometric methods.
	Lastly, to the best of our knowledge,
	matrix-geometric methods are used mainly
	when the Markov chain shows repetitive patterns,
	i.e., when the inner (resp.\ upward, downward) transitions are similar across superstates,
	except for superstate~$0$;
	\replace{graph-based product-form}{graph-based product form} is not so much related to repetition
	as to exploiting the structure of the transition diagram.}

\paragraph*{Symbolic solutions}

Our graph-based product-form method
can also be seen \add{as} an algorithmic way to discover a particular
type of product-form \replace{relationships}{relationship} in Markov chains, giving a clean symbolic solution for the stationary distribution.
If the Markov chains are structured, as in the examples in \Cref{sec:examples}, these relationships can be found \replace{explicitly}{by direct inspection}.
However, if algorithm searching is required, we give an algorithmic approach to discover single-source cuts in the underlying graph in \Cref{algo:cut-graph} in $O(|V|^2 |E|)$ time, if the Markov chain has a finite transition graph $G = (V, E)$, allowing us to discover whether a product-form relationship exists.

Prior to our approach, one could symbolically find the stationary distribution for a general symbolic Markov chain in $O(|V|^2)$ time, by symbolically solving the balance equations. However, there is no guarantee that the resulting symbolic expression would be in a simple form.
Simplifying and factorizing the resulting symbolic expression, which might have $O(|V|^2)$ terms,
does not have a known efficient, deterministic algorithm.
In fact, polynomial factorization is a more complicated version of the polynomial identity testing (PIT) problem,
for which no polynomial-time deterministic algorithm is known \cite{A09};
the two problems were recently proven equivalent, in the sense that a deterministic polynomial-time algorithm for one would imply the same for the other \cite{K15}.
Finding such an algorithm has remained a major open problem.

\paragraph*{Other related methods}

Product-form stationary distributions
for DTMCs or CTMCs
have been studied in many different contexts,
such as Stochastic Petri networks,
which sometimes lead to constructive and algebraic methods
that assume particular structure of the transition rates~\cite{D92,B12}.
Orthogonally, the graph structure of a Markov chain
has also been used for other purpose
than deriving a simple closed-form expression
for the stationary distribution.
For instance, the survey \cite{H87} focuses on iterative methods
to approximate the stationary distribution of a finite
Markov chain with transition matrix~$A$
using updates of the form $\pi_{t+1} = \pi_t A$.
The algorithms described in \cite{H87},
called \textit{aggregation-disaggregation methods},
aim at speeding-up iterative methods
by occasionally replacing $\pi_{t+1}$ with $\tilde\pi_{t+1} = S(\pi_{t+1})$,
where $S$ is a function that exploits
structure in the Markov chain's transition diagram.

\section{Model and definitions}

We start by introducing preliminary graph notation and terminology in \Cref{sec:graph-notation}, in particular \emph{set-avoiding path\add{s}} and \emph{ancestor sets},
then \add{we} introduce the key notions of \replace{an}{a} \emph{formal Markov chain} and \emph{\replace{graph-based product-form}{graph-based product form}}
in \Cref{sec:product-form-def}.

\subsection{Graph notation and terminology}
\label{sec:graph-notation}

The focus of this paper is on the directed graphs
that underlie Markov chains and on cuts in these graphs.
Besides recalling classical graph-theoretic notions,
we introduce \emph{set-avoiding paths} and \emph{ancestor sets}
that will be instrumental in the rest of the paper.

A \emph{directed graph} is a pair $G = (V, E)$,
where $V$ is a countable set of nodes, and
$E \subseteq V \times V$ is a set of directed edges.
The graph $G$ is called \emph{finite} if $V$ is finite
and \emph{infinite} if $V$ is countably infinite.
A \emph{cut} of a directed graph~$G = (V, E)$ is a pair $(A, B)$
of nonempty sets that form a partition of~$V$, that is,
$A \cup B = V$ and $A \cap B = \emptyset$.
An edge $(u, v) \in E$ is then said to cross the cut $(A, B)$
if either $u \in A$ and $v \in B$, or $u \in B$ and $v \in A$.
A \emph{path} in a directed graph~$G = (V, E)$
is a sequence $v_1, v_2, \ldots, v_n$
of \emph{distinct} nodes in~$V$, with $n \in \{1, 2, \ldots\}$,
such that $(v_p, v_{p+1}) \in E$ for each $p \in \{1, 2, \ldots, n-1\}$;
the length of the path is the number $n-1$ of edges that form it.
In particular, a path of length~0 consists of a single node and no edges.
A graph~$G = (V, E)$ is called \emph{strongly connected} if,
for each $i, j \in V$,
there exists a path from node~$i$ to node~$j$ in the graph~$G$.

In \Cref{sec:product-form-def}, we will relate the existence of cuts that yield convenient balance equations with the following two definitions that will be instrumental in the paper.

\begin{definition}[Set-avoiding subgraph and set-avoiding path]
	Consider a directed graph~$G = (V, E)$
	and let $U \subseteq V$.
	The \emph{set-avoiding subgraph} $G {\setminus} U = (V {\setminus} U, E')$
	is defined with
	$E' = \{(i, j) \in E: i, j \notin U \}$.
	Given $i, j \in V {\setminus} U$,
	we let $P(i \to j {\setminus} U)$ denote
	an arbitrary path $v_1, v_2, \ldots, v_n$ in~$G$,
	with $n \in \{1, 2, \ldots\}$,
	with source node~$v_1 = i$ and destination node~$v_n = j$,
	and such that $v_p \notin U$ for each $p \in \{1, 2, \ldots, n\}$.
	Such a path is said to \emph{avoid} the set~$U$.
	Equivalently, a path $P(i \to j {\setminus} U)$
	is a path from node~$i$ to node~$j$ in the subgraph $G {\setminus} U$.
	If $U = \{u\}$ is a singleton,
	we write $G {\setminus} u$ for $G {\setminus} \{u\}$
	and $P(i \to j {\setminus} u)$ for $P(i \to j {\setminus} \{u\})$.
\end{definition}

\begin{definition}[Ancestor and ancestor set]
	Consider a directed graph $G = (V, E)$
	and let~$i, j \in V$.
	Node~$i$ is called an \emph{ancestor} of node~$j$ (in~$G$)
	if there exists a directed path from node~$i$ to node~$j$ (in~$G$), i.e.,
	if there exists a path $v_1, v_2, \ldots, v_n$ (in $G$)
	with $v_1 = i$ and $v_n = j$.
	For each $i \in V$, $A_i(G)$ denotes the set of ancestors of node~$i$ (in~$G$).
	For each $I \subseteq V$,
	$A_I(G) = \bigcup_{i \in I} A_i(G)$
	denotes the \emph{ancestor set} of node set~$I$ (in~$G$).
\end{definition}

The ancestor set of a node contains the node itself
(via a path of length zero),
so that $I \subseteq A_I(G)$ for each $I \subseteq V$.
A directed graph~$G = (V, E)$ is strongly connected if and only if
the ancestor set of each node is the whole set $V$.
Procedure \textsc{Ancestors} in \Cref{algo:ancestors}
is a classical breadth-first-search algorithm
that returns the ancestor set of a node set in a finite graph~$G$.
This algorithm can run in time~$O(|E|)$
with the appropriate data structure
(e.g., the graph is encoded as
a list of ancestor lists for each node)
because each edge is visited at most once
over all executions of \Cref{row:union}.

\begin{algorithm}[ht]
	\caption{Returns the ancestor set of a node set in a finite graph}
	\label{algo:ancestors}
	\begin{algorithmic}[1]
		\Procedure{Ancestors}{finite directed graph $G = (V, E)$, set~$I \subseteq V$}
		$\to$ set $A \subseteq V$
		\State $A \gets \emptyset$
		\Comment Ancestor set under construction
		\State $F \gets I$
		\Comment Set of ``frontier'' nodes: nodes that have been visited
		\State \Comment but whose neighbor list has not yet been read
		\While{$F \neq \emptyset$}
		\State $A \gets A \cup F$
		\State $N \gets \bigcup_{\ell \in F} \{k \in V {\setminus} A: (k, \ell) \in E\}$ 
		\Comment New frontier nodes
		\label{row:union}
		\State $F \gets N$
		\EndWhile
		\State \Return $A$
		\EndProcedure
	\end{algorithmic}
\end{algorithm}

\subsection{Markov chains and product-form relationship}
\label{sec:product-form-def}

As announced in \Cref{sec:intro},
our goal is to identify necessary and sufficient conditions
on a Markov chain's transition diagram~$G$
for which the associated stationary measures have a product-form relationship,
for all values of the transition rates.
Therefore, we start by defining a formal Markov chain,
where the transition rates are free variables rather than fixed values,
and we define the corresponding stationary distribution.
We then specify our definition of a product-form Markov chain.

\paragraph*{Formal Markov chain}

Our goal is to understand how the structure
of a Markov chain's transition diagram impacts
the relationship between its transition rates
and stationary measures.
This motivates the definition of a formal Markov chain.
\add{As we observe below, all our results apply directly to all the stationary measures of a formal Markov chain, irrespective of whether or not the instantiations of this Markov chain are positive recurrent.}

\begin{definition}[Formal Markov chain]
	Let $G = (V, E)$ be a (possibly infinite) strongly-connected directed graph.
	Define the corresponding \emph{formal Markov chain}
	to have transition rate from node~$i$ to node~$j$
	equal to $q_{i, j} > 0$ for each $(i, j) \in E$
	and $0$ for each $(i, j) \in (V \times V) {\setminus} E$.
	Note that $q_{i,j}$ is a free variable, not instantiated to a specific rate.
\end{definition}

For each strongly-connected graph~$G$, there is a single corresponding formal Markov chain, and vice versa. We will therefore refer to the two interchangeably. The quantities $q_{i, j}$ can be interpreted either as transition rates (CTMC) or as transition probabilities (DTMC; introducing the additional assumption that $\sum_{j \in V} q_{i, j} = 1$ for each $i \in V$).

For each formal Markov chain $G = (V, E)$,
we can define the associated (formal) stationary distribution $\pi$
to be the solution, as a function of the free variables $q_{i, j}$,
to the balance equation and normalization requirement:
\begin{align}
	\label{eq:balance-def}
	\pi_i \sum_{j \mid (i, j) \in E} q_{i, j}
	&= \sum_{k \mid (k,i) \in E} \pi_k q_{k, i},
	\quad i \in V, \\
	\label{eq:summation-def}
	\sum_{i \in V} \pi_i &= 1.
\end{align}
Because $q_{i,j}$ are free variables,
if $G$ is an infinite graph,
one cannot in general guarantee
that the summation requirement \eqref{eq:summation-def} is satisfied.
Thus for infinite graphs, we will instead consider a stationary measure and omit \eqref{eq:summation-def}, but we will still refer to stationary distributions for simplicity.
\remove{In fact, all our results apply directly to all the stationary measures of a formal Markov chain, irrespective of whether or not the instantiations of this Markov chain are positive recurrent.}

\paragraph*{Product-form relationship}

We now come to the central concept of the paper.
\Cref{def:graph-based-product-form}
gives our definition of
a product-form relationship between two nodes,
while \Cref{def:product-form-distribution}
considers the entire graph.

\begin{definition}[\replace{Graph-based product-form}{Graph-based product form}]
	\label{def:graph-based-product-form}
	Consider a formal Markov chain $G = (V, E)$
	and let $i, j \in V$.
	Nodes~$i$ and~$j$
	are in a \emph{graph-based product-form relationship}
	with one another 
	if, letting $\pi$ denote the Markov chain's stationary distribution,
	we have
	\begin{align} \label{eq:product-form}
		\pi_i f_{i, j} = \pi_j f_{j, i},
	\end{align}
	where $f_{i, j}$ and $f_{j, i}$ are polynomials (or more generally, rational functions)
	in the transition rates of the formal Markov chain.
	The complexity of a product-form relationship
	will be measured by the complexity of the associated polynomials
	(e.g., degree, number of monomials, arithmetic circuit complexity, arithmetic circuit depth, etc.).
	The term \emph{graph-based product form}
	will only be employed in association with a certain complexity level,
	as will be defined below.
\end{definition}

\begin{definition}[Product-form distribution]
	\label{def:product-form-distribution}
	A formal Markov chain has a \emph{product-form stationary distribution} at a given complexity level if all pairs of nodes in the graph are in a graph-based product-form relationship at that complexity level.
\end{definition}

For brevity, we will often say ``product-form'', leaving ``graph-based'' implicit. Note however that this notion of product-form is more specific than the general concept in the literature, as discussed in \Cref{sec:prior-work}.
\add{We initially focus in \cref{sec:s-product-form} on two simple types of product-form relationships, which we define here.
	We define more complex types of product-form relationships in \cref{sec:second-level}.}
\remove{We will be particularly interested in the following types of product-form relationships, in increasing order of complexity:}
\begin{description}
	\item[S-product-form:] S stands for \emph{sum}.
	Let $N_i := \{j \in V \mid (i, j) \in E\}$ denote the out-neighborhood of $i$.
	Focusing on $f_{i, j}$, we say that nodes~$i$ and~$j$ are
	in an S-product-form relationship
	if there exists some $S_{i,j} \subseteq N_i$, such that
	\begin{align*}
		f_{i, j} &= \sum_{k \in S_{i,j}} q_{i, k}.
	\end{align*}
	\item[PS-product-form:] P stands for \emph{product}.
	Again focusing on $f_{i, j}$,
	we say that nodes~$i$ and~$j$ are in a PS-product-form relationship
	if $f_{i, j}$ is the product of sums of subsets of transition rates emerging from nodes in the graph: for some $F_{i,j} \subseteq V$,
	there exists $S_{a, i, j} \subseteq N_a$ for each $a \in F_{i,j}$, such that
	\begin{align*}
		f_{i, j} &= \prod_{a \in F_{i,j}} \sum_{k \in S_{a, i,j}} q_{a, k}.
	\end{align*}
	\remove{\item[SPS-product-form:] We add another layer of alternation. Each sum is over neighboring vertices, while the products are over arbitrary vertices.
		We also allow the terms in the products to be the \emph{inverses} of sums, as well as direct sums. 
		Focusing on $f_{i, j}$,
		we say that nodes~$i$ and~$j$ are in an SPS-product-form relationship if
		there exist $S_{i,j} \subseteq N_i$, $F_{k,i,j} \subseteq (V \times \{-1, 1\})$ for each $k \in S_{i,j}$,
		and $S_{a,k,i,j} \subseteq N_a$ for each $(a, p) \in F_{k,i,j}$, such that
		\begin{align*}
			f_{i, j} = \sum_{k \in S_{i,j}} \prod_{(a, p) \in F_{k,i,j}} \Big(\sum_{k' \in S_{a,k,i,j}} q_{a,k'}\Big)^p.
		\end{align*}
		\item[Higher-order:] We can similarly define PSPS-product-form, SPSPS-product-form and more generally (PS)$^n$ and S(PS)$^{n}$ product form for any $n \in \{1, 2, 3, \ldots\}$.}
\end{description}

\add{In this paper, we give necessary and sufficient conditions on the graph structure of a Markov chain for product-form relationships to exist.
	Moreover, we explicitly construct the rational functions which appear in these product form relationships:
	For S-product-form relationships, we give necessary and sufficient conditions in \cref{theo:s-product-form}, and give the explicit rational functions in \eqref{eq:first-level-product-form}.
	For PS-product-form relationships, we give a sufficient condition in \cref{lem:cut-graph-ps}, and give the explicit rational functions in \eqref{eq:cut-graph-product-form}.
	\remove{For SPS-product-form relationships, we give a sufficient condition in \cref{lem:second-level-sps}, and give the explicit rational functions in \eqref{eq:sps-product-form}.}}

\paragraph*{\add{Arithmetic circuit complexity}}

\add{We make a short parenthesis to explain how these product-form classes are related to standard complexity measures. Let us first recall the definition of \emph{arithmetic circuit complexity classes} \cite{SY10}, a means of defining families of simple rational functions.}

\begin{definition}
	\label{def:arithmetic-circuit}
	\add{An arithmetic circuit is a directed acyclic graph in which each node with in-degree zero is an input variable and each node with positive in-degree is a basic arithmetic function (addition, multiplication, or reciprocal); there is a single node with out-degree zero called the output. Operations are performed in the direction indicated by edges, so that the output of a gate (a synonym of node in this context) is used as an input by its children.}
	\add{The \emph{depth} of an arithmetic circuit is the maximum number of operations in a path from the inputs to the output.}
\end{definition}

\replace{These}{The} product-form classes \add{defined above} correspond to \replace{limited-depth arithmetic circuit complexity classes}{classes of arithmetic circuits with a limited depth}: S-product-form is a subclass of depth-1 circuits, PS-product-form is a subclass of depth-2 circuits\remove{, and so forth}. The width of these arithmetic circuits, or more specifically the in-degree of the gates, is also limited by the degrees of the graph's nodes. Thus simple product-form classes correspond to simple arithmetic circuits.

\subsection{Cuts with a given source node set} \label{sec:cuts}

The following lemma is borrowed from \cite[Lemma~1.4]{kelly}.
It shows that cuts in a Markov chain's transition diagram
can be exploited to derive (from the balance equations)
a new set of equations, called \emph{cut equations},
that can sometimes be used to derive the stationary distribution more easily.

\begin{lemma} \label{lem:cut-equations}
	Consider a formal Markov chain $G = (V, E)$
	and let $\pi$ denote its stationary distribution.
	For each cut $(A, B)$ of the graph~$G$, we have
	\begin{align} \label{eq:cut}
		\sum_{(i, j) \in E \cap (A \times B)} \pi_i q_{i, j}
		&= \sum_{(j, i) \in E \cap (B \times A)} \pi_j q_{j, i}.
	\end{align}
\end{lemma}

\begin{proof}
	Equation~\eqref{eq:cut} follows by summing the balance equations~\eqref{eq:balance-def} over all $i \in A$ and making simplifications.
	More directly, \eqref{eq:cut} can be obtained by applying the strong law of large numbers for ergodic Markov chains. Indeed, the left-hand-side of~\eqref{eq:cut} is the long-run rate at which the chain makes a transition from~$A$ to~$B$, and the right-hand side is the long-run rate at which the chain makes a transition from~$B$ to~$A$.
\end{proof}

\begin{figure}[b!]
	\centering
	\begin{tikzpicture}
		\foreach \i in {0, 1, 2, 3, 4} {
			\node[draw, circle] (\i) at (2.5*\i, 0) {$\i$};
		}
		\node (5) at (2.5*5, 0) {$\cdots$};
		
		\foreach \i/\j in {0/1, 1/2, 2/3, 3/4, 4/5} {
			\draw (\i) edge[->, bend left] (\j);
			\draw (\j) edge[->, bend left] (\i);
			\draw[dashed] ($(\i)!.5!(\j)+(0,.6)$) -- ($(\i)!.5!(\j)-(0,.6)$);
			\node[anchor=south] at ($(\i)!.5!(\j)+(0,.6)$) {Cut $\i$};
		}
		
		\foreach \i/\j in {0/1, 1/2, 2/3, 3/4} {
			\node[anchor=north]
			at ($(\i)!.5!(\j)+(0, -0.6)$)
			{\scriptsize \add{$\pi_\i q_{\i, \j} = \pi_\j q_{\j, \i}$}};
		}
	\end{tikzpicture}
	\caption{A birth-and-death process.}
	\label{fig:bd}
\end{figure}
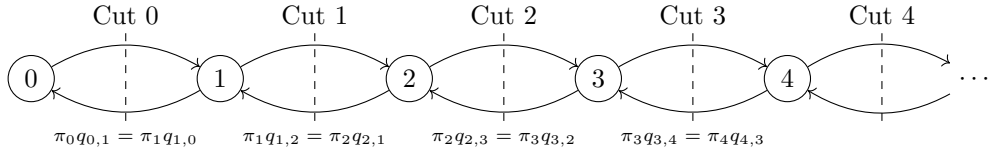

\begin{figure}[b!]
	\centering
	\begin{tikzpicture}
		\foreach \i in {0, 1, ..., 3} {
			\node[state] (\i) at (1.5*\i, 0) {$\i$};
		}
		
		\node[state] (bar1) at (1*1.5, -1.5) {$4$};
		\node[state] (bar2) at (2*1.5, -1.5) {$5$};
		\node[state] (bar3) at (3*1.5, -1.5) {$6$};
		
		\foreach \i/\j in {0/1, 1/2, 2/3} {
			\draw[->] (\j) -- (\i);
			\draw[->] (\i) -- (bar\j);
		}
		
		\foreach \i/\j in {1/2, 2/3} {
			\draw[->] (bar\i) -- (bar\j);
		}
		
		\foreach \i/\j in {0/1, 1/2, 2/3} {
			\draw[->] (bar\j) -- (\j);
		}
		
		\foreach \i/\j/\k in {0/1/0, 1/2/2, 2/3/4} {
			\draw[dashed]
			($(\i)!.5!(\j)+(0,.6)$)
			-- ($(\i|-bar\j)!.5!(bar\j)-(0,.2)$);
			\node[anchor=south]
			at ($(\i)!.5!(\j)+(0,.6)$) {\add{Cut $\k$}};
		}
		
		\foreach \i/\j/\k/\l in {0/1/2/1, 1/2/3/3} {
			\draw[dashed] ($(\i)!.6!(\j)+(0, .2)$)
			edge[bend right] ($(\j)!.5!(bar\j)$);
			\draw[dashed] ($(\j)!.5!(bar\j)$)
			edge[bend left] ($(bar\j)!.4!(bar\k)-(0, .6)$);
			\node[anchor=north]
			at ($(bar\j)!.5!(bar\k)-(0, .6)$) {\add{Cut $\l$}};
		}
		
		\draw[dashed] ($(2)!.6!(3)+(0, .2)$)
		edge[bend right] ($(3)!.5!(bar3)$);
		\draw[dashed] ($(3)!.5!(bar3)$)
		edge[bend right] ($(3)!.5!(bar3)+(.2, 0)$);
		\node[anchor=west] at ($(3)!.5!(bar3)+(.2, 0)$) {\add{Cut~$5$}};
		
		\node[anchor=west, align=left]
		at ($(3)!.5!(bar3)+(1.8, 0)$) {
			\add{Cut~0 yields
				$\pi_1 q_{1, 0} = \pi_0 q_{0, 4}$} \\
			\add{Cut~1 yields
				$\pi_4 (q_{4, 1} + q_{4, 5})
				= \pi_1 q_{1, 0}$} \\
			\add{Cut~2 yields
				$\pi_2 q_{2, 1} = \pi_1 q_{1, 5} + \pi_4 q_{4, 5}$} \\
			\add{Cut~3 yields
				$\pi_5 (q_{5, 2} + q_{5, 6}) = \pi_2 q_{2, 1}$} \\
			\add{Cut~4 yields
				$\pi_3 q_{3, 2} = \pi_2 q_{2, 6} + \pi_5 q_{5, 6}$} \\
			\add{Cut~5 yields
				$\pi_6 q_{6, 3} = \pi_3 q_{3, 2}$}
		};
		
	\end{tikzpicture}
	\caption{A simple formal Markov chain exhibiting a product form relationship between \replace{nodes $1$ and $4$, as a result of the cut $(\{0,4\}, \{1,2,3,5,6\})$.}{nodes~$0$ and $1$ (via cut~0), $1$ and $4$ (via cut~1), $2$ and $5$ (via cut~3), and 3 and 6 (via cut~5). In addition, cut~2 allows us to express $\pi_2$ as a function of $\pi_0$ and $\pi_1$, and cut~4 allows us to express $\pi_3$ as a function of $\pi_2$ and $\pi_5$. All in all, these six cut equations allow us to express $\pi_i$ as a function of $\pi_0$ for each $i \in \{1, 2, \ldots, 6\}$.}}
	\label{fig:jaf-cuts}
\end{figure}
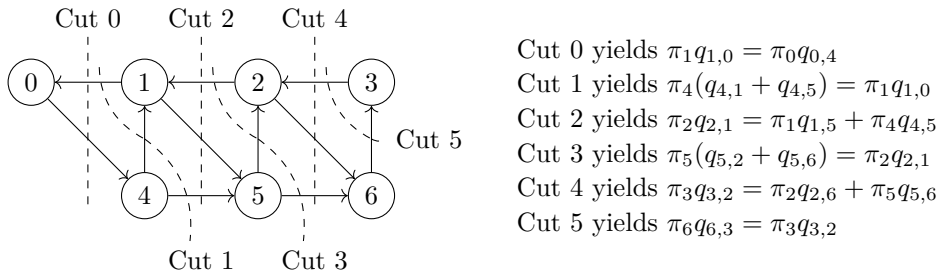

The rest of the paper focuses on identifying
necessary and sufficient conditions on the graph structure
that guarantee the existence of ``nice cuts''
that yield a simple expression
for the Markov chain's stationary distribution via~\eqref{eq:cut}.
A simple and famous example is a birth-and-death process,
as shown in \Cref{fig:bd}:
for each $i \in \{0, 1, 2, \ldots, \}$,
cut~$i$ relates $\pi_i$ and $\pi_{i+1}$ via the cut equation
$\pi_i q_{i, i+1} = \pi_{i+1} q_{i+1, i}$,
so that nodes~$i$ and~$i+1$ are on an S-product-form relationship.
\add{We obtain a closed-form expression for the stationary distribution up to a positive multiplicative constant:
	\begin{align*}
		\pi_i = \frac{q_{i-1, i}}{q_{i, i-1}} \frac{q_{i-2, i-1}}{q_{i-1, i-2}}
		\cdots \frac{q_{0, 1}}{q_{1, 0}} \pi_0,
		\quad \text{for each $i \in \{0, 1, 2, \ldots\}$}.
\end{align*}}%
Although the\remove{se} cut equations follow from the balance equations
$\pi_i (q_{i, i-1} \indicator{i \ge 1} + q_{i, i+1})
= \pi_{i-1} q_{i-1, i} \indicator{i \ge 1} + \pi_{i+1} q_{i+1, i}$ for $i \in \{0, 1, 2, \ldots\}$,
they allow us to \replace{explicitate}{derive} the stationary distribution more directly.
As a slightly more intricate toy example,
in \Cref{fig:jaf-cuts},
the cut $(\{0, 4\}, \{1, 2, 3, 5, 6\})$
implies $\pi_4 (q_{4, 1} + q_{4, 5}) = \pi_1 q_{1, 0}$,
hence nodes~$1$ and~$4$ are on an S-product-form relationship.

In general, a cut equation~\eqref{eq:cut} as given in \Cref{lem:cut-equations}
is more convenient \replace{that}{than} the balance equations~\eqref{eq:balance-def}
if the set of nodes~$i$ such that $\pi_i$ appears
on either side of the equation is small.
This set is called the \emph{source} of the corresponding cut,
as it consists of the nodes that are the sources of the edges that cross the cut.

\begin{definition}[Source] \label{def:source}
	Consider a formal Markov chain $G = (V, E)$.
	The source of a cut $(A, B)$ of the graph~$G$
	is the pair $(I, J)$ defined by
	\begin{align} \label{eq:sourced-cut}
		I = \{i \in A: E \cap (\{i\} \times B) \neq \emptyset\},
		\qquad \qquad
		J = \{j \in B: E \cap (\{j\} \times A) \neq \emptyset\}.
	\end{align}
	Equivalently, $(A, B)$ is called an \replace{$(I, J)$}{$I, J$}-sourced cut.
\end{definition}

\add{Considering again the examples above,
	cut~$i \in \{0, 1, 2, \ldots\}$ in \Cref{fig:bd} has source $(i, i+1)$,
	cut~1 in \Cref{fig:jaf-cuts} has source $(1, 4)$,
	and cut~2 in \Cref{fig:jaf-cuts} has source $(\{1, 4\}, 2)$.}
In \Cref{sec:s-product-form},
we will focus on the special case
where the source sets $I$ and $J$ are both singletons, $I = \{i\}$ and $J = \{j\}$.
We refer to such a cut as a \emph{single-sourced} cut.
\add{All cuts in \Cref{fig:bd} are single-sourced cuts,
	and so are cuts 0, 1, 3, and 5 in \Cref{fig:jaf-cuts}.}

\subsection{Joint-ancestor freeness} \label{sec:jaf}

In \Cref{def:jaf} below,
we identify a simpler condition,
called \emph{joint-ancestor freeness},
that we will prove to be necessary and sufficient for
the existence of a cut with a particular source pair
in \Cref{sec:s-product-form-2}.

\begin{definition}[Joint-ancestor freeness] \label{def:jaf}
	Consider a formal Markov chain $G = (V, E)$
	and two disjoint nonempty sets $I, J \subsetneq V$.
	A node~$k \in V$ is a \emph{joint ancestor} of node sets~$I$ and~$J$
	if $k \in A_I(G {\setminus} J) \cap A_J(G {\setminus} I)$,
	i.e., there is both a path from node~$k$ to some node in~$I$ that avoids set~$J$
	and a path from node~$k$ to some node in~$J$ that avoids set~$I$.
	Node sets~$I$ and~$J$ are said to be \emph{joint-ancestor free}
	if $A_I(G {\setminus} J) \cap A_J(G {\setminus} I) = \emptyset$.
	In the special case where $I$ and/or $J$ are singletons,
	we drop the curly brackets in the notation,
	e.g., we write $A_i(G {\setminus} J)$ for $A_{\{i\}}(G {\setminus} J)$.
\end{definition}

To make this definition more concrete,
lets us again consider the birth-and-death process
of \Cref{fig:bd}.
Focusing on $I = \{2\}$ and $J = \{3\}$,
we have $A_2(G {\setminus} 3) = \{0, 1, 2\}$
and $A_3(G {\setminus} 2) = \{3, 4, 5, \ldots\}$,
so that nodes~2 and~3 are joint-ancestor free.
To see why $A_2(G {\setminus} 3) = \{0, 1, 2\}$,
it suffices to observe that
the subgraph $G {\setminus} 3$ consists of two strongly connected components:
$\{0, 1, 2\}$ and $\{4, 5, 6, \ldots\}$.
\replace{Anticipating over}{Anticipating} \Cref{prop:jaf-cuts-1} below,
we observe that 
$(A_2(G {\setminus} 3), A_3(G {\setminus} 2))
= (\{0, 1, 2\}, \{3, 4, 5, \ldots\})$
is exactly Cut~2 in \Cref{fig:bd}.
Similarly, $I = \{1\}$ and $J = \{2, 4\}$
are joint-ancestor free because
$A_1(G {\setminus} \{2, 4\}) = \{0, 1\}$
and $A_{\{2, 4\}}(G {\setminus} 1) = \{2, 3, 4, \ldots\}$.
On the contrary, nodes~1 and~3 are not joint-ancestor free
because $A_1(G {\setminus} 3) = \{0, 1, 2\}$
and $A_3(G {\setminus} 1) = \{2, 3, 4, \ldots\}$
have non-empty intersection $\{2\}$.

The \textsc{MutuallyAvoidingAncestors} procedure
in \Cref{algo:jaf}
returns the joint-ancestor sets $A_I(G {\setminus} J)$ and $A_J(G {\setminus} I)$
in time $O(|E|)$ in a finite graph $G$,
by calling the \textsc{Ancestors} procedure
from \Cref{algo:ancestors}.
\textsc{MutuallyAvoidingAncestors} can be used to test
if two node sets are joint-ancestor free.

\begin{algorithm}[ht]
	\caption{Returns the mutually-avoiding ancestors of two node sets}
	\label{algo:jaf}
	\begin{algorithmic}[1]
		\Procedure{MutuallyAvoidingAncestors}{finite directed graph $G = (V, E)$, disjoint non\-empty sets~$I, J \subseteq V$}
		$\to$ ancestor sets $A_I(G {\setminus} J), A_J(G {\setminus} I)$
		\State $A_I \gets$ \Call{Ancestors}{$I$, $G {\setminus} J$}
		\State $A_J \gets$ \Call{Ancestors}{$J$, $G {\setminus} I$}
		\State \Return $A_I, A_J$
		\EndProcedure
	\end{algorithmic}
\end{algorithm}

\section{S-product-form, cuts, and joint-ancestor freeness}
\label{sec:s-product-form}

In this section, we focus on the S-product-form relationship
introduced in \Cref{sec:product-form-def}.
\Cref{theo:s-product-form}, the main result of this section,
is stated in \Cref{sec:s-product-form-main}
and illustrated on toy examples in \Cref{sec:s-product-form-examples}.
The proof of \Cref{theo:s-product-form}
relies on intermediary results shown in
\Cref{sec:s-product-form-1,sec:s-product-form-2,sec:s-product-form-3}.
Higher-order product-form relationships,
such as PS-product-form,
will be considered in \Cref{sec:higher-levels}.

\subsection{Main theorem} \label{sec:s-product-form-main}

\Cref{theo:s-product-form} below is our first main contribution:
it gives simple necessary and sufficient conditions
under which two nodes are in an S-product-form relationship.
This result relies on the two graph-based notions introduced earlier,
namely, cuts with a given source node (\Cref{sec:cuts})
and joint-ancestor freeness (\Cref{sec:jaf}).
The rest of \Cref{sec:s-product-form} will give further insights into this result.

\begin{theorem} \label{theo:s-product-form}
	Consider a formal Markov chain~$G = (V, E)$
	and let~$i, j \in V$.
	The following statements are equivalent:
	\begin{enumerate}[(i)]
		\item \label{item:S-product-form-1}
		Nodes~$i$ and~$j$ are in an S-product-form relationship.
		\item \label{item:S-product-form-2}
		There is an $i, j$-sourced cut.
		\item \label{item:S-product-form-3}
		Nodes~$i$ and $j$ are joint-ancestor free.
	\end{enumerate}
	If these statements are true, then
	the S-product-form between nodes~$i$ and~$j$ has factors
	\begin{align}
		\label{eq:first-level-product-form}
		f_{i, j} = \sum_{\substack{k \in A_j(G {\setminus} i): \\ (i, k) \in E}} q_{i, k}
		\quad \text{and} \quad
		f_{j, i} = \sum_{\substack{k \in A_i(G {\setminus} j): \\ (j, k) \in E}} q_{j, k}.
	\end{align}
\end{theorem}

\begin{proof}
	The implication \eqref{item:S-product-form-2} $\implies$ \eqref{item:S-product-form-1}
	is a classical result that will be recalled
	in \Cref{lem:cut-implies-product-form} in \Cref{sec:s-product-form-1}.
	The equivalence \eqref{item:S-product-form-2} $\iff$ \eqref{item:S-product-form-3}
	will be shown in \Cref{prop:jaf-cuts-1} in \Cref{sec:s-product-form-2}.
	The implication \eqref{item:S-product-form-1} $\implies$ \eqref{item:S-product-form-3}
	will be shown in \Cref{theo:no-product-form} in \Cref{sec:s-product-form-3}.
	Lastly, Equation~\eqref{eq:first-level-product-form} follows for instance
	by combining \Cref{lem:cut-implies-product-form} and \Cref{prop:jaf-cuts-1}.
\end{proof}

The equivalence between conditions
\eqref{item:S-product-form-1} and \eqref{item:S-product-form-2}
is reminiscent of classical sufficient conditions
on the existence of a product-form relationship,
except that the focus is now on the transition graph
rather than on the transition rates.
Now, the equivalence between conditions
\eqref{item:S-product-form-2} and \eqref{item:S-product-form-3}
can be intuitively understood as follows.
If two nodes $i, j \in V$ are joint-ancestor free,
meaning that  $A_i(G {\setminus} j) \cap A_j(G {\setminus} i) = \emptyset$,
then one can verify that $(A_i(G {\setminus} j), A_j(G {\setminus} j))$
forms a cut and that its source nodes are $i$ and $j$.
On the contrary, if $i$ and $j$ are not joint-ancestor free,
there exists $k \in V {\setminus} \{i, j\}$
such that there are two paths $P(k \to i {\setminus} j)$
and $P(k \to j {\setminus} i)$.
The existence of these two paths
precludes any cut $(A, B)$
from having source \replace{$i, j$}{$(i, j)$}.
Indeed, assuming for example that $i, k \in A$ and $j \in B$,
the path $P(k \to j {\setminus} i)$
needs to go from part~$A$ (containing~$k$)
to part~$B$ (containing~$j$),
and it cannot do so via node~$i$
because this is an $i$-avoiding path.
Therefore, the source of part~$A$ cannot be reduced to node~$i$.

Thanks to \Cref{theo:s-product-form},
we can directly apply procedure \textsc{MutuallyAvoidingAncestors} from \Cref{algo:jaf}
to verify if two nodes~$i$ and~$j$ are in an S-product-form relationship
and, if yes, compute the corresponding factors,
all with time complexity $O(|E|)$.
This is far more efficient than directly testing each cut in the graph to see if its sources are $i$ and $j$:
there are $2^{|V|}$ such cuts, each of which would take $|E|$ time to check.
Testing the S-product-form relationship of all pairs of nodes in the graph
can be done in time $O(|V|^2 |E|)$
(also see \Cref{sec:cut-graph}).

Note that Theorem 2 implies that if nodes $i$ and $j$ are in an S-product-form relationship,
then the sets $S_{i, j}$ and $S_{j, i}$ as defined in
\replace{\cref{def:product-form-distribution}}{\Cref{sec:product-form-def}}
must be disjoint,
as these sets are subsets of the disjoint ancestry sets $A_j(G {\setminus} i)$ and $A_i(G {\setminus} j)$,
as specified in \eqref{eq:first-level-product-form}.

\subsection{Illustrative examples} \label{sec:s-product-form-examples}

Before we prove the intermediary results
that appear in the proof of \Cref{theo:s-product-form},
let us illustrate this connection between single-sourced cuts, joint-ancestor freeness, and product-form relationships\add{.}\remove{with a few toy examples, shown in Figures~1 and~3. We first revisit the birth-and-death process already discussed in Sections~2.3 and~2.4. For a more sophisticated example involving S-product-form, please see Section~5.1.}

\paragraph*{\add{Toy examples}}

\add{We first revisit the birth-and-death process of \Cref{fig:bd} already discussed in \Cref{sec:cuts,sec:jaf}, and then we consider the toy examples shown in \Cref{fig:illustrative-examples}.}

\begin{example}[Birth-and-death process] \label{ex:birth-and-death}
	Consider a formal Markov chain $G = (V, E)$ with
	$V = \{0, 1, 2, \ldots\}$ and $E = \bigcup_{i \in V} \{(i, i+1), (i+1, i)\}$,
	as in \Cref{fig:bd}.
	For each $i \in V$,
	nodes~$i$ and~$i+1$ are in an S-product-form relationship
	through the $i, i+1$-sourced cut formed by
	$A_i(G {\setminus} i+1) = \{1, 2, \ldots, i\}$
	and $A_{i+1}(G {\setminus} i) = \{i+1, i+2, \ldots\}$.
	However, for each $i, j \in V$
	such that $i \le j - 2$,
	nodes~$i$ and~$j$ are not in an S-product-form relationship
	because $A_i(G {\setminus} j) = \{0, 1, 2, \ldots, j-1\}$
	and $A_j(G {\setminus} i) = \{i+1, i+2, \ldots, n\}$
	intersect at
	$A_i(G {\setminus} j) \cap A_j(G {\setminus} i)
	= \{i+1, i+2, \ldots, j-1\}$.
	\add{Observe that, unlike the next two examples, this process is (time-)reversible when it is positive-recurrent. Reversibility will be discussed again in \Cref{ex:tree}.}
\end{example}

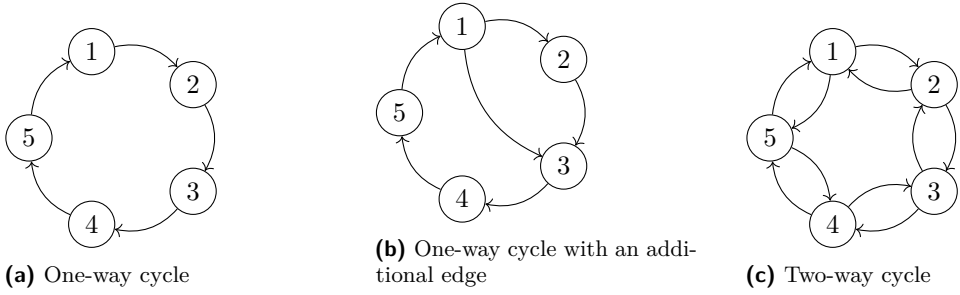
\begin{figure}[ht]
	\begin{subfigure}{.3\linewidth}
		\begin{tikzpicture}
			\foreach \a in {1, 2, ..., 5} {
				\node[state] (\a) at (180-\a*360/5: 1.2cm) {\a};
			}
			\foreach \a/\b in {1/2, 2/3, 3/4, 4/5, 5/1} {
				\draw (\a) edge[->, bend left] (\b);
			}
		\end{tikzpicture}
		\caption{One-way cycle}
		\label{fig:one-way-cycle}
	\end{subfigure}
	\hfill
	\begin{subfigure}{.3\linewidth}
		\begin{tikzpicture}
			\foreach \a in {1, 2, ..., 5} {
				\node[state] (\a) at (180-\a*360/5: 1.2cm) {\a};
			}
			\foreach \a/\b in {1/2, 2/3, 3/4, 4/5, 5/1} {
				\draw (\a) edge[->, bend left] (\b);
			}
			\draw (1) edge[->, bend right] (3);
		\end{tikzpicture}
		\caption{One-way cycle with an additional edge}
		\label{fig:one-way-cycle-with-edge}
	\end{subfigure}
	\hfill
	\begin{subfigure}{.3\linewidth}
		\begin{tikzpicture}
			\foreach \a in {1, 2, ..., 5} {
				\node[state] (\a) at (180-\a*360/5: 1.2cm) {\a};
			}
			\foreach \a/\b in {1/2, 2/3, 3/4, 4/5, 5/1} {
				\draw (\a) edge[->, bend left] (\b);
				\draw (\b) edge[->, bend left] (\a);
			}
		\end{tikzpicture}
		\caption{Two-way cycle}
		\label{fig:two-way-cycle}
	\end{subfigure}
	\caption{Illustrative examples of S-product-form relationships.}
	\label{fig:illustrative-examples}
\end{figure}

\begin{example}[One-way cycle] \label{ex:one-way-cycle}
	Consider a formal Markov chain $G = (V, E)$ with
	$V = \{1, 2, \ldots, n\}$ for some $n \ge 3$
	and $E = \{(i, i+1) | i \in V\} \cup E'$,
	where $E' \subseteq \{(i, i) | i \in V\}$,
	with the convention that nodes are numbered modulo~$n$.
	For instance, \Cref{fig:one-way-cycle}.
	For each $i, j \in V$, say with $i < j$,
	the sets
	$A_i(G {\setminus} j)= \{j+1, j+2, \ldots, n, 1, 2, \ldots, i\}$
	and $A_j(G {\setminus} i) = \{i+1, i+2, \ldots, j\}$
	are disjoint and therefore form an $i, j$-sourced cut.
	\Cref{theo:s-product-form} implies that
	nodes~$i$ and~$j$ are in an S-product-form relationship with
	$f_{i, j} = q_{i, i+1}$ and $f_{j, i} = q_{j, j+1}$.
	Equivalently, we can check visually
	that there is no node from which there is a directed path
	to node~$i$ without visiting node~$j$ and vice versa.
	This example shows in particular that an $i, j$-sourced cut may exist
	even if there is neither an edge~$(i, j)$ nor an edge~$(j, i)$.
	\add{\Cref{theo:clique} in Appendix~\ref{app:clique} shows that the one-way cycle is the \emph{only} finite formal Markov chain in which all pairs of nodes are on an S-product-form relationship.}
\end{example}

\begin{example}[One-way cycle with an additional edge] \label{ex:one_way_cycle}
	Consider the same directed cycle,
	but with an additional edge
	from node~$1$ to some node~$k \in \{3, 4, \ldots, n\}$.
	For instance, \Cref{fig:one-way-cycle-with-edge}.
	For each $i \in \{2, 3, \ldots, k - 1\}$
	and $j \in \{k, k + 2, \ldots, n\}$,
	nodes~$i$ and~$j$ are no longer joint-ancestor free because
	$1 \in A_i(G {\setminus} j) \cap A_j(G {\setminus} i)$.
	Indeed,
	$1 \to 2 \to 3 \to \cdots \to i$
	is a $P(1 \to i {\setminus} j)$ path
	and
	$1 \to k \to k + 1 \to \dots \to j$
	is a $P(1 \to j {\setminus} i)$ path.
	Therefore, nodes~$i$ and~$j$ are no longer
	in an S-product-form relationship.
\end{example}

\begin{example}[Two-way cycle]
	Lastly, consider a directed graph $G = (V, E)$ with
	$V = \{1, 2, \ldots, n\}$ for some $n \ge 3$
	and $E = \bigcup_{i \in V} \{(i, i+1), (i+1, i)\}$,
	again with the convention that nodes are numbered modulo~$n$.
	For instance, \Cref{fig:two-way-cycle}.
	For each $i, j \in V$, the sets
	$A_i(G {\setminus} j) = V {\setminus} \{j\}$
	and $A_j(G {\setminus} i) = V {\setminus} \{i\}$
	have nonempty intersection
	$A_i(G {\setminus} j) \cap A_j(G {\setminus} i) = V {\setminus} \{i, j\}$.
	Hence, there are no S-product-form relationships in this graph.
\end{example}

\paragraph*{\add{Relations with other product-form or structural conditions}}

\add{We now consider two larger examples that illustrate \replace{graph-based product-form}{graph-based product form} and its relationship to other structural properties of Markov chains. The reader is also invited to take a look at \Cref{sec:examples} for examples of Markov chains that model actual queueing systems and exhibit S-product-form relationships.}

\begin{example}[\add{Trees and reversibility}] \label{ex:tree}
	\add{The following extension of birth-and-death processes again has an S-product-form stationary distribution and is reversible when it is positive-recurrent;
		the result is from \cite[Lemma~1.5]{kelly} and \cite[Theorem~2.2]{serfozo}.
		Consider a formal Markov chain $G = (V, E)$
		and let $H$ denote its communication graph,
		that is, the undirected graph $H = (V, F)$ with
		$F = \{\{i, j\} \subseteq V: (i, j) \in E \text{ or } (j, i) \in E\}$.
		Assume that (i) $H$ is a tree, and
		(ii) for each $i, j \in V$, we have $(i, j) \in E$ if and only if $(j, i) \in E$.
		Then, for each $i, j \in V$
		such that $\{i, j\} \in H$,
		nodes~$i$ and~$j$ are on a particularly simple form of S-product-form relationship
		given by $\pi_i q_{i, j} = \pi_j q_{j, i}$.
		Indeed, Assumptions~(i) and~(ii) imply that
		removing edges $(i, j)$ and $(j, i)$ from~$G$
		divides~$G$ into two strongly connected subgraphs.
		Letting $A_i$ (resp.\ $A_j$) denote the set of nodes
		in the part containing node~$i$ (resp.\ $j$),
		we conclude not only that $(A_i, A_j)$ is an $i, j$-sourced cut,
		but also that the only edges across this cut are $(i, j)$ and $(j, i)$.
		\Cref{lem:cut-equations} then yields the product-form relationship $\pi_i q_{i, j} = \pi_j q_{j, i}$.}
	
	\add{Conversely, one can verify that the only formal Markov chains
		in which the S-product-form relationship
		implies reversibility (when the chain is positive recurrent)
		for all transition rates
		are those satisfying Assumptions~(i) and~(ii).
		Other reversible and quasi-reversible queueing models do not have an S-product-form stationary distribution (nor any sort of \replace{graph-based product-form}{graph-based product form}).
		Focusing on Whittle networks~\cite[Chapter~1]{serfozo}, one can verify that, except for a single-queue network (that forms a birth-and-death process), the product-form of the distribution holds only because the departure rates satisfy the so-called \emph{balance property}. Similarly, an order-independent queue~\cite{K11} with a single class forms a birth-and-death process, but an order-independent queue with at least two classes only has a product-form stationary distribution because the departure rates satisfy the so-called \emph{order-independence condition}.}
\end{example}

\begin{example}[\add{Quasi birth-and-death process}] \label{ex:qbd}
	\add{Another possible extension of the birth-and-death process of \Cref{ex:birth-and-death}
		is a Markov chain whose structure is a particular case of a
		quasi-birth-and-death (QBD) process~\cite{H93,L99,P11}.
		Consider a formal Markov chain $G = (V, E)$
		and assume $V$ to be finite for simplicity.
		Assume that $V$ can be partitioned into subsets $V_0, V_1, V_2, \ldots, V_m$,
		called superstates,
		such that transitions across superstates occur only
		between neighboring superstates,
		and only via designated states.
		More formally, assume that,
		for each $j \in \{0, 1, \ldots, m-1\}$,
		there exist $i_{j, \text{up}} \in V_j$ and $i_{j+1, \text{down}} \in V_{j+1}$
		such that
		\begin{align*}
			E \cap \left(V_j \times \bigcup_{k = j+1}^m V_k \right)
			&= \{i_{j, \text{up}}\} \times V_{j+1}, &
			E \cap \left(V_{j+1} \times \bigcup_{k = 0}^j V_k \right)
			&= \{i_{j+1, \text{down}}\} \times V_j.
		\end{align*}
		For instance, the toy example of \Cref{fig:jaf-cuts}
		fits this structure with 
		$V_0 = \{0, 4\}$, $V_1 = \{1, 5\}$, $V_2 = \{2, 6\}$, and $V_3 = \{3\}$,
		$i_{j, \text{up}} = j + 4$ for each $j \in \{0, 1, 2\}$, and
		$i_{j, \text{down}} = j$ for each $j \in \{1, 2, 3\}$.
		In general, for each $j \in \{0, 1, \ldots, m-1\}$,
		states $i_{j, \text{up}}$ and $i_{j+1, \text{down}}$ are on an S-product-form relationship,
		as the previous equation implies that
		$(\bigcup_{k = 0}^j V_j, \bigcup_{k = j+1}^m V_j)$
		is an $i_{j, \text{up}}, i_{j+1, \text{down}}$-sourced cut,
		so that
		\begin{align*}
			\pi_{i_{j, \text{up}}} \sum_{k \in V_{j+1}} q_{i_{j, \text{up}}, k}
			&= \pi_{i_{j+1, \text{down}}} \sum_{k \in V_j} q_{i_{j+1, \text{down}}, k}.
		\end{align*}
		While in general deriving the stationary probabilities $\pi_i$ for each $i \in V$
		requires solving several systems of linear equations (one for each superstate),
		we will see in
		\Cref{sec:msj-example,sec:batch1}
		that in some cases, graph-based product form can also be applied
		to derive the complete stationary distribution,
		without resorting to linear algebra.
	}
\end{example}

\subsection{The existence of a single-sourced cut implies S-product-form} \label{sec:s-product-form-1}

\Cref{lem:cut-implies-product-form} below shows the implication
\eqref{item:S-product-form-2} $\implies$ \eqref{item:S-product-form-1}
from \Cref{theo:s-product-form}.
This result is recalled for completeness,
but it follows directly by combining
the definition of a sourced cut (\Cref{def:source})
with \Cref{lem:cut-equations}.

\begin{lemma}
	\label{lem:cut-implies-product-form}
	Consider a formal Markov chain~$G = (V, E)$
	and two nodes~$i, j \in V$.
	If there is an $i, j$-sourced cut in the graph~$G$,
	then these nodes are in an S-product-form relationship~\eqref{eq:product-form},
	with factors as given in~\eqref{eq:first-level-product-form}.
\end{lemma}

\subsection{The existence of a single-sourced cut is equivalent to joint-ancestor freeness} \label{sec:s-product-form-2}

We now prove
the equivalence
\eqref{item:S-product-form-2} $\iff$ \eqref{item:S-product-form-3}
from \Cref{theo:s-product-form},
that is, joint-ancestor freeness
is necessary and sufficient
for the existence of a cut with a particular source pair.
We show that this equivalence holds
both when the source pair is a pair \replace{$i, j$}{$(i, j)$} of nodes
(\Cref{prop:jaf-cuts-1})
and more generally for any pair \replace{$I, J$}{$(I, J)$} of source sets
(\Cref{prop:jaf-cuts-2}).
Cuts where the source pair is a general pair of source sets \replace{$I, J$}{$(I, J)$} give rise to
higher-level cuts, as we show and discuss in \Cref{sec:higher-levels}.

\begin{proposition} \label{prop:jaf-cuts-1}
	Consider a formal Markov chain $G = (V, E)$
	and let $i, j \in V$.
	Nodes~$i$ and~$j$ are joint-ancestor free
	if and only if
	there exists an $i, j$-sourced cut.
	In this case, the only $i, j$-sourced cut is
	$(A_i(G {\setminus} j), A_j(G {\setminus} i))$.
\end{proposition}

\Cref{prop:jaf-cuts-1} is a special case of \Cref{prop:jaf-cuts-2},
which is stated and proved later in this section.
To illustrate the intuition behind \Cref{prop:jaf-cuts-1},
again consider the toy example of \Cref{fig:jaf-cuts}.
Nodes $1$ and $4$ are joint-ancestor free
because $A_1(G {\setminus} 4) = \{1, 2, 3, 5, 6\}$
and $A_4(G {\setminus} 1) = \{0, 4\}$ are disjoint
and therefore form a $1, 4$-sourced cut.
In contrast, nodes~$1$ and~$2$ are not joint-ancestor free
because $A_1(G {\setminus} 2) = \{0, 1, 4\}$
and $A_2(G {\setminus} 1) = \{2, 3, 4, 5, 6\}$
intersect at node~$4$.
Any cut $(A, B)$ such that $1 \in A$ and $2 \in B$
has to be crossed by an edge from a path $P(4 \to 2 {\setminus} 1)$
(if $4 \in A$)
or from a path $P(4 \to 1 {\setminus} 2)$
(if $4 \in B$),
which makes it impossible to build a cut
whose sources are restricted to nodes~$1$ and~$2$.
This intuition is formalized
in Statement~\ref{prop:ancestors-2}
of \Cref{lem:jaf-cuts} below.
\Cref{prop:jaf-cuts-1} implies in particular that
an $i, j$-sourced cut is unique when it exists,
hence we can say \emph{the} $i, j$-sourced cut.

For \Cref{prop:jaf-cuts-2} below, the situation is slightly more complicated.
When considering
joint-ancestor free sets $I, J$ containing more than one node,
cuts may not be unique, and may not have the entire sets $I, J$ as sources.
Nonetheless, there is still a bidirectional relationship between mutually-avoiding ancestor sets
and cut-source sets.

\begin{proposition} \label{prop:jaf-cuts-2}
	Consider a formal Markov chain $G = (V, E)$
	and two disjoint non\-empty sets $I, J \subseteq V$.
	We have the following:
	\begin{enumerate}[(i)]
		\item \label{prop:ancestors-4}
		If $I$ and $J$ are joint-ancestor free,
		then the cut $(A_I(G {\setminus} J), A_J(G {\setminus} I))$
		is an \replace{$(\underline{I}, \underline{J})$}{$\underline{I}, \underline{J}$}-sourced cut, for some non-empty sets $\underline{I} \subseteq I$ and $\underline{J} \subseteq J$.
		\item \label{prop:ancestors-5}
		If $(A_I(G {\setminus} J), A_J(G {\setminus} I))$ is a cut and has source\remove{s} $(I, J)$, then it is the unique cut with sources ($I$, $J$).
		\item \label{prop:ancestors-3}
		If $I$ and $J$ are not joint-ancestor free,
		then there is no \replace{$(I, J)$}{$I, J$}-sourced cut.
	\end{enumerate}
\end{proposition}

Again considering the example of \Cref{fig:jaf-cuts},
let $I = \{1, 4\}$ and $J = \{2, 5\}$.
The ancestor sets
$A_I(G {\setminus} J) = \{0, 1, 4\}$
and $A_J(G {\setminus} I) = \{2, 3, 5, 6\}$
are disjoint,
and we can verify that they form an
\replace{$(\underline{I}, \underline{J})$ sourced-cut}{$\underline{I}, \underline{J}$-sourced cut}
with $\underline{I} = \{1, 4\} = I$
and $\underline{J} = \{2\} \subsetneq J$.
For a negative example, node sets
$I' = \{1\}$ and $J = \{2, 5\}$ are not joint-ancestor free
because $A_{I'}(G {\setminus} J) = \{0, 1, 4\}$
and $A_{J}(G {\setminus} I') = \{2, 3, 4, 5, 6\}$
intersect at node~$4$.
Correspondingly, one can verify that there is no \replace{$(I', J)$}{$I', J$}-sourced cut in the graph.

Before proving \Cref{prop:jaf-cuts-1,prop:jaf-cuts-2},
we prove the following intermediary lemma,
which will also be instrumental for later results.

\begin{lemma} \label{lem:jaf-cuts}
	Consider a formal Markov chain $G = (V, E)$
	and two disjoint \replace{nonsets}{non-empty sets} $I, J \subsetneq V$.
	We have the following:
	\begin{enumerate}[(i)]
		\item \label{prop:ancestors-1}
		$A_I(G {\setminus} J) \cup A_J(G {\setminus} I) = V$.
		\item \label{prop:ancestors-2}
		If $(A, B)$ is an \replace{$(I, J)$}{$I, J$}-sourced cut,
		then $A_I(G {\setminus} J) \subseteq A$
		and $A_J(G {\setminus} I) \subseteq B$.
	\end{enumerate}
\end{lemma}
\begin{proof}[Proof of \Cref{lem:jaf-cuts}]
	Let us first prove \Cref{lem:jaf-cuts}\ref{prop:ancestors-1}.
	Let $k \in V$.
	Since $G$ is strongly connected
	and $I$ is nonempty,
	there is a directed path
	$v_1, v_2, \ldots, v_n$ in $G$,
	with $n \ge 1$,
	such that $v_1 = k$ and $v_n \in I$.
	Let $p = \min\{q \in \{1, 2, \ldots, n\} | k_q \in I \cup J\}$.
	Then $k \in A_I(G {\setminus} J)$ if $k_p \in I$
	and $k \in A_J(G {\setminus} I)$ if $k_p \in J$.
	Hence, $k \in A_I(G {\setminus} J) \cup A_J(G {\setminus} I)$
	for each $k \in V$, which implies that $V = A_I(G {\setminus} J) \cup A_J(G {\setminus} I)$.
	
	We now prove \Cref{lem:jaf-cuts}\ref{prop:ancestors-2}.
	Assume that $(A, B)$ is an \replace{$(I, J)$}{$I, J$}-sourced cut
	and let $k \in A_I(G {\setminus} J)$:
	there is a directed path $v_1, v_2, \ldots, v_n$
	such that $v_1 = k$, $v_n \in I$,
	and $v_p \notin J$ for each $p \in \{1, 2, \ldots, n\}$.
	Our goal is to prove that $k \in A$.
	If $n = 1$, we have directly $k \in I \subseteq A$.
	Now consider the case where $n \ge 2$.
	Assume for the sake of contradiction that $k \notin A$,
	that is, $k \in B$,
	so that $v_1 = k \in B$ and $v_n \in A$.
	Then we can define
	$p = \max\{q \in \{1, 2, \ldots, n-1\} | v_q \in B\}$,
	and we have $(v_p, v_{p+1}) \in E \cap (B \times A)$.
	We also know by construction of the path that $k_p \notin J$.
	This contradicts our assumption that $J$ is the second source of the cut~$(A, B)$.
	Hence, $k \in A$.
\end{proof}

\begin{proof}[Proof of \Cref{prop:jaf-cuts-1,prop:jaf-cuts-2}]
	\Cref{prop:jaf-cuts-1} is a special case of \Cref{prop:jaf-cuts-2}
	because the only nonempty subset of a singleton is the singleton itself.
	Therefore, in the remainder, we focus on proving \Cref{prop:jaf-cuts-2}.
	
	Let us first prove \Cref{prop:jaf-cuts-2}\ref{prop:ancestors-4}.
	Assume that $I$ and $J$ are joint-ancestor free,
	that is, $A_I(G {\setminus} J) \cap A_J(G {\setminus} I) = \emptyset$.
	Combining this assumption with
	\Cref{lem:jaf-cuts}\ref{prop:ancestors-1}
	shows that $(A_I(G {\setminus} J), A_J(G {\setminus} I))$
	is a cut.
	Assume for the sake of contradiction that
	the source $(\underline{I}, \underline{J})$
	of the cut $(A_I(G {\setminus} J), A_J(G {\setminus} I))$
	does not satisfy $\underline{I} \subseteq I$
	and $\underline{J} \subseteq J$.
	Specifically, suppose
	there exists $(k, \ell) \in E \cap (A_I(G {\setminus} J) \times A_J(G {\setminus} I))$
	such that $k \notin I$.
	Since $\ell \in A_J(G {\setminus} I)$,
	there exists a directed path
	$v_1, v_2, \ldots, v_n$, with $n \ge 1$,
	such that $v_1 = \ell$, $v_n = j \in J$,
	and $v_p \notin I$ for each $p \in \{1, 2, \ldots, n\}$.
	Since we assumed $k \notin I$,
	it follows that $k, v_1, v_2, \ldots, v_n$ is a directed path
	from $k$ to~$j$ in $G {\setminus} I$,
	hence $k \in A_J(G {\setminus} I)$,
	which is impossible because we assumed $k \in A_I(G {\setminus} J)$
	and $A_I(G {\setminus} J) \cap A_J(G {\setminus} I) = \emptyset$.
	
	Next, we prove \Cref{prop:jaf-cuts-2}\ref{prop:ancestors-5}.
	By \Cref{lem:jaf-cuts}\ref{prop:ancestors-2}, any cut $(A, B)$ with source $(I, J)$ is such that
	$A_I(G {\setminus} J) \subseteq A$
	and $A_J(G {\setminus} I) \subseteq B$.
	But we assumed $(A_I(G {\setminus} J), A_J(G {\setminus} I))$ forms a cut,
	and no nodes can be outside of the cut.
	Thus, $A_I(G {\setminus} J) = A$ and $A_J(G {\setminus} I) = B$.
	
	Lastly, we prove \Cref{prop:jaf-cuts-2}\ref{prop:ancestors-3}.
	Assume that $I$ and $J$ are not joint-ancestor free,
	i.e., $A_I(G {\setminus} J) \cap A_J(G {\setminus} I) \neq \emptyset$.
	Assume for the sake of contradiction that
	there is a cut $(A, B)$ with source $(I, J)$.
	\Cref{lem:jaf-cuts}\ref{prop:ancestors-2} implies that
	$A_I(G {\setminus} J) \subseteq A$
	and $A_J(G {\setminus} I) \subseteq B$,
	which in turn implies that
	$A \cap B \supseteq A_I(G {\setminus} J) \cap A_J(G {\setminus} I) \neq \emptyset$.
	This contradicts our assumption that $(A, B)$ is a cut.
\end{proof}

\subsection{S-product-form implies joint-ancestor freeness} \label{sec:s-product-form-3}

Our last step in the proof of \Cref{theo:s-product-form}
is to show the implication
\eqref{item:S-product-form-1} $\implies$ \eqref{item:S-product-form-3},
that is,
S-product-form relationship implies joint-ancestor freeness.
We have shown so far that statements
\eqref{item:S-product-form-2} and \eqref{item:S-product-form-3}
are equivalent to each other,
and that they imply \eqref{item:S-product-form-1}.
In other words,
we proved that if a\add{n} $i,j$-sourced cut exists, or equivalently if nodes~$i$ and~$j$ have no joint ancestor, then an S-product-form relationship between $i$ and $j$ exists. Specifically, a product-form relationship exists where $f_{i,j}$ depends only on transition rates along edges whose source is $i$, and $f_{j,i}$ depends only on transition rates along edges whose source is $j$.

\Cref{theo:no-product-form} below proves this condition is necessary.
The intuition behind the proof is as follows.
If an $i,j$-sourced cut does not exist, then by \Cref{prop:jaf-cuts-1} there exist nodes 
which are \add{joint} ancestors of \remove{both }$i$ and $j$, i.e., $A_i(G\setminus j) \cap A_j(G {\setminus} i)$ is nonempty. In particular, there must exist a node $k$ which is an ancestor of both nodes $i$ and $j$ via disjoint paths.
If such a node $k$ exists, we show that the ratio $\frac{\pi_i}{\pi_j}$ depends on edge weights $q_{k, k'}$ with source $k$.
This violates the definition of S-product-form given in \Cref{sec:product-form-def}, so \Cref{theo:no-product-form} shows that a joint ancestor implies no S-product-form, or equivalently that S-product-form implies no joint ancestor.

\begin{theorem}
	\label{theo:no-product-form}
	Consider a formal Markov chain $G = (V, E)$
	and let $i, j \in V$.
	Suppose nodes~$i$ and~$j$ are not joint-ancestor free,
	i.e., $A_i(G {\setminus} j) \cap A_j(G {\setminus} i) \neq \emptyset$.
	
	Then there exists a node $k \in A_i(G {\setminus} j) \cap A_j(G {\setminus} i)$ such that the stationary probability ratio $\frac{\pi_i}{\pi_j}$ depends on at least one of \add{the} edge weights emerging from $k$.
\end{theorem}

\begin{proof}
	\add{We will demonstrate the existence
		of two instantiations of the vector~$q$,
		denoted by $\qa$ and $\qb$, such that
		(i) $\qa$ and $\qb$ differ only at edges
		emerging from a particular source node
		$k \in A_i(G {\setminus} j) \cap A_j(G {\setminus} i)$,
		and (ii) the stationary distributions
		$\pia$ and $\pib$
		associated with $\qa$ and $\qb$ satisfy
		\begin{align} \label{eq:inequality}
			\frac{\pia_i}{\pia_j} \neq \frac{\pib_i}{\pib_j}.
		\end{align}
		The proof is divided into 6 steps.
		In Step~\ref{step1}, we identify the vertex~$k \in A_i(G {\setminus} j) \cap A_j(G {\setminus} i)$ for which the result will eventually be proven.
		In Step~\ref{step2}, we derive a convenient expression for the ratios in~\eqref{eq:inequality}, in terms of a subchain of the original Markov chain restricted to nodes~$i$, $j$, and~$k$.
		As this expression is only valid for positive-recurrent Markov chains, in the remainder of the proof we will focus on instantiations of the vector~$q$ that are positive recurrent (in the sense that the corresponding Markov chain is positive recurrent).
		In Step~\ref{step3}, we build a particular vector~$q$,
		and we apply the Foster-Lyapunov theorem to show
		that~$q$ is positive recurrent.
		In Step~\ref{step4}, we construct~$\qa$ and~$\qb$ by altering~$q$,
		and we show that these vectors are positive recurrent,
		again using the Foster-Lyapunov theorem.
		In Steps~\ref{step5} and~\ref{step6},
		we finally prove that $\qa$ and~$\qb$
		satisfy~\eqref{eq:inequality},
		using the expression derived in Step~\ref{step2}.
		Throughout the proof, we see $q$, $\qa$, and $\qb$ as transition \emph{probability} vectors, i.e., we focus on DTMCs.}
	
	\newcounter{step}
	\setcounter{step}{0}
	
	\refstepcounter{step} 
	\paragraph*{\add{Step~\thestep: Specify the vertex~$k$}}
	\label{step1}
	
	We \remove{will} choose $k$ to be a node in $A_i(G {\setminus} j) \cap A_j(G {\setminus} i)$ such that $k$ is an ancestor of both nodes $i$ and $j$ via disjoint paths.
	To see why such a node must exist, consider an arbitrary node $k'$ in $A_i(G {\setminus} j) \cap A_j(G {\setminus} i)$. From any such node, there exist paths from $k'$ to $i$ and $k'$ to $j$.
	Consider the shortest such paths. Either these paths are disjoint,
	or there is a node $k''$ which is a joint ancestor of $i$ and $j$ via shorter paths than $k'$.
	Over all nodes in $A_i(G {\setminus} j) \cap A_j(G {\setminus} i)$,
	let $k$ be the node that is the ancestor of $i$ via the shortest path length, choosing arbitrarily in case of a tie.
	By the above argument, the shortest paths from $k$ to $i$ and from $k$ to $j$ must be disjoint.
	
	\refstepcounter{step} 
	\paragraph*{\add{Step~\thestep: Define the subchains restricted to nodes~$i$, $j$, and $k$}}
	\label{step2}
	
	\add{Next, for an arbitrary transition probability vector~$q$ giving rise to a positive-recurrent Markov chain,
		we examine the probability of transitioning between nodes $i, j,$ and~$k$. More specifically,}
	\replace{C}{c}onsider the subchain,
	with node set $U = \{i, j, k\}$,
	obtained by looking at the subsequence of states
	visited by the original Markov chain $G = (V, E)$
	inside the set~$U$.
	This subchain satisfies the Markov property,
	and for each $u, v \in U$\add{,}
	we let $p_{u, v}$ denote
	its transition probability
	from state~$u$ to state~$v$;
	these are functions of the original Markov chain's
	transition rates~$q_{i, j}$ for $i, j \in V$.
	Critically, letting $\pi_i$, $\pi_j$, and $\pi_k$
	denote the stationary probabilities for states $i$, $j,$ and $k$
	in the original Markov chain \add{under edge probabilities $q$},
	as defined by~\eqref{eq:balance-def},
	we can verify that
	\begin{align*}
		\frac{\pi_i}{\pi_i + \pi_j + \pi_k}, \quad
		\frac{\pi_j}{\pi_i + \pi_j + \pi_k}, \quad
		\text{and} \quad
		\frac{\pi_k}{\pi_i + \pi_j + \pi_k}
	\end{align*}
	form the stationary distribution of the \add{embedded DTMC corresponding to the }subchain.
	Using this definition,
	we will quantify the relative state probabilities $\pi_i$ and $\pi_j$ in terms of the subchain transition probabilities $p_{u, v}$ for $u,v \in U$.
	
	Using the balance equations for the \add{embedded DTMC corresponding to the} subchain, we obtain successively:
	\begin{align}
		\label{eq:subchain-1}
		\pi_k &= \replace{p_{ik}}{p_{i,k}} \pi_i + \replace{p_{jk}}{p_{j,k}} \pi_j + \replace{p_{kk}}{p_{k,k}} \pi_k, \\
		\label{eq:subchain-2}
		\pi_k &= \frac{\replace{p_{ik}}{p_{i,k}} \pi_i + \replace{p_{jk}}{p_{j,k}} \pi_j}{1-\replace{p_{kk}}{p_{k,k}}}, \\
		\label{eq:subchain-3}
		\pi_i &= \replace{p_{ii}}{p_{i,i}} \pi_i + \replace{p_{ji}}{p_{j,i}} \pi_j + \replace{p_{ki}}{p_{k,i}} \pi_k, \\
		\label{eq:subchain-4}
		\pi_i &= \replace{p_{ii}}{p_{i,i}} \pi_i + \replace{p_{ji}}{p_{j,i}} \pi_j + \replace{p_{ki}}{p_{k,i}} \frac{\replace{p_{ik}}{p_{i,k}} \pi_i + \replace{p_{jk}}{p_{j,k}} \pi_j}{1-\replace{p_{kk}}{p_{k,k}}}, \\
		\label{eq:subchain-5}
		\pi_i \left(1 - \replace{p_{ii}}{p_{i,i}} - \frac{\replace{p_{ik}}{p_{i,k}} \replace{p_{ki}}{p_{k,i}}}{1-\replace{p_{kk}}{p_{k,k}}} \right)
		&= \pi_j \left(\replace{p_{ji}}{p_{j,i}} + \frac{\replace{p_{jk}}{p_{j,k}} \replace{p_{ki}}{p_{k,i}}}{1-\replace{p_{kk}}{p_{k,k}}}\right), \\
		\label{eq:subchain-6}
		\pi_i \left(\replace{p_{ij}}{p_{i,j}} + \replace{p_{ik}}{p_{i,k}} - \frac{\replace{p_{ik}}{p_{i,k}} \replace{p_{ki}}{p_{k,i}}}{\replace{p_{ki}}{p_{k,i}} + \replace{p_{kj}}{p_{k,j}}} \right)
		&= \pi_j \left(\replace{p_{ji}}{p_{j,i}} + \frac{\replace{p_{jk}}{p_{j,k}} \replace{p_{ki}}{p_{k,i}}}{\replace{p_{ki}}{p_{k,i}} + \replace{p_{kj}}{p_{k,j}}}\right), \\
		\label{eq:subchain-7}
		\pi_i \left(\replace{p_{ij}}{p_{i,j}} + \replace{p_{ik}}{p_{i,k}} \frac{\replace{p_{kj}}{p_{k,j}}}{\replace{p_{ki}}{p_{k,i}} + \replace{p_{kj}}{p_{k,j}}} \right)
		&= \pi_j \left(\replace{p_{ji}}{p_{j,i}} + \replace{p_{jk}}{p_{j,k}} \frac{\replace{p_{ki}}{p_{k,i}}}{\replace{p_{ki}}{p_{k,i}} + \replace{p_{kj}}{p_{k,j}}}\right).
	\end{align}
	Equations~\eqref{eq:subchain-1} and~\eqref{eq:subchain-3}
	are the balance equations of the \add{embedded DTMC of the} subchain
	for states~$k$ and~$i$.
	Equation~\eqref{eq:subchain-2} follows
	by solving~\eqref{eq:subchain-1} with respect to~$\pi_k$,
	and once injected into~\eqref{eq:subchain-3}
	it yields~\eqref{eq:subchain-4}.
	Equation~\eqref{eq:subchain-5}
	follows by rearranging~\eqref{eq:subchain-4},
	and becomes~\eqref{eq:subchain-6} after injecting
	$\replace{p_{ii}}{p_{i,i}} + \replace{p_{ij}}{p_{i,j}} + \replace{p_{ik}}{p_{i,k}} = 1$.
	Equation~\eqref{eq:subchain-7} then follows
	by \replace{rearraging}{rearranging}~\eqref{eq:subchain-6}.
	
	\refstepcounter{step} 
	\paragraph*{\add{Step~\thestep: Construct a positive-recurrent vector~$q$}}
	\label{step3}
	
	\add{We now construct a transition probability vector~$q$ with support~$E$ that is positive-recurrent (in the sense that the Markov chain it gives rise to is positive recurrent). We can see~$q$ as a possible instantiation of the formal Markov chain's edge weights. This vector will be our starting point for constructing $\qa$ and $\qb$ in the next step and showing that $\qa$ and $\qb$ are also positive recurrent. The reason why we want $\qa$ and $\qb$ to be positive recurrent, is to guarantee that their associated subchain (see Step~\ref{step2}) is well-defined.}
	
	\add{We start by arbitrarily selecting a root node $r \in V$.
		Consider the function $h: \ell \in V \mapsto h(\ell) \in \{0, 1, 2, \ldots\}$ such that
		$h(\ell)$ is the directed-shortest-path distance from $\ell$ to the root node~$r$,
		for each $\ell \in V$.
		We will use~$h$ first to define the vector~$q$, and then as a Lyapunov function to prove positive recurrence of~$q$.}
	
	\add{For a given node of the graph $m$, we now define the transition probabilities $q_{m, n}$ for transitions leaving node $m$.
		Note that, for each $m \in V {\setminus} \{r\}$,
		there must be a node $m' \in N_m$ for which $h(m') = h(m) - 1$,
		where $N_m$ is the neighborhood of~$m$.
		We may select $m'$ to be the first step on a shortest path from $m$ to the root node~$r$.
		Let us label the nodes in $N_m {\setminus} \{m'\}$ as $m_1, m_2, \ldots$ in some arbitrary order.
		We assign the transition probabilities
		\begin{align*}
			q_{m, m_i} &= \frac{1}{2} \frac{1}{2^i} \frac{1}{h(m_i)+1}
			~\text{for each $i \in \{1, 2, \ldots, |N_m|\}$}, &
			q_{m, m'} &= 1 - \sum_{i=1}^{|N_m|} q_{m, m_i}.
		\end{align*}
		Note that $\sum_{i=1}^{|N_m|} q_{m, m_i} \le \frac{1}{2}$, as $h$ is non-negative, and that $|N_m| \in \{1, 2, \ldots\} \cup \{+\infty\}$.
		This procedure defines $q_{\ell, \ell'}$ for each $(\ell, \ell') \in E$ with $\ell \neq r$.
		To define $q_{r, \ell'}$ for each $(r, \ell') \in E$,
		we select some arbitrary $m'_0 \in N_r$
		and use the same procedure.}
	
	\add{With the transition probabilities $q$ defined in this way,
		we now prove that the Markov chain is positive recurrent,
		using $h$ as a Lyapunov function for the Foster-Lyapunov theorem \cite[Theorem~7.1.1]{B20}.
		We define the finite set $F = \{r\}$ consisting only of the root node, and we let $\delta = 1/2$.
		We will show that the following two conditions hold:
		\begin{align}
			\label{eq:negative-drift}
			\sum_{\ell' \in N_\ell} q_{\ell, \ell'} h(\ell')
			&\le h(\ell) - \delta,
			\quad \text{for each $\ell \in V {\setminus} F$}, \\
			\label{eq:finite-expectation}
			\sum_{\ell' \in N_\ell} q_{\ell, \ell'} h(\ell')
			&< \infty,
			\quad \text{for each $\ell \in F$}.
		\end{align}
		By the Foster-Lyapunov theorem, these conditions imply that $q$ is positive recurrent.}
	
	\add{First, let us prove the negative-drift equation~\eqref{eq:negative-drift} for each node
		in $V {\setminus} \{r\}$.
		We call this node $m$ to match the notation above:
		\begin{align*}
			\sum_{\ell' \in N_m} q_{m, \ell'} h(\ell')
			&\overset{\text{(i)}}= q_{m, m'} h(m')
			+ \sum_{i=1}^{|N_m|} q_{m, m_i} h(m_i), \\
			&\overset{\text{(ii)}}= q_{m, m'} (h(m) - 1)
			+ \sum_{i=1}^{|N_m|} \frac{1}{2} \frac{1}{2^i} \frac{1}{h(m_i)+1} h(m_i), \\
			&\overset{\text{(iii)}}\le q_{m, m'} (h(m) - 1)
			+ \sum_{i=1}^{|N_m|} \frac{1}{2} \frac{1}{2^i}, \\
			&\overset{\text{(iv)}}= q_{m, m'} (h(m) - 1) + \frac{1}{2}
			\overset{\text{(v)}}\le h(m) - 1 + \frac{1}{2} = h(m) - \frac{1}{2},
		\end{align*}
		where (i) follows by enumerating the out-neighbors of~$m$
		using the same notation as before,
		(ii) from the definition of the transition rates,
		(iii) because $h$ is positive,
		(iv) by upper-bounding the sum
		over $\{1, 2, \ldots, |N_m|\}$
		with the same sum over $\{1, 2, \ldots, +\infty\}$,
		and (v) from the fact that $q_{m, m'} \le 1$.
		We thus confirm \eqref{eq:negative-drift} for all nodes $m$ other than the root node $r$.}
	
	\add{For the root node $r$, a similar derivation
		confirms the finite-expectation equation~\eqref{eq:finite-expectation} by showing that
		\begin{align*}
			\sum_{\ell' \in N_r} q_{r, \ell'} h(\ell') \le h(m'_0) + \frac{1}{2}.
		\end{align*}%
	}
	
	\refstepcounter{step}
	\paragraph*{\add{Step~\thestep: Construct positive-recurrent vectors $\qa$ and $\qb$ by altering~$q$}}
	\label{step4}
	
	\add{We have shown that $q$ is positive recurrent, in the sense that the Markov chain defined by the graph~$G$ and the transition probability vector~$q$ is positive recurrent.
		We are now ready to define $\qa$ and $\qb$, two transition probability vectors that differ from~$q$ only along edges with source~$k$.
		We will show later in this step that $\qa$ and $\qb$
		are positive recurrent,
		and in Steps~\ref{step5} and~\ref{step6} we will show
		that their associated stationary distributions~$\pia$ and~$\pib$ satisfy~\eqref{eq:inequality}.}
	
	\replace{Let $a$ be the shortest path from $k$ to $i$, $a := k, a_2, a_3, \ldots, i$. Let $b$ be the shortest path from $k$ to $j$, $b := k, b_2, b_3, \ldots, j$.}{Let $a = k, a_2, a_3, \ldots, a_{|a|}, i$ be the shortest path from $k$ to $i$ and $b = k, b_2, b_3, \ldots, b_{|b|}, j$ be the shortest path from $k$ to $j$.}
	\add{Also let $V_a = \{k, a_2, a_3, \ldots, a_{|a|}\}$ (resp.\ $V_b = \{k, b_2, b_3, \ldots, b_{|b|}\}$) denote the set nodes that belong to path~$a$ (resp.\ $b$), except~$i$ (resp.\ $j$).}
	\add{Occasionally, we will use the notation
		$a_{|a|+1} = i$ and $b_{|b|+1} = j$.}
	\replace{As shown above,}{Recall from the definition of $k$ above that} $a$ and $b$ share no nodes except~$k$.
	
	\add{We now define transition probability vectors $\qa$ and $\qb$ that take the same value at all edges except $(k, a_2)$ and $(k, b_2)$.
		Let $\epsilon \in (0, 1)$ be a sufficiently small constant, smaller than a value to be precised later.}
	
	\add{First, for each edge $(\ell, \ell') \in E {\setminus} \{(k, a_2), (k, b_2)\}$, we define $\qa_{\ell, \ell'} = \qb_{\ell, \ell'}$ as follows:
		\begin{itemize}
			\item If $\ell \in V {\setminus} (V_a \cup V_b)$,
			i.e., either $\ell \in \{i, j\}$
			or $\ell$ is in neither $a$ nor $b$,
			we define
			\begin{align*}
				\qa_{\ell, \ell'} = \qb_{\ell, \ell'}
				\triangleq q_{\ell, \ell'}.        
			\end{align*}
			\item Next, we turn to the case where $\ell \in V_a {\setminus} \{k\}$, i.e., $\ell$ is within $a$ but is not $k$ or $i$, so that we have $\ell = a_p$ for some $p \in \{2, 3, \ldots, |a|\}$.    
			We let, for each $p \in \{2, 3, \ldots, |a|\}$,
			\begin{align*}
				\qa_{a_p, \ell'}
				= \qb_{a_p, \ell'}
				\triangleq \begin{cases}
					1 - \epsilon,
					&\text{if $\ell' = a_{p+1}$}, \\
					\displaystyle \epsilon \frac{q_{a_p,\ell'}}{\sum_{a'' \in N_{a_p} {\setminus} \{a_{p+1}\}} q_{a_p, a''}}
					&\text{if $\ell' \in N_{a_p} {\setminus} \{a_{p+1}\}$}.
				\end{cases}
			\end{align*}
			We define in the same manner weights of edges whose source node belongs to~$V_b {\setminus} \{k\}$.
			\item If $\ell = k$ and $\ell' \notin \{a_2, b_2\}$, we define
			\begin{align*}
				\qa_{k,\ell'} = \qb_{k,\ell'}
				&\triangleq \frac{\epsilon}{2} \frac{q_{k,\ell'}}{\sum_{\ell'' \in N_{k} {\setminus} \{a_2, b_2\}} q_{k, \ell''}}.
			\end{align*}
		\end{itemize}
		Finally, we define the two transition probabilities that differ between $\qa$ and $\qb$, namely, those corresponding to the edges $(k, a_2)$ and $(k, b_2)$:
		\begin{align*}
			\qa_{k,a_2} &\triangleq 1 - \epsilon, &
			\qa_{k,b_2} &\triangleq \frac{\epsilon}{2}, &
			\qb_{k,b_2} &\triangleq 1 - \epsilon, &
			\qb_{k,a_2} &\triangleq \frac{\epsilon}{2}.
		\end{align*}
		One can verify that $\sum_{\ell' \in V} q_{\ell, \ell'} = 1$
		for each $\ell \in V$.
		We assume that $\epsilon \in (0, 1)$ is small enough
		so that the transition probabilities in $\qa$ and $\qb$
		defined as multiple values of $\epsilon$
		are smaller than their $q$ counterparts;
		in other words, we assume that
		\begin{align*}
			\qa_{a_p, \ell'} &= \qb_{a_p, \ell'} \le q_{a_p, \ell'}
			\text{ for each $p \in \{2, 3, \ldots, |a|\}$ and $\ell' \in N_{a_p} \setminus \{a_{p+1}\}$}, \\
			\qa_{b_p, \ell'} &= \qb_{b_p, \ell'} \le q_{b_p, \ell'}
			\text{ for each $p \in \{2, 3, \ldots, |b|\}$
				and $\ell' \in N_{b_p} \setminus \{b_{p+1}\}$}, \\
			\qa_{k, \ell'} &= \qb_{k, \ell'} \le q_{k, \ell'}
			\text{ for each $\ell' \in N_k \setminus \{a_2, b_2\}$}, \\
			\qa_{k, b_2} &\le q_{k, b_2}, \quad \text{and} \quad
			\qb_{k, a_2} \le q_{k, a_2}.
		\end{align*}%
	}
	
	\add{Let us now prove that $\qa$ and $\qb$ are positive recurrent.
		Similarly to Step~\ref{step3},
		we prove variants of~\eqref{eq:negative-drift}
		and~\eqref{eq:finite-expectation}
		in order to apply the Forster-Lyapunov theorem.
		Compared to Step~\ref{step3},
		we use the same Lyapunov function $h$, namely, $h(\ell)$ is the directed-shortest-path distance from $\ell$ to the same root node~$r$, for each $\ell \in V$,
		and we replace the finite exclusion set~$F$
		with $F' = \{r\} \cup V_a \cup V_b$.
		We now prove that $\qa$ and $\qb$ satisfy these new variants of
		the negative-drift equation~\eqref{eq:negative-drift}
		and the finite-expectation equation~\eqref{eq:finite-expectation}:
		\begin{itemize}
			\item For the negative-drift equation~\eqref{eq:negative-drift},
			note that
			$q^a_{\ell,\ell'} = q^b_{\ell, \ell'} = q_{\ell, \ell'}$
			for all $\ell \in V {\setminus} F'$ and $\ell' \in N_\ell$,
			and we already proved that~$q$ satisfies~\eqref{eq:negative-drift}
			over $V {\setminus} F \supseteq V {\setminus} F'$.
			\item Let us now prove that $\qa$ satisfies
			the finite-expectation equation~\eqref{eq:finite-expectation};
			the proof for $\qb$ is symmetrical.
			Let $\ell \in F'$.
			Since the paths $a$ and $b$ intersect only in~$k$
			and each of them visits each node at most once
			(by definition of a path),
			at most one transition outside~$\ell$ has a larger probability
			in $\qa$ than in $q$,
			namely, the successor of~$\ell$ in $a$ or~$b$, if any;
			in the special case where $\ell = k$,
			$\ell$ has two successors in~$a$ and~$b$ ($a_2$ and $b_2$),
			but we chose $\epsilon$ such~ that $\qa_{k, b_2} \le q_{k, b_2}$.
			This observation implies that
			the finiteness of $\sum_{\ell' \in V} \qa_{\ell, \ell'} h(\ell')$
			follows from that of $\sum_{\ell' \in V} q_{\ell, \ell'} h(\ell')$,
			which follows from~\eqref{eq:negative-drift} and~\eqref{eq:finite-expectation}.
			For instance, if $\ell = a_p$
			for some $p \in \{2, 3, \ldots, |a| - 1\}$,
			we have $\qa_{a_p, \ell'} < q_{a_p, \ell'}$
			for each $\ell' \in N_{a_p} {\setminus} \{a_{p+1}\}$.
			Thus, recalling that $h$ is positive
			and $\qa_{a_p, a_{p+1}} \le 1$,
			we can bound the expected value of~$h$
			after one transition:
			\begin{align*}
				\sum_{\ell' \in N_{a_p}} \qa_{a_p,\ell'} h(\ell')
				\le h(a_{p+1}) + \sum_{\ell' \in N_{a_p}} q_{a_p, \ell'} h(\ell').
			\end{align*}
			The right-hand side is finite because~$q$
			satisfies~\eqref{eq:negative-drift} and~\eqref{eq:finite-expectation}.
			A similar argument confirms finite expectation for all nodes in $F'$ under each of $\qa$ and $\qb$.
		\end{itemize}
		This completes the proof of positive recurrence via the Foster-Lyapunov theorem.}
	
	\add{Let $\pia$ and $\pib$ denote the stationary distributions
		associated with $\qa$ and $\qb$, respectively.
		Because $\qa$ and $\qb$ are positive recurrent,
		we can use them to construct subchains restricted to nodes~$i$, $j$, and~$k$,
		as we did in Step~\ref{step2},
		and we denote their transition probabilities by $\pa$ and $\pb$, respectively.
		It follows from~\eqref{eq:subchain-7} in Step~\ref{step2} that
		\begin{align*}
			\frac{\pia_i}{\pia_j}
			&= \frac
			{\pa_{j, i} + \pa_{j, k} \frac{\pa_{k, i}}{\pa_{k, i} + \pa_{k, j}}}
			{\pa_{i, j} + \pa_{i, k} \left( 1 - \frac{\pa_{k, i}}{\pa_{k, i} + \pa_{k, j}} \right)}, &
			\frac{\pib_i}{\pib_j}
			&= \frac
			{\pb_{j, i} + \pb_{j, k} \frac{\pb_{k, i}}{\pb_{k, i} + \pb_{k, j}}}
			{\pb_{i, j} + \pb_{i, k} \left( 1 - \frac{\pb_{k, i}}{\pb_{k, i} + \pb_{k, j}} \right)}.
		\end{align*}
		Thus, $\pia$ and $\pib$ satisfy~\eqref{eq:inequality}
		if and only if $\pa$ and $\pb$ are such that
		\begin{align*}
			\frac
			{\pa_{j, i} + \pa_{j, k} \frac{\pa_{k, i}}{\pa_{k, i} + \pa_{k, j}}}
			{\pa_{i, j} + \pa_{i, k} \left( 1 - \frac{\pa_{k, i}}{\pa_{k, i} + \pa_{k, j}} \right)}
			&\neq \frac
			{\pb_{j, i} + \pb_{j, k} \frac{\pb_{k, i}}{\pb_{k, i} + \pb_{k, j}}}
			{\pb_{i, j} + \pb_{i, k} \left( 1 - \frac{\pb_{k, i}}{\pb_{k, i} + \pb_{k, j}} \right)}.
		\end{align*}
		In turn, to show this inequality, it is sufficient to prove that
		\begin{align}
			\label{eq:prop-pa-pb}
			\pa_{i, j} &= \pb_{i, j}, &
			\pa_{i, k} &= \pb_{i, k}, &
			\pa_{j, i} &= \pb_{j, i}, &
			\pa_{j, k} &= \pb_{j, k}, &
			\frac{\pa_{k, i}}{\pa_{k, i} + \pa_{k, j}}
			&\neq \frac{\pb_{k, i}}{\pb_{k, i} + \pb_{k, j}}.
		\end{align}
		The equalities in~\eqref{eq:prop-pa-pb} are shown in Step~\ref{step5},
		and the inequality is shown in Step~\ref{step6}.}
	
	\refstepcounter{step} 
	\paragraph*{\add{Step~\thestep: Verify that $\qa$ and $\qb$ imply the equalities in~\eqref{eq:prop-pa-pb}}}
	\label{step5}
	
	\add{Our goal at this step is to prove the equalities in~\eqref{eq:prop-pa-pb}.
		These equalities follow from the fact that
		$\qa_{\ell, \ell'} = \qb_{\ell, \ell'}$
		for each $(\ell, \ell') \in E$ with $\ell \neq k$.
		Because the argument is similar for all four equalities,
		we focus on the first.
		By definition, $\pa_{i, j}$ and $\pb_{i, j}$ are the sums of the probabilities of the trajectories in~$G$ that start in state~$i$, end in state~$j$, and never visit states $i$, $j$, or $k$ in-between; unlike our definition of a path, a trajectory may visit the same node multiple times and may therefore have an arbitrary (but finite) length.
		We let $\mathcal{F}_{i, j}$ denote this set of trajectories,
		which is entirely determined by~$G$.
		The probability of a trajectory in~$\mathcal{F}_{i, j}$ is
		the product of the transition probabilities
		along the edges that compose the trajectory.
		By definition, none of the trajectories in $\mathcal{F}_{i, j}$
		contains any transitions leaving state~$k$.
		It follows that $\qa_{\ell, \ell'} = \qb_{\ell, \ell'}$
		for any edge $(\ell, \ell')$ that appears
		in the trajectories in~$\mathcal{F}_{i, j}$,
		and then that $\pa_{ij} = \pb_{ij}$.
		A similar argument proves the other equalities in~\eqref{eq:prop-pa-pb}.}
	
	\refstepcounter{step} 
	\paragraph*{\add{Step~\thestep: Verify that $q^a$ and $q^b$ imply the inequality in~\eqref{eq:prop-pa-pb}}}
	\label{step6}
	
	\add{Our goal at this step is to prove the inequality in~\eqref{eq:prop-pa-pb}.
		We will prove this inequality by lower-bounding the left-hand side
		and upper-bounding the right-hand side.
		First focusing on $\qa$, we have
		\begin{align*}
			\pa_{k, i}
			&\overset{\text{(i)}}{\ge} \qa_{k, a_2} \qa_{a_2, a_3} \cdots \qa_{a_{|a|}, i}
			\overset{\text{(ii)}}{=} (1 - \epsilon)^{|a|},
		\end{align*}
		where (i) follows by observing that path~$a$ is a trajectory
		that starts in state~$k$, ends in state~$i$,
		and never visits states $i$, $j$, or $k$ in-between
		(i.e., $a \in \mathcal{F}_{k, i}$ with the notation of Step~6),
		and (ii) follows by definition of~$\qa$.
		By applying a similar reasoning for path~$b$
		under the transition probability vector~$\qb$,
		we conclude that $\pb_{k, j} \ge (1 - \epsilon)^{|b|}$.}
	
	\add{Let us now assume~$\epsilon \le \frac1{3 \max(|a|, |b|)}$.
		One can verify that, for each $x \ge 1$, we have
		$(1 - \frac1{3x})^x \in [\frac23, e^{-\frac13})$.
		It follows that $\pa_{k, i} \ge \frac23$
		and $\pb_{k, j} \ge \frac23$.
		We conclude that
		\begin{align*}
			\frac{\pa_{k, i}}{\pa_{k, i} + \pa_{k, j}}
			&\overset{\text{(i)}}{\ge} \pa_{k, i}
			\overset{\text{(ii)}}{\ge} \frac23, &
			\frac{\pb_{k, i}}{\pb_{k, i} + \pb_{k, j}}
			&\overset{\text{(iii)}}{\le} \frac{\pb_{k, i}}{\frac23}
			\overset{\text{(iv)}}{\le} \frac12,
		\end{align*}
		where (i) follows from the fact that
		$\pa_{k, i} + \pa_{k, j} \le 1$,
		(ii) from the fact that $\pa_{k, i} \ge \frac23$,
		(iii) from the fact that
		$\pb_{k, i} + \pb_{k, j} \ge \pb_{k, j} \ge \frac23$,
		and (iv) from the fact that $\pb_{k, i} \le 1 - \pb_{k, j} \le \frac13$.
		These bounds imply the inequality in~\eqref{eq:prop-pa-pb}, completing the proof.
	}
\end{proof}

\section{Cut graph and higher-level cuts} \label{sec:higher-levels}

In this section, we explore product-form relationships \add{that go} beyond the S-product-form
\replace{which was the focus of}{considered in} \Cref{sec:s-product-form}.
In \Cref{sec:cut-graph}, we define the cut graph, which is an undirected graph whose edges represent S-product-form relationships between nodes.
Based on this definition, in \Cref{sec:ps-product-form}, we explore PS-product-form relationships, corresponding to combinations of S-product-form relationships produced by paths in the cut graph.
In \Cref{sec:second-level,sec:second-level-product}, we \add{define SPS-product-form relationships and beyond, and we} introduce and explore \add{the corresponding }higher-level cuts\remove{, which correspond to SPS-product-form relationships and beyond}.

\subsection{Cut graph}
\label{sec:cut-graph}

Let us first introduce the cut graph of a formal Markov chain.

\begin{definition}[Cut graph] \label{def:cut-graph}
	Consider a \replace{strongly-connected directed graph}{formal Markov chain} $G = (V, E)$.
	The \emph{cut graph} of~$G$
	is the undirected graph $C_1(G) = (V, R)$
	where $R$ is the family of doubletons $\{i, j\} \subseteq V$
	such that \add{states}~$i$ and $j$ are in an S-product-form relationship.
	In other words, $C_1(G)$ is the graph of
	the S-product-form binary relation.
\end{definition}

Recall that, by \Cref{theo:s-product-form},
two nodes~$i$ and~$j$ are in an S-product-form relationship
if and only if there exists an $i, j$-sourced cut,
that is, $i$ and~$j$ are joint-ancestor free.
Using this observation,
\Cref{algo:cut-graph} returns the cut graph
of a formal Markov chain with a finite number of states.
It uses the \textsc{Mutually\-Avoiding\-Ancestors} procedure
from \Cref{algo:jaf}.
Because each call to this procedure takes time $O(|E|)$,
the \textsc{CutGraph} procedure runs in time $O(|V|^2 |E|)$.
If the cut graph is connected,
then the stationary distribution
can be entirely computed by applying S-product-form relationships.
Lastly, observe that the S-product-form relationship is not transitive,
i.e., if the pairs $i, j$ and $j, k$ of nodes are both \replace{o}{i}n an S-product-form relationship,
this does not imply that nodes $i, k$ are.
Instead, we show in \Cref{sec:ps-product-form} that nodes $i, k$ are then in a PS-product-form relationship.

\begin{algorithm}[htb]
	\caption{Returns the cut graph $C_1(G)$ of a finite directed graph~$G = (V, E)$}
	\label{algo:cut-graph}
	\begin{algorithmic}[1]
		\Procedure{CutGraph}{finite directed graph $G = (V, E)$}
		$\to$ the cut graph~$C_1(G)$
		\State $E' \gets \emptyset$
		\For{each pair of distinct nodes $i, j \in V$}
		\State $A_i, A_j \gets \Call{MutuallyAvoidingAncestors}{G, i, j}$
		\If{$A_i \cap A_j = \emptyset$}
		\State add edge $\{i, j\}$ to $E'$
		\EndIf
		\EndFor
		\State \Return $(V, E')$
		\EndProcedure
	\end{algorithmic}
\end{algorithm}

\replace{I}{To help build more intuition on \replace{graph-based product-form}{graph-based product form}, i}n Appendix~\ref{app:clique},
we relate the structure of a formal Markov chain~$G$
to the existence of a clique in its cut graph~$C_1(G)$.
This condition can be seen as an extension of \Cref{prop:jaf-cuts-1},
as an edge is a clique of size~2.
Appendix~\ref{app:clique} shows in particular that
the one-way cycle of \Cref{ex:one-way-cycle}
is the only finite formal Markov chain whose cut graph is the complete graph.

\subsection{PS-product-form relationships}
\label{sec:ps-product-form}

The next lemma shows that, if two nodes are connected in the cut graph $C_1(G)$, they are in a PS-product-form relationship.
In this way, each connected component of the cut graph forms a set of nodes that are pairwise in PS-product-form relationships. An example \replace{of}{with} a fully-connected cut graph appears in \Cref{sec:msj-example}.

\begin{lemma}
	\label{lem:cut-graph-ps}
	Consider a formal Markov chain~$G = (V, E)$
	and two nodes $i, j \in V$.
	\add{Assume that $i$ and $j$ belong to the same connected component of $C_1(G)$.}
	\replace{Let $i = k_1, k_2, \ldots, k_d, k_{d+1} = j$ denote a path of length~$d$
		between nodes~$i$ and~$j$ in the cut graph $C_1(G)$,
		where $d$ is the distance between nodes~$i$ and~$j$ in $C_1(G)$.}{Let $d$ denote the distance between $i$ and $j$ in $C_1(G)$ and $i = k_1, k_2, \ldots, k_d, k_{d+1} = j$ a path of length~$d$ in $C_1(G)$.}
	Then
	\begin{align} \label{eq:cut-graph-product-form}
		\pi_i \prod_{p = 1}^d f_{k_p, k_{p+1}}
		&= \pi_j \prod_{p = 1}^d f_{k_{p+1}, k_p}.
	\end{align}
	where the $f$'s are given by Equation~\eqref{eq:first-level-product-form} in \Cref{theo:s-product-form}. \replace{In this case}{Hence}, nodes~$i$ and $j$ are in a PS-product-form relationship.
\end{lemma}
\Cref{lem:cut-graph-ps} follows from the fact that each edge in the cut graph represents an S-product-form relationship, as proven in \Cref{lem:cut-implies-product-form},
which chains together along the path to give a PS-product-form relationship.
We can be more specific than simply saying that $i$ and $j$ are in a PS-product-form
relationship.
The arithmetic circuit \add{(see \cref{def:arithmetic-circuit})}
associated with the left-hand side of~\eqref{eq:cut-graph-product-form}
has depth~2 and size \add{(total number of gates)}
\begin{align*}
	1 + d + \sum_{p = 1}^d \left| E \cap \left( \{k_p\} \times A_{k_{p+1}}(G {\setminus} k_p) \right) \right|.
\end{align*}
In particular, if the out-degree of each node on the path is upper-bounded by~$D$,
the size of the arithmetic circuit is upper-bounded by $1 + d + dD$.

\subsection{Higher-level \add{product-form relationships and} cuts}
\label{sec:second-level}

\add{We give a general definition of higher-level product-form relationships and cuts, with a special focus on the second level.}

\paragraph*{\add{Higher-level product-form relationships}}

\add{So far, we have studied S-type and PS-type product-form relationships, as defined in \cref{sec:product-form-def}. We refer to these as ``first-level'' product-form relationships. We now turn to studying higher-level product-form relationships, which we define as follows:
	\begin{description}
		\item[SPS-product-form:] We add another layer of alternation to PS-product-form. Each sum is over neighboring vertices, while the products are over arbitrary vertices.
		We also allow the terms in the products to be the \emph{inverses} of sums, as well as direct sums. 
		Focusing on $f_{i, j}$,
		we say that nodes~$i$ and~$j$ are in an SPS-product-form relationship if
		there exist $S_{i,j} \subseteq N_i$, $F_{k,i,j} \subseteq (V \times \{-1, 1\})$ for each $k \in S_{i,j}$,
		and $S_{a,k,i,j} \subseteq N_a$ for each $(a, p) \in F_{k,i,j}$, such that
		\begin{align*}
			f_{i, j} = \sum_{k \in S_{i,j}} \prod_{(a, p) \in F_{k,i,j}} \Big(\sum_{k' \in S_{a,k,i,j}} q_{a,k'}\Big)^p.
		\end{align*}
		\item[Higher-order:] We can similarly define PSPS-product-form, SPSPS-product-form and more generally (PS)$^n$ and S(PS)$^{n}$ product form for any $n \in \{1, 2, 3, \ldots\}$.
\end{description}}

\paragraph*{\add{Higher-level cuts}}

\add{Corresponding to these higher-level product-form relationships, we also study higher-level cuts.}
Up to this point, we have focused on cuts with a single source vertex on each of side of the cut: $i, j$-sourced cuts, where $i$ and $j$ are each single vertices.
We refer to such cuts as ``first-level'' cuts. These cuts correspond to pairs of states which are joint-ancestor free, and result in S-type \replace{product form}{product-form} relationships between these states.
In addition to these first-level cuts, we are also interested in \emph{second-level} cuts.
We now define such second-level cuts with reference to the cut graph $C_1(G)$ introduced in \Cref{sec:cut-graph}.
In \Cref{lem:second-level-sps} we will prove that
second-level cuts give rise to SPS-product-form relationships.
\Cref{tab:cut-graph-vs-product-form} summarizes this section
by showing what levels of cuts give rise to the \replace{graph-structure product-form}{graph-based product-form} relationships enumerated in \Cref{sec:product-form-def}.

\begin{table}[htb]
	\scalebox{0.92}{%
		\begin{tabular}{|c|c|c|}
			\hline
			Nodes~$i$ and~$j$ are neighbors in $C_1(G)$ & S-product-form & \Cref{theo:s-product-form} \\ \hline
			Nodes~$i$ and~$j$ are connected by a path in $C_1(G)$ & PS-product-form & \Cref{lem:cut-graph-ps} \\ \hline
			Nodes~$i$ and~$j$ are connected by a hyperpath in $C_2(G)$ containing & \multirow{2}{*}{SPS-product-form} & \multirow{2}{*}{\Cref{lem:second-level-sps}} \\
			at most one hyperedge made of more than two nodes & & \\ \hline
			Nodes~$i$ and~$j$ are connected by a hyperpath in $C_2(G)$ & PSPS-product-form & \\ \hline
		\end{tabular}
	}
	\caption{Sufficient conditions under which two distinct states~$i$ and~$j$ of a formal Markov chain $G = (V, E)$ are in a product-form relationship. A hyperpath is defined as a sequence of distinct nodes such that each pair of consecutive nodes in the path is connected by a hyperedge \add{(sets of 2 or more nodes)}.}
	\label{tab:cut-graph-vs-product-form}
\end{table}

\remove{\paragraph*{Second-level cuts}}

We define two kinds of second-level cuts, closely related but subtly different.
\begin{definition}[Second-level cuts]
	Consider a formal Markov chain $G = (V, E)$. \\
	A \textit{broad second-level cut} is a cut with source $(I, J)$ such that the nodes in $I$ are connected to one another via the cut graph $C_1(G)$, and the same is true of $J$. 
	In other words, an \replace{$(I, J)$}{$I, J$}-sourced cut is a broad second-level cut if there exists a pair \replace{$K_1$ and $K_2$}{$(K_1, K_2)$} of connected components of the cut graph $C_1(G)$
	such that $I \subseteq K_1$ and $J \subseteq K_2$. \\
	A \textit{narrow second-level cut} is a cut arising from a joint-ancestor free relationship between two connected components of the cut graph $C_1(G)$.
	Specifically, the narrow second-level cut arising from \replace{two connected component~$K_1$ and $K_2$}{a pair $(K_1, K_2)$ of connected components} of $C_1(G)$ is the cut $(A_{K_1}(G {\setminus} K_2), A_{K_2}(G {\setminus} K_1))$.
\end{definition}

Note that a narrow second-level cut is also a broad second-level cut: its sources must be subsets of $K_1$ and $K_2$, by \Cref{prop:jaf-cuts-2}. The reverse is not as clear\add{: the broad second-level cut with source $(I, J)$, if it exists, is given by $(A_I(G {\setminus} J), A_J(G {\setminus} I))$, and \textit{a priori} we may not have $A_I(G {\setminus} J) = A_{K_1}(G {\setminus} K_2)$ and $A_J(G {\setminus} I) = A_{K_2}(G {\setminus} K_1)$}. Nonetheless, we conjecture that whenever a broad second-level cut exists, a corresponding narrow second-level cut also exists; see \Cref{conj:second-level} later in this section for details.

If a broad second-level cut exists in the graph, generated by $I \subseteq K_1$ and $J \subseteq K_2$, we will show in \Cref{lem:second-level-sps} that this cut gives rise to an \replace{SPS-product form}{SPS-product-form} relationship between any pair of vertices $i \in K_1, j \in K_2$. It follows that, if all connected components of the cut graph are connected via broad second-level cuts, then $G$ exhibits a \replace{PSPS-product form}{PSPS-product-form}; an example appears in \Cref{sec:batch1}. Our primary motivation for introducing narrow second-level cuts is that they can be algorithmically discovered more easily than broad second-level cuts; see the following sub-subsection for details.

One can similarly define broad third-level cuts, fourth-level cuts, and so on, which give rise to S(PS)$^n$ and (PS)$^n$ product-form relationships for larger $n$.
\add{An example of formal Markov chain with cuts of arbitrary levels will appear in \Cref{sec:batch-2}.}
A broad third-level cut is a cut \replace{whose}{such that all nodes within each of the} source sets $I$ and $J$ are connected by a combination of first-level cuts and broad second-level cuts, and so forth.
One can also define narrow third-level cuts with reference to narrow second-level cuts, and so forth.
Intuitively, each additional sum (S) appears by applying a cut equation, and each additional product (P) appears by combining combining several product-form relationships.

\add{The second-level equivalent of the first-level cut graph $C_1(G)$
	is the narrow second-level cut \emph{hypergraph} $C_2(G)$.
	A hypergraph can contain \emph{hyperedges}, which are sets containing 2 or more nodes, expanding the standard notion of 2-node edges.}
We define the narrow second-level cut graph $C_2(G)$ as follows. Starting with the first-level cut graph $C_1(G)$, for each pair of connected components $K_1, K_2$ which form a narrow second-level cut, we add a hyperedge containing the sources of the cut $(A_{K_1}(G {\setminus} K_2), A_{K_2}(G {\setminus} K_1))$.
Assuming \Cref{conj:second-level}, which claims that broad and narrow second-level cuts are equivalent, $C_2(G)$ contains all necessary information to identify all first-level and second-level cuts, and we can characterize the corresponding S, PS, SPS, and PSPS product-form relationships. We can similarly define narrow third-level and higher-level cut graphs. If \Cref{conj:second-level} fails, then we can define a distinct broad second-level cut graph, and higher-level broad cut graphs.

\paragraph*{SPS-product-form}
\label{sec:second-level-product}

We show that if there exists a broad second-level cut with source $(I, J)$, then not only are every pair of vertices in $I$ and $J$ in an SPS-product-form relationship\remove{ (\Cref{def:graph-based-product-form})}, but in fact every pair of vertices in $K_1$ and $K_2$ are in an SPS-product-form relationship, where $K_1$ and $K_2$ are the connected components of the cut graph $C_1(G)$ that include $I$ and $J$, respectively.

\begin{lemma}\label{lem:second-level-sps}
	Consider a formal Markov chain $G = (V, E)$ \add{and two disjoint non-empty sets $I, J \subsetneq V$}.
	\replace{Given}{If there exists} an \replace{$(I, J)$}{$I, J$}-sourced broad second level cut, with $I \subseteq K_1$ and $J \subseteq K_2$, and $K_1, K_2$ connected components of $C_1(G)$,
	then every pair of vertices $i \in K_1$ and $j \in K_2$ are in an SPS-product-form relationship.
\end{lemma}
\begin{proof}
	First, recall from \Cref{lem:cut-graph-ps} that because $K_1$ is a connected component of $C_1(G)$, every pair of vertices $i, i' \in K_1$ is in a PS relationship. Letting $i = k_1, k_2, \ldots k_d, k_{d+1} = i'$ be a path connecting $i$ and $i'$ in $K_1$, a PS-product-form relationship is given by:
	\begin{align}\label{eq:ps-product-form}
		f_{i,i'} &= \prod_{p=1}^d f_{k_p, k_{p+1}}, \quad
		f_{i',i} = \prod_{p=1}^d f_{k_{p+1}, k_p}, \\
		\pi_{i} f_{i, i'} &= \pi_{i'} f_{i', i}.
	\end{align}
	A similar PS-product-form relationship holds for any pair $j, j' \in K_2$.
	
	Next, let us apply \Cref{lem:cut-equations} to the cut with source sets
	$(I, J)$.
	By \Cref{lem:cut-equations}, we have
	\begin{align} \label{eq:source-cut-eqns}
		\sum_{(i,j) \in E \cap (I \times J)} \pi_i q_{i, j}
		= \sum_{(j,i) \in E \cap (J \times I)} \pi_j q_{j, i}.
	\end{align}
	Note that all edges that cross the cut belong to either $I \times J$ or $J \times I$.
	
	Let $i_*$ and $j_*$ be an arbitrary pair of vertices, $i_* \in K_1$
	and $j_* \in K_2$. We will now explicitate the \replace{SPS-product form}{SPS-product-form} relationship between $i_*$ and $j_*$.
	Applying the \replace{PS-product form}{PS-product-form} relationships within $K_1$ and $K_2$,
	given by \eqref{eq:ps-product-form},
	we can rewrite \eqref{eq:source-cut-eqns} in terms of $\pi_{i_*}$
	and $\pi_{j_*}$:
	\begin{align}
		\nonumber
		\sum_{(i,j) \in E \cap (I \times J)} \pi_{i_*} \frac{f_{i_*, i}}{f_{i, i_*}} q_{i, j}
		&= \sum_{(j,i) \in E \cap (J \times I)} \pi_{j_*} \frac{f_{j_*, j}}{f_{j, j_*}} q_{j, i}, \\
		\label{eq:sps-product-form}
		\pi_{i_*} \sum_{(i,j) \in E \cap (I \times J)} \frac{f_{i_*, i}}{f_{i, i_*}} q_{i, j} &= 
		\pi_{j_*} \sum_{(j,i) \in E \cap (J \times I)} \frac{f_{j_*, j}}{f_{j, j_*}} q_{j, i}.
	\end{align}
	This gives the \replace{SPS-product form}{SPS-product-form} relationship between $i_*$ and $j_*$ as desired. Note that we have made use of the flexibility of the \replace{SPS-product form}{SPS-product-form} definition, which allows us to invert the sums within the SPS formula, or equivalently to invert the products within $f_{i, j}$ as defined in \eqref{eq:ps-product-form}.
	Because $i_*$ and $j_*$ were an arbitrary pair of vertices in $K_1, K_2$, this completes the proof.
\end{proof}

\paragraph*{Algorithmic discovery}

We now discuss how to algorithmically and efficiently find all second-level cuts which exist in a given graph $G$, akin to \Cref{algo:cut-graph}, which did the same for first-level cuts.

Specifically, we find all narrow second-level cuts, via the following straightforward but efficient algorithm. We iterate over all pairs $K_1, K_2$ of connected components in the cut graph $C_1(G)$.
For each pair of components, we use \Cref{algo:jaf} to check whether the components are joint-ancestor free, and hence form a narrow second-level cut.

In contrast, attempting to discover broad second-level cuts directly via a similar procedure is not nearly as straightforward, as one may in principle be required to search over all pairs of subsets $I \subseteq K_1, J \subseteq K_2$, which is inefficient.
However, in all cases that we have examined, broad second-level cuts are only present between subsets of components that form narrow second-level cuts. This motivates the following conjecture.

\paragraph*{Conjecture: Equivalence of broad and narrow second-level cuts}

We conjecture that all broad second-level cuts have a corresponding narrow second-level cut, in the sense specified in \Cref{conj:second-level} below.
As a result, we conjecture that the algorithmic procedure described above discovers all components $K_1$ and $K_2$ that contain broad second-level cuts:
If there exist $I \subseteq K_1$ and $J \subseteq K_2$ which form a second-level cut, then $K_1$ and $K_2$ are joint-ancestor free, and the above procedure will discover a narrow second-level cut between $K_1$ and $K_2$, even if that cut's source\remove{s} \replace{are not necessarily $I$ and $J$}{is not necessarily $(I, J)$ or $(K_1, K_2)$}.

\begin{conjecture} \label{conj:second-level}
	Consider a formal Markov chain $G = (V, E)$.
	For each pair of connected components $K_1$ and $K_2$ of $C_1(G)$,
	if there exist two nonempty sets
	$I \subseteq K_1$ and $J \subseteq K_2$
	such that $I$ and $J$ are join\add{t}-ancestor free in~$G$,
	then we conjecture that $K_1$ and $K_2$ are joint-ancestor free in~$G$.
	In other words, if there is a broad second-level cut with source $(I, J)$,
	then we conjecture that there is a narrow second-level cut arising from $K_1$ and $K_2$.
\end{conjecture}
If \Cref{conj:second-level} holds, then we would be able to remove the distinction between broad and narrow second level cuts.

The reason this claim is nontrivial is that in principle, $K_1$ and $K_2$ may have a joint ancestor even if $I$ and $J$ are joint-ancestor free. However, we were not able to build such an example.
Thus, to prove \Cref{conj:second-level}, one must leverage the fact that $K_1$ and $K_2$ are connected components of the cut graph $C_1(G)$.

We explored the following avenue towards proving this conjecture. Given joint-ancestor free subsets $I \subseteq K_1$ and $J \subseteq K_2$, we further conjecture that there always exists \remove{either }an additional node by which either $I$ or $J$ can be expanded, while preserving the joint-ancestor free property:
\begin{conjecture} \label{conj:second-level-expand}
	Given two joint-ancestor free subsets $I \subseteq K_1$ and $J \subseteq K_2$, where either $I \neq K_1, J \neq K_2$, or both,
	there exists either
	\begin{itemize}
		\item $i \in K_1 {\setminus} I$ such that $I \cup \{i\}$ and $J$ are joint-ancestor free, or
		\item $j \in K_2 {\setminus} J$ such that $I$ and $J \cup \{j\}$ are joint-ancestor free.
	\end{itemize}
\end{conjecture}

If \Cref{conj:second-level-expand} holds, then by applying it inductively we show $K_1$ and $K_2$ must be joint-ancestor free in $G$ whenever $I$ and $J$ are, which would complete the proof of \Cref{conj:second-level}.

In Appendix~\ref{app:second-level}, we prove in \Cref{theo:induction-step-1} that in the special case where $J$ contains a single vertex ($|J| = 1$), there exists an $i \in K_1 {\setminus} I$ such that $I \cup \{i\}$ and $J$ are joint-ancestor free. However, we were not able to resolve the general conjecture, either \Cref{conj:second-level-expand} or \Cref{conj:second-level}, and leave both as open problems.

\section{Examples}
\label{sec:examples}

We now give several examples of formal Markov chains arising from queueing systems which exhibit graph-based product form.

\subsection{Multiserver jobs}
\label{sec:msj-example}

First, we give a practical example of a Markov chain with \replace{PS-product form}{PS-product-form} arising from its graph structure. This setting is taken from \cite{GHS23}\remove{.} \replace{The setting is}{and} explored in more detail in a technical report \cite{GHS20}.

Consider the following multiserver-job (MSJ) queueing system.
There are two types of jobs:
class-1 jobs, which each require 3 servers to enter service,
and class-2 jobs, which require 10 servers to enter service.
There are 30 servers in total.
Servers are assigned to jobs in first-come-first-served order,
with head-of-line blocking.
We will specifically consider a \emph{saturated} queueing system,
meaning that fresh jobs are always available, rather than having an external arrival process.
Saturated queueing systems are used to characterize the stability region \cite{BF95,FK04}
and mean response time \cite{GHHS23}
of the corresponding open system with external arrivals.

In the saturated system, fresh jobs are always available,
and a new job enters into the system whenever it is possible that this fresh job might receive service,
based on the number of available servers.
Specifically, if at least 3 servers are available, a fresh job will enter the system.
If that job is a class-2 job,
and if \replace{not 10 servers are available}{there are not 10 servers available},
the class-2 job will wait in the queue.
There will always be at most one job in the queue,
and only a class-2 job can be in the queue.
Said differently, when a job completes service,
we uncover the classes of as many fresh jobs as needed
to verify that no more fresh jobs can be added to service.
The service time of each job is exponentially distributed,
with a rate based on the class of the job:
$\mu_1$ for class-1 jobs,
and $\mu_2$ for class-2 jobs.

There are several Markov chains corresponding to this system
that we could examine at this point.
For instance, we could consider the standard CTMC,
or the embedded DTMC with steps at epochs where jobs complete.
These two Markov chains both exhibit product-form stationary distributions,
but they do not exhibit graph-based product form: the specific transition rates, not just the graph structure, produce the product-form behavior.

However, we instead examine a nonstandard DTMC,
with epochs whenever a job either \emph{enters the system} or completes.
\add{We call this DTMC the \emph{arrival-and-completion DTMC}, in contrast to the more standard \emph{completion-only DTMC} defined in \cite{GHS23}.}
In particular, whenever a job leaves upon service completion,
the DTMC transitions through a the sequence of states,
starting right after the job has completed and no fresh job has \replace{been}{yet} entered the system,
then after a the first fresh job has entered, then the second, and so on,
until there is no possibility that any further fresh jobs could immediately enter service.
This DTMC \emph{does} exhibit graph-based product form, \add{as we will show,}
and this product-form behavior transfers to the other two Markov chains
mentioned above.

\newcommand{\msjobs}{
	\foreach \i in {0, 1, 2} {
		\node[state] (\i) at (\i, 0) {\i};
		\node[phantom] (bar\i) at (\i, -1) {$\bar{\i}$};
	}
	
	\foreach \i in {3, 4, 5} {
		\node[state] (\i) at (\i, -1) {\i};
		\node[phantom] (bar\i) at (\i, -2) {$\bar{\i}$};
	}
	
	\foreach \i in {6, 7, 8, 9} {
		\node[state] (\i) at (\i, -2) {\i};
		\node[phantom] (bar\i) at (\i, -3) {$\bar{\i}$};
	}
	
	\node[state] (10) at (10, -3) {10};
}

\begin{figure}[htb]
	\begin{subfigure}{\linewidth}
		\centering
		\begin{tikzpicture}
			\msjobs
			
			\foreach \i in {1, 2, 4, 5, 7, 8, 9} {
				\pgfmathsetmacro{\j}{\i-1}
				\draw[->] (\i) -- (\j);
			}
			
			\foreach \i in {1, 2, 4, 5, 7, 8, 9} {
				\pgfmathsetmacro{\j}{\i-1}
				\draw[->] (bar\j) -- (bar\i);
			}
			
			\foreach \i in {3, 6, 10} {
				\pgfmathsetmacro{\j}{\i-1}
				\draw[-{Straight Barb[right]}, transform canvas={yshift=.8}]
				(\i) -- (bar\j);
				\draw[-{Straight Barb[right]}, transform canvas={yshift=-.8}]
				(bar\j) -- (\i);
			}
			
			\foreach \i in {0, 1, ..., 6} {
				\draw[-{Straight Barb[right]}, transform canvas={xshift=-.8}]
				(\i) -- (bar\i);
				\draw[-{Straight Barb[right]}, transform canvas={xshift=.8}]
				(bar\i) -- (\i);
			}
			
			\foreach \i in {7, ..., 9} {
				\draw[->] (bar\i) -- (\i);
			}
		\end{tikzpicture}
		\caption{Transition diagram graph $G$.}
		\label{fig:msjob-diagram}
	\end{subfigure} \\
	\begin{subfigure}{\linewidth}
		\centering
		\begin{tikzpicture}
			\msjobs
			
			\foreach \i/\j in {0/1, 1/2, 2/3, 3/4, 4/5, 5/6, 6/7, 7/8, 8/9, 9/10} {
				\draw (\i) -- (bar\i) -- (\j);
			}
			\foreach \i/\j in {7/8, 8/9, bar6/bar7, bar7/bar8, bar8/bar9} {
				\draw (\i) -- (\j);
			}
		\end{tikzpicture}
		\caption{Cut graph~$C_1(G)$.}
		\label{fig:msjob-cut}
	\end{subfigure}
	\caption{Multiserver jobs example.}
	\label{fig:msjob}
\end{figure}
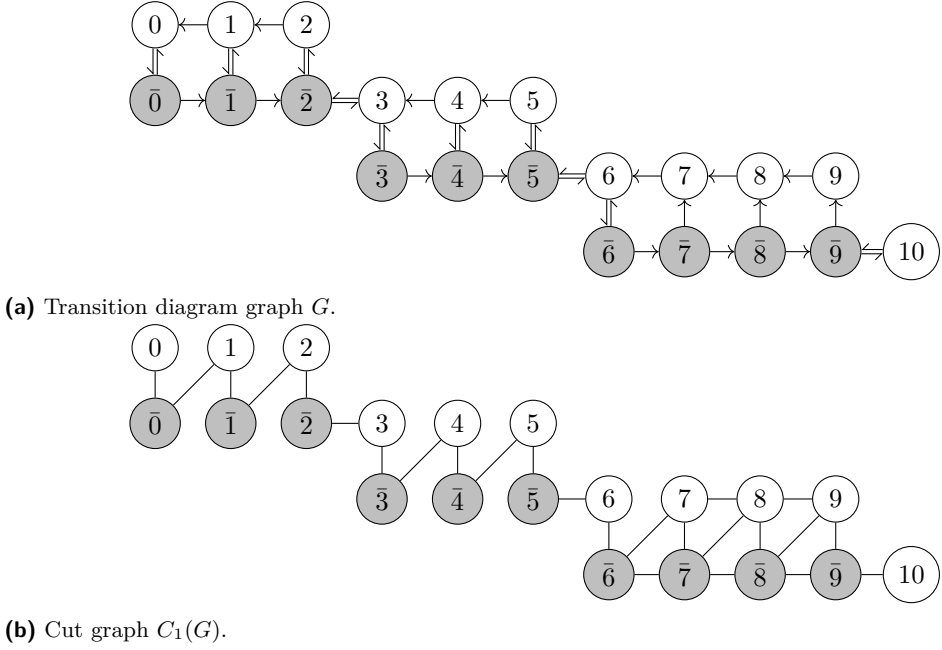

\add{The process of switching back from a DTMC associated with a more fine-grained embedding sequence (e.g., arrivals and completions) to a DTMC associated with a more coarse-grained embedding sequence (e.g., only completions) always preserves product-form relationships between the coarse-grained states, as we can see it as a subchain of the fine-grained DTMC restricted to a subset of the states. See Step~\ref{step2} in the proof of \Cref{theo:no-product-form} for a similar idea. However, switching embeddings may turn graph-based product-form relationships into probability-specific (non-graph-based) product-form relationships, as multiple transitions in the fine-grained DTMC may become a single transition in the coarse-grained DTMC, with transition probability written as a product of transitions in the fine-grained DTMC. This distinction is present in our setting.}

More specifically, to reveal the graph-based product form,
we examine the embedded DTMC of this system, examining moments at which jobs enter the system or complete.
There are two kinds of states in the system:
\emph{completion states},
from which the next transition is a service completion,
and \emph{arrival states},
from which the next transition is a job arrival.
In any state where at least 3 servers are available, and no job is in the queue,
a fresh job enters the system.
That job is a class-1 job with probability $p_1$, and a class-2 job otherwise.
Otherwise, jobs complete.
Class-1 jobs have service rate~$\mu_1$, and class-2 jobs have service rate~$\mu_2$,
resulting in some probability of a completion of each class of job.
After a completion, the job in the queue (if any) moves into service if enough servers are available,
transitioning to a new state.

We denote states by the number of class-1 jobs in the system, from 0 to 10,
and by whether the state is a completion state or an \replace{entering}{arrival} state.
This naming convention is sufficient to differentiate all states in the system (see more details in \cite[Section~4]{GHS23}).
An overbar denotes an arrival state, while the number \replace{alone}{without an overbar} denotes a completion state.
For instance, state $4$ consists of $4$ class-1 jobs and $1$ class-2 job in service, for a total of 22 servers occupied, with another class-2 job in the queue.
State $\bar{4}$ is the same state but without the class-2 job in the queue.
Note that the arrival states such as $\bar{4}$ are present in this arrivals-and-completions embedded DTMC, and would not be present in a more standard completions-only DTMC.
\add{Such a completions-only DTMC can be obtained from the arrivals-and-completions DTMC by combining a completion transition with the immediately following arrival transitions into a single transition in the completions-only DTMC.}

\Cref{fig:msjob-diagram} shows the graph~$G$ underlying the Markov chain for this saturated MSJ queue, with white backgrounds for completion states and grey backgrounds for arrival states.
For instance, starting in state $4$, a class-1 job can complete, transitioning to state $3$,
with $3$ class-1 jobs and $2$ class-2 jobs in service,
or a class-2 job can complete, transitioning to state $\bar{4}$.
From state $\bar{4}$, a class-1 job can enter the system, transitioning to state $\bar{5}$,
with $5$ class-1 jobs and $1$ class-2 job in service,
or a class-2 job can enter, transitioning to state $4$.

\Cref{fig:msjob-cut} shows the (first-level) cut graph~$C_1(G)$ corresponding to this graph.
For instance, there is a first-level cut between nodes $4$ and $\bar{4}$.
This cut partitions the graph \replace{notes}{nodes} into two subsets:
$A_{\bar{4}}(G {\setminus} 4) = \{0, \bar{0}, \ldots, 3, \bar{3}, \bar{4}\}$
and $A_4(G {\setminus} \bar{4}) = \{4, 5, \bar{5}, \ldots, \bar{9}, 10\}$.
The only edges crossing this cut are the outgoing edges from $4$ and $\bar{4}$.

\begin{table}[htb]
	\centering
	\mbox{
		\begin{tabular}{l|l}
			Nodes & Product-form Relation \\
			\hline
			$0$ and $\bar 0$
			& \replace{$\pi(\bar 0) (q_{\bar 0, 0} + q_{\bar 0, 1}) = \pi(0) q_{0, \bar 0}$}{$\pi_{\bar 0} (q_{\bar 0, 0} + q_{\bar 0, \bar 1}) = \pi_0 q_{0, \bar 0}$} \\
			$\bar 0$ and 1
			& \replace{$\pi(1) q_{1, 0} = \pi(\bar 0) q_{\bar 0, \bar 1}$}{$\pi_1 q_{1, 0} = \pi_{\bar 0} q_{\bar 0, \bar 1}$} \\
			$1$ and $\bar 1$
			& \replace{$\pi(\bar 1) (q_{\bar1, 1} + q_{\bar 1, \bar 2}) = \pi(1) (q_{1, 0} + q_{1, \bar 1})$}{$\pi_{\bar 1} (q_{\bar1, 1} + q_{\bar 1, \bar 2}) = \pi_1 (q_{1, 0} + q_{1, \bar 1})$} \\
			$\bar 1$ and $2$
			& \replace{$\pi(2) q_{2, 1} = \pi(\bar 1) q_{\bar 1, \bar 2}$}{$\pi_2 q_{2, 1} = \pi_{\bar 1} q_{\bar 1, \bar 2}$} \\
			$2$ and $\bar 2$
			& \replace{$\pi(\bar 2) q_{\bar 2, 3} = \pi(2) (q_{2, 1} + q_{2, \bar 2})$}{$\pi_{\bar 2} q_{\bar 2, 3} = \pi_2 (q_{2, 1} + q_{2, \bar 2})$} \\
			$\bar 2$ and $3$
			& \replace{$\pi(3) q_{3, \bar 2} = \pi(\bar 2) q_{\bar 2, 3}$}{$\pi_3 q_{3, \bar 2} = \pi_{\bar 2} q_{\bar 2, 3}$} \\
			$3$ and $\bar 3$
			& \replace{$\pi(\bar 3) (q_{\bar 3, 3} + q_{\bar 3, 4}) = \pi(3) q_{3, \bar 3}$}{$\pi_{\bar 3} (q_{\bar 3, 3} + q_{\bar 3, \bar 4}) = \pi_3 q_{3, \bar 3}$} \\
			$\bar 3$ and $4$
			& \replace{$\pi(4) q_{4, 3} = \pi(\bar 3) q_{\bar 3, 4}$}{$\pi_4 q_{4, 3} = \pi_{\bar 3} q_{\bar 3, \bar 4}$}
	\end{tabular}}
	\caption{First-level cuts involving states $i$ or $\bar i$,
		for $i \in \{0, 1, 2, 3\}$, in the multiserver job example.}
	\label{tab:msjob}
\end{table}

\Cref{tab:msjob} lists some of the cuts
and the corresponding product-form relationships between pairs of nodes,
where the transition rate from state~$i$ to state~$j$ is denoted by~$q_{i, j}$.
Note that each first-level cut gives rise to an S-product-form relationship,
defined in \Cref{sec:product-form-def},
as shown in \Cref{lem:cut-implies-product-form}.
Because the cut graph is fully connected, as shown in \Cref{fig:msjob-cut},
every pair of nodes has a PS-product-form relationship, by composing S-product-form relationships
along a path connecting those nodes in the cut graph, as described in \cref{lem:cut-graph-ps}.
As a result, the entire formal Markov chain has a PS-product-form stationary distribution\replace{.}{, as follows:}
\add{\begin{align*}
		\pi_{i} &= \pi_0 \prod_{j=0}^{i-1}
		\frac{q_{j,j-1} \indicator{j \not\in \{0,3,6,10\}} + q_{j,\bar{j}}
		}{q_{\bar{j},j} + q_{\bar{j},\overline{j+1}} \indicator{j \not\in \{2,5,9\}}}
		\frac{q_{\bar{j},\overline{j+1}}}{q_{j+1,j}}, \\
		\pi_{\bar{i}} &= \pi_0 \prod_{j=0}^{i}
		\frac{q_{j,j-1} \indicator{j \not\in \{0,3,6,10\}} + q_{j,\bar{j}}
		}{q_{\bar{j},j} + q_{\bar{j},\overline{j+1}} \indicator{j \not\in \{2,5,9\}}}
		\prod_{j=0}^{i-1} \frac{q_{\bar{j},\overline{j+1}}}{q_{j+1,j}}.
\end{align*}}%
If the three structural parameters of the system are changed (3 servers for \replace{class 1}{class-1} jobs, 10 servers for \replace{class 2}{class-2} jobs, 30 servers total), \replace{PS product form}{PS-product-form} continues to hold.

In prior work \cite{GHS23}, this saturated queueing system was shown to have a PS-product-form stationary distribution, and that stationary distribution was given with respect to the specific transition probabilities $q_{i,j}$ for that system. This work further illuminates the system, demonstrating that the underlying graph of the Markov chain gives rise to the product-form stationary distribution. A similar product form would exist for a system with arbitrary transition rates and the same graph.

\remove{Critical to this \replace{graph-based product-form}{graph-based product form} is the inclusion of the arrival states~$\bar{i}$, in addition to the completion states~$i$. Consider for instance the embedded DTMC with only completion states and no arrival states: arrival states~$\bar{i}$ are deleted, and transitions involving these states are replaced with transitions between completion states~$i$ whose rates take the form of a product. This completion-only embedded DTMC, which is the one studied by \cite{GHS23}, still has a product-form stationary distribution, related to the stationary distribution of our completion-and-arrival DTMC by removing all arrival states and renormalizing over the completion states. In particular, all ratios of stationary probability between completion states are preserved, thus maintaining the product-form distribution. However, this DTMC does not exhibit \replace{graph-based product-form}{graph-based product form}. One can verify that its product-form is ``probability-specific'' in the sense that it is a consequence of the fact that some transition rates in the completions-only DTMC are themselves written as products.}

\subsection{Queue with batch arrivals: \add{version 1}}
\label{sec:batch1}

Next, we give an example of a queueing Markov chain whose associated cut graph is not connected via first-level cuts, but is connected via second-level cuts as discussed in \Cref{sec:second-level}, giving rise to \replace{PSPS-product form}{PSPS-product-form} from its graph structure.
Note that for this particular Markov chain, \Cref{conj:second-level} holds,
so we do not need to distinguish between broad second-level cuts and narrow second-level cuts.

Consider a single-server queueing system with structured batch arrivals. Jobs arrive in batches of size sampled from a geometric distribution, but the batch size is truncated to not bring the total number of jobs in the system above the next multiple of~3. \replace{More specifically}{For instance}, if there were $4$ jobs present prior to an arrival, the number of jobs in queue would be increased to $\min(4 + N, 6)$ where $\mathbb{P}(N = n) = (1 - p) p^{n-1}$ for each $n \in \{1, 2, 3, \ldots\}$,
with the truncation ensuring that the total number of jobs after the batch arrival does not exceed 6.
Jobs are indistinguishable, with exponentially-distributed service times. Batch arrivals occur according to a Poisson process and have i.i.d.\ sizes.

\begin{figure}[htb]
	\begin{subfigure}{\textwidth}
		\centering
		\begin{tikzpicture}
			\batchfirstnodes
			
			\foreach \i/\j in {0/1, 1/2, 2/3, 3/4, 4/5, 5/6, 6/7, 7/8} {
				\draw[->] (\j) -- (\i);
				\draw[->] (\i) -- (bar\j);
			}
			
			\foreach \i/\j in {1/2, 2/3, 4/5, 5/6, 7/8} {
				\draw[->] (bar\i) -- (bar\j);
			}
			
			\foreach \i/\j in {0/1, 1/2, 2/3, 3/4, 4/5, 5/6, 6/7} {
				\draw[->] (bar\j) -- (\j);
			}
		\end{tikzpicture}
		\caption{Transition diagram graph~$G$.}
		\label{fig:batch1-diagram}
	\end{subfigure} \\[.4cm]
	\begin{subfigure}{\textwidth}
		\centering
		\begin{tikzpicture}
			\batchfirstnodes
			
			\foreach \i in {1, 2, ..., 7} {
				\draw (\i) -- (bar\i);
			}
			
			\foreach \i/\j in {3/bar4, 6/bar7} {
				\draw (\i) -- (\j);
			}
			
			\draw (1) -- (0) -- (bar1);
			\draw (3) -- (4) (6) -- (7);
			
			\batchfirstcliques
			
			\batchfirstnodes
			
			\foreach \i in {1, 2, ..., 7} \draw[newedge] (\i) -- (bar\i);
			\foreach \i/\j in {3/bar4, 6/bar7} \draw[newedge] (\i) -- (\j);
			\draw[newedge] (1) -- (0) -- (bar1);
			\draw[newedge] (3) -- (4) (6) -- (7);
		\end{tikzpicture}
		\caption{First-level cut graph $C_1(G)$. Edges (solid lines) connect nodes that are in an S-product-form relationship. Dashed contours outline connected components of~$C_1(G)$.}
		\label{fig:batch1-cut1}
	\end{subfigure} \\[.4cm]
	\begin{subfigure}{\textwidth}
		\centering
		\begin{tikzpicture}
			\batchfirstnodes
			
			\batchfirstcliques
			
			\foreach \i in {1, 2, ..., 7} \draw (\i) -- (bar\i);
			\foreach \i/\j in {3/bar4, 6/bar7} \draw (\i) -- (\j);
			\draw (1) -- (0) -- (bar1);
			\draw (3) -- (4) (6) -- (7);
			
			\foreach \i\j in {1/2, 2/3, 4/5, 5/6} {
				\node[circle, fill, inner sep=1pt]
				(dot) at ($(\i)!.5!(bar\j)$) {};
				\draw[newedge] (dot) -- (\i) (dot) -- (bar\i) (dot) -- (\j);
			}
			
			\batchfirstnodes
		\end{tikzpicture}
		\caption{Second-level cut graph~$C_2(G)$. Solid lines intersecting at a dot represent second-level cuts, showing the source nodes of the cut.}
		\label{fig:batch1-cut2}
	\end{subfigure}
	\caption{Queue with batch arrivals, first version.}
	\label{fig:batch1}
\end{figure}
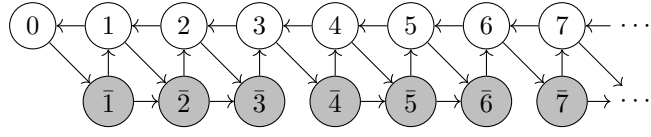
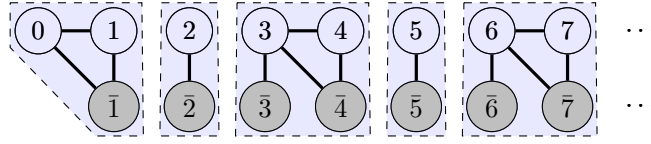
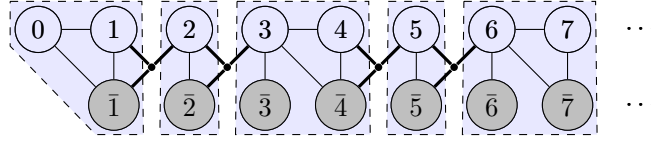

We consider the embedded DTMC of the system, examining moments when jobs enter the system or complete. 
\Cref{fig:batch1-diagram} shows the graph~$G$ underlying the Markov chain for this batch-arrivals system. It is infinite in one direction,
in contrast to the finite graph we considered in \Cref{sec:msj-example}.
In the same spirit as \Cref{sec:msj-example},
there are two kinds of states: during a batch, when a job has just arrived, and in the bulk of time, when jobs may complete or a new batch may arrive. We denote states by the number of jobs in the system, and by whether the state is immediately-post-arrival, shown in grey, or spanning a nonzero amount of time, shown in white.
An overbar denotes an immediately-post-arrival state.
\remove{For instance, state $4$ has four jobs in the system, and jobs may complete or a new batch may begin. 
	State $\bar{7}$ has seven jobs, and a batch of arrivals is ongoing.}
When a batch ends, the system transitions from the immediately-post-arrival state to the corresponding general-time state, such as from state $\bar{2}$ to state $2$.
\add{For instance, state $4$ has four jobs in the system, and jobs may complete or a new batch may begin; state~$\bar{4}$ also has four jobs, but a batch of arrivals is ongoing. If a new batch begins while the Markov chain is in state~$4$, the Markov chain jumps state~$\bar 5$, meaning that the batch contains \emph{at least} one job. Once it is in state~$\bar 5$, the Markov chain either moves to state~$5$ (if the batch is of size 1) or to state~$\bar 6$ (if the batch size is at least 2). In the latter case, the Markov chain necessarily jumps to state~$6$ because batches are truncated.}

\Cref{fig:batch1-cut1} shows the first-level cut graph $C_1(G)$ corresponding to this graph. For instance, there is a first-level cut between nodes $3$ and $\bar{4}$.
This cut partitions the graph into two subsets: $A_{\bar{4}}(G {\setminus} 3) = \{\bar{4}\}$, and $A_3(G {\setminus} \bar{4}) = V {\setminus} \{\bar{4}\}$.
This cut gives rise to an \replace{S-product form}{S-product-form} relationship between nodes $3$ and $\bar{4}$. More first-level cuts and corresponding S-product-form relationships are listed in \Cref{tab:batch1}.

\begin{table}[ht]
	\centering
	\begin{tabular}{c|l|l}
		Level & Nodes & Equation \\
		\hline
		\multirow{11}{.2cm}{1}
		& $0$ and $1$
		& \replace{$\pi(0) q_{0, \bar 1} = \pi(1) q_{1, 0}$}{$\pi_0 q_{0, \bar 1} = \pi_1 q_{1, 0}$} \\
		& $0$ and $\bar 1$
		& \replace{$\pi(0) q_{0, \bar 1} = \pi(\bar 1) (q_{\bar 1, 1} + q_{\bar 1, \bar 2})$}{$\pi_0 q_{0, \bar 1} = \pi_{\bar 1} (q_{\bar 1, 1} + q_{\bar 1, \bar 2})$} \\
		& $1$ and $\bar 1$
		& \replace{$\pi(1) q_{1, 0} = \pi(\bar 1) (q_{\bar 1, 1} + q_{\bar 1, \bar 2})$}{$\pi_1 q_{1, 0} = \pi_{\bar 1} (q_{\bar 1, 1} + q_{\bar 1, \bar 2})$} \\
		& $2$ and $\bar 2$
		& \replace{$\pi(2) q_{2, 1} = \pi(\bar 2) (q_{\bar 2, 2} + q_{\bar 2, \bar 3})$}{$\pi_2 q_{2, 1} = \pi_{\bar 2} (q_{\bar 2, 2} + q_{\bar 2, \bar 3})$} \\
		& $3$ and $\bar 3$
		& \replace{$\pi(3) q_{3, 2} = \pi(\bar 3) (q_{\bar 3, 3} + q_{\bar 3, \bar 2})$}{$\pi_3 q_{3, 2} = \pi_{\bar 3} (q_{\bar 3, 3} + q_{\bar 3, \bar 2})$} \\
		& $3$ and $4$
		& \replace{$\pi(3) q_{3, \bar 4} = \pi(4) q_{4, 3}$}{$\pi_3 q_{3, \bar 4} = \pi_4 q_{4, 3}$} \\
		& $4$ and $\bar 4$
		& \replace{$\pi(4) q_{4, 3} = \pi(\bar 4) (q_{\bar 4, 4} + q_{\bar 4, \bar 3})$}{$\pi_4 q_{4, 3} = \pi_{\bar 4} (q_{\bar 4, 4} + q_{\bar 4, \bar 3})$} \\
		& $5$ and $\bar 5$
		& \replace{$\pi(5) q_{5, 4} = \pi(\bar 5) (q_{\bar 5, 5} + q_{\bar 5, \bar 4})$}{$\pi_5 q_{5, 4} = \pi_{\bar 5} (q_{\bar 5, 5} + q_{\bar 5, \bar 4})$} \\
		& $6$ and $\bar 6$
		& \replace{$\pi(6) q_{6, 5} = \pi(\bar 6) (q_{\bar 6, 6} + q_{\bar 6, \bar 5})$}{$\pi_6 q_{6, 5} = \pi_{\bar 6} (q_{\bar 6, 6} + q_{\bar 6, \bar 5})$} \\
		& $6$ and $7$
		& \replace{$\pi(6) q_{6, \bar 7} = \pi(7) q_{7, 6}$}{$\pi_6 q_{6, \bar 7} = \pi_7 q_{7, 6}$} \\
		& $7$ and $\bar 7$
		& \replace{$\pi(7) q_{7, 6} = \pi(\bar 7) (q_{\bar 7, 7} + q_{\bar 7, \bar 6})$}{$\pi_7 q_{7, 6} = \pi_{\bar 7} (q_{\bar 7, 7} + q_{\bar 7, \bar 6})$} \\
		\hline
		\multirow{4}{.2cm}{2}
		& $\{1, \bar 1\}$ and $2$
		& \replace{$\pi(1) q_{1, \bar 2} + \pi(\bar 1) q_{\bar 1, \bar 2} = \pi(2) q_{2, 1}$}{$\pi_1 q_{1, \bar 2} + \pi_{\bar 1} q_{\bar 1, \bar 2} = \pi_2 q_{2, 1}$} \\
		& $\{2, \bar 2\}$ and $3$
		& \replace{$\pi(2) q_{2, \bar 3} + \pi(\bar 2) q_{\bar 2, \bar 3} = \pi(3) q_{3, 2}$}{$\pi_2 q_{2, \bar 3} + \pi_{\bar 2} q_{\bar 2, \bar 3} = \pi_3 q_{3, 2}$} \\
		& $\{4, \bar 4\}$ and $5$
		& \replace{$\pi(4) q_{4, \bar 5} + \pi(\bar 4) q_{\bar 4, \bar 5} = \pi(5) q_{5, 4}$}{$\pi_4 q_{4, \bar 5} + \pi_{\bar 4} q_{\bar 4, \bar 5} = \pi_5 q_{5, 4}$} \\
		& $\{5, \bar 5\}$ and $6$
		& \replace{$\pi(5) q_{5, \bar 6} + \pi(\bar 5) q_{\bar 5, \bar 6} = \pi(6) q_{6, 5}$}{$\pi_5 q_{5, \bar 6} + \pi_{\bar 5} q_{\bar 5, \bar 6} = \pi_6 q_{6, 5}$}
	\end{tabular}
	\caption{Cuts associated with
		the single-server queue batch example of \Cref{fig:batch1}.}
	\label{tab:batch1}
\end{table}

The connected components of $C_1(G)$, illustrated with dashed outlines in \Cref{fig:batch1-cut1}, each contain two to four vertices. Between these components, there exist second-level cuts, as shown in \Cref{fig:batch1-cut2}.
For instance, between components $K_1 = \{2, \bar{2}\}$ and $K_2 = \{3, \bar{3}, 4, \bar{4}\}$,
there exists a cut with source\remove{s} \add{$(I, J)$, where} $I = K_1 = \{2, \bar{2}\}$ and $J = \{3\}, J \subseteq K_2$. Correspondingly, $K_1$ and $K_2$ are joint-ancestor free.

This cut partitions the graph into two subsets: $A_I(G {\setminus} J) = \{0, 1, \bar{1}, 2, \bar{2}\},$
and $A_J(G {\setminus} I) = \{3, \bar{3}, 4, \bar{4}, 5, \bar{5}, \ldots\}$.
Due to this second-level cut, the graph induces an SPS-product-form relationship between each pair of vertices in $K_1$ and $K_2$, by \Cref{lem:second-level-sps}. Each connected component has a second-level cut with each of its neighbors, as shown in \Cref{fig:batch1-cut2} and in \Cref{tab:batch1}, so the Markov chain has a PSPS-product-form. \add{Every pair of states has a PSPS-product-form relationship. For example, to find a PSPS-relationship between states 1 and 3, we can combine the first-level cuts between states 1 and $\bar{1}$ and between states 2 and $\bar{2}$, with the second level cuts between $\{1, \bar{1}\}$ and $2$ and between $\{2, \bar{2}\}$ and 3, as given in }\cref{tab:batch1}.
\add{By doing so, we obtain the following PSPS-relationship between states 1 and 3:}
\add{\begin{align*}
		\pi_1
		\left(q_{1,\bar{2}} + \frac{q_{1,0} q_{\bar{1},\bar{2}}}{q_{\bar{1},1} + q_{\bar{1},\bar{2}}}\right)
		\left(q_{2,\bar{3}} + \frac{q_{2,1} q_{\bar{2},\bar{3}}}{q_{\bar{2},2} + q_{\bar{2},\bar{3}}}\right)
		=
		\pi_3 q_{2, 1}  q_{3,2}.
\end{align*}}

\replace{The same is true}{The system still has PSPS-product-form} if one varies the parameters defining the system and its Markov chain, for instance by changing the multiplier 3 that truncates the batches to some other multiplier, or truncating the batches at an arbitrary sequence of cutoff values.

\add{As in \Cref{sec:msj-example},
	the addition of the immediately-post-arrival states $\bar i$
	is instrumental in obtaining a \replace{graph-based product-form}{graph-based product form}.
	To see why, consider the embedded DTMC in which the immediately-post-arrival states $\bar i$ are deleted,
	and transitions involving these states are replaced with transitions
	between general-time states whose transition probabilities take the form of a product.
	For instance, this embedded DTMC has a transition from state~$4$ to state~$5$
	with rate $q_{4, \bar 5} q_{\bar 5, 5}$.
	This embedded DTMC still has a product-form stationary distribution.
	However, this product-form follows not only from the graph structure,
	but also from the fact that the arrival transition rates are products.
	Intuitively, by adding the immediately-post-arrival states,
	we encode explicitly in the transition diagram the fact that
	the arrival transitions can be written are products.
	A similar intuition will hold for the example of \Cref{sec:batch-2}.}

\subsection{Queue with batch arrivals: \add{version~2}} \label{sec:batch-2}

Finally, we give an example of a queueing system with a positive number of first-level cuts, second-level, third-level, and so forth, but for which the entire graph is not connected via any finite level of cuts.

Consider a single-server queueing system with \remove{un}structured batch arrivals, of size either 1 or 2. Jobs arrive in batches of size 1 with probability $p_1$ and 2 with probability $p_2$, with $p_1 + p_2 = 1$. Jobs are \add{statistically} indistinguishable, with exponential service time\add{s}. Batch arrivals occur according to a Poisson process.

We examine the embedded DTMC of the system, examining moments when jobs enter the system or complete. We use the same state representation as in \Cref{sec:batch1}.
\Cref{fig:batch2-diagram} shows the graph~$G$ underlying the Markov chain for this batch-arrivals system. \Cref{fig:batch2-cut1} shows the first-level cut graph $C_1(G)$ corresponding to this graph. \Cref{tab:batch2-1} lists some such cuts and the corresponding \replace{S-product form}{S-product-form} relationships.

\newcommand{\batchsecond}{
	\foreach \i in {0, 1, ..., 5} {
		\node[state] (\i) at (\i, 0) {$\i$};
	}
	
	\foreach \i in {1, 3, 5} {
		\node[phantom] (bar\i) at (\i, 1) {$\bar \i$};
	}
	
	\foreach \i in {2, 4} {
		\node[phantom] (bar\i) at (\i, -1) {$\bar \i$};
	}
	
	\node (6) at (6, 0) {$\cdots$};
	\node (bar6) at (6, -1) {$\cdots$};
}

\begin{figure}[htb]
	\begin{subfigure}{.49\textwidth}
		\centering
		\begin{tikzpicture}
			\batchsecond
			
			\foreach \i/\j/\k in {0/1/2, 1/2/3, 2/3/4, 3/4/5, 4/5/6} {
				\draw[->] (\j) -- (\i);
				\draw[->] (\i) -- (bar\j);
				\draw[->] (bar\j) -- (\j);
				\draw[->] (bar\j) -- (\k);
			}
			\draw[->] (6) -- (5);
			\draw[->] (bar6) -- (5);
		\end{tikzpicture}
		\caption{Transition diagram~$G$.}
		\label{fig:batch2-diagram}
	\end{subfigure} \hfill
	\begin{subfigure}{.49\textwidth}
		\centering
		\begin{tikzpicture}
			\batchsecond
			
			\draw[newedge] (1) -- (0) -- (bar1) -- (1) -- (bar2);
			\draw[newedge] (2) -- (bar3);
			\draw[newedge] (3) -- (bar4);
			\draw[newedge] (4) -- (bar5);
			\draw[newedge] (5) -- (bar6);
		\end{tikzpicture}
		\caption{First-level cut graph~$C_1(G)$.}
		\label{fig:batch2-cut1}
	\end{subfigure} \\[.4cm]
	\begin{subfigure}{.49\textwidth}
		\centering
		\begin{tikzpicture}
			\batchsecond
			
			\draw (1) -- (0) -- (bar1) -- (1) -- (bar2);
			\draw (2) -- (bar3);
			\draw (3) -- (bar4);
			\draw (4) -- (bar5);
			\draw (5) -- (bar6);
			
			\node[circle, fill, inner sep=1pt] (dot) at ($(1)!.5!(2)$) {};
			\draw[newedge] (dot) -- (bar1) (dot) -- (2) (dot) -- (bar2);
		\end{tikzpicture}
		\caption{Second-level cut graph $C_2(G)$.}
		\label{fig:batch2-cut2}
	\end{subfigure}
	\begin{subfigure}{.49\textwidth}
		\centering
		\begin{tikzpicture}
			\batchsecond
			
			\draw (1) -- (0) -- (bar1) -- (1) -- (bar2);
			\draw (2) -- (bar3);
			\draw (3) -- (bar4);
			\draw (4) -- (bar5);
			\draw (5) -- (bar6);
			
			\node[circle, fill, inner sep=1pt] (dot) at ($(1)!.5!(2)$) {};
			\draw (dot) -- (bar1) (dot) -- (2) (dot) -- (bar2);
			
			\node[circle, fill, inner sep=1pt] (dot) at ($(2)!.5!(3)$) {};
			\draw[newedge] (dot) -- (bar2) (dot) -- (3) (dot) -- (bar3);
		\end{tikzpicture}
		\caption{Third-level cut graph $C_3(G)$.}
		\label{fig:batch2-cut3}
	\end{subfigure}
	\caption{Queue with batch arrivals, second version.}
	\label{fig:batch2}
\end{figure}
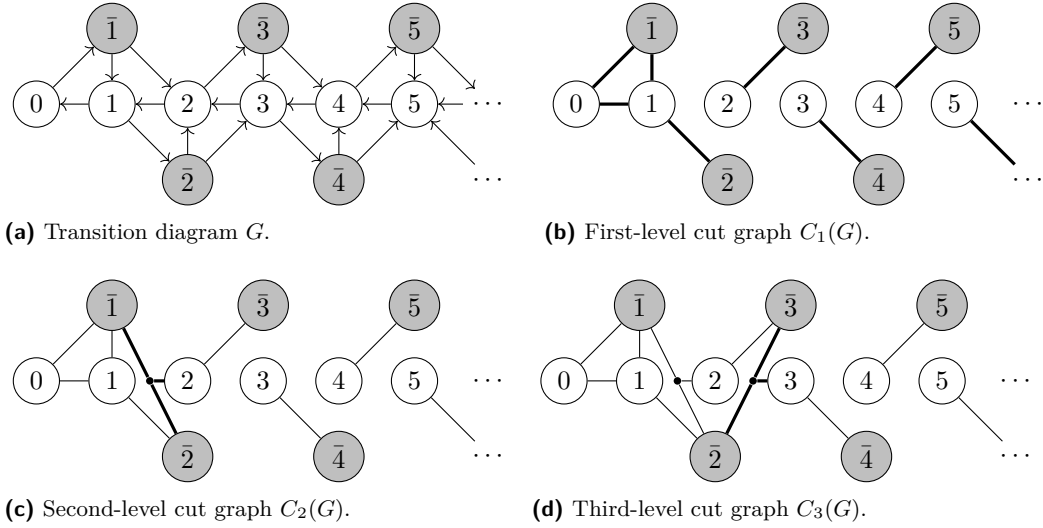

In contrast to the \remove{structured }batch setting in \Cref{sec:batch1}, there are very few second-level cuts in this graph. In fact, the only second level cuts are between subsets of the two leftmost components of the cut graph, namely $K_1 = \{0, 1, \bar{1}, \bar{2}\}$ and $K_2 = \{2, \bar{3}\}$, as shown in \Cref{fig:batch2-cut2}. 
A narrow second-level cut exists between these components, namely $A_{K_1}(G {\setminus} K_2) = K_1$ and $A_{K_2}(G {\setminus} K_1) = G {\setminus} K_1$. This cut has source $(\{\bar{1}, \bar{2}\}, \{2\})$, as shown in \Cref{fig:batch2-cut2}.

This cut implies an SPS product-form relationship between every pair of vertices $i \in K_1$ and $j \in K_2$. The key equation for defining that SPS relationship is the second-level equation in \Cref{tab:batch2-1}.

\begin{table}[b]
	\centering
	\begin{tabular}{c|l|l}
		Level & Nodes & Equation \\
		\hline
		\multirow{6}{.2cm}{1}
		& $0$ and $1$
		& \replace{$\pi(0) q_{0, \bar 1} = \pi(1) q_{1, 0}$}{$\pi_0 q_{0, \bar 1} = \pi_1 q_{1, 0}$} \\
		& $0$ and $\bar 1$
		& \replace{$\pi(0) q_{0, \bar 1} = \pi(\bar 1) (q_{\bar 1, 1} + q_{\bar 1, 2})$}{$\pi_0 q_{0, \bar 1} = \pi_{\bar 1} (q_{\bar 1, 1} + q_{\bar 1, 2})$} \\
		& $1$ and $\bar 1$
		& \replace{$\pi(1) q_{1, 0} = \pi(\bar 1) (q_{\bar 1, 1} + q_{\bar 1, 2})$}{$\pi_1 q_{1, 0} = \pi_{\bar 1} (q_{\bar 1, 1} + q_{\bar 1, 2})$} \\
		& $2$ and $\bar 3$
		& \replace{$\pi(2) q_{2, \bar 3} = \pi(\bar3) (q_{\bar 3, 3} + q_{\bar 3, 4})$}{$\pi_2 q_{2, \bar 3} = \pi_{\bar 3} (q_{\bar 3, 3} + q_{\bar 3, 4})$} \\
		& $3$ and $\bar 4$
		& \replace{$\pi(3) q_{3, \bar 4} = \pi(\bar4) (q_{\bar 4, 4} + q_{\bar 4, 5})$}{$\pi_3 q_{3, \bar 4} = \pi_{\bar 4} (q_{\bar 4, 4} + q_{\bar 4, 5})$} \\
		& $4$ and $\bar 5$
		& \replace{$\pi(4) q_{4, \bar 5} = \pi(\bar5) (q_{\bar 5, 5} + q_{\bar 5, 6})$}{$\pi_4 q_{4, \bar 5} = \pi_{\bar 5} (q_{\bar 5, 5} + q_{\bar 5, 6})$} \\
		\hline
		2 & $\{\bar 1, \bar 2\}$ and $2$
		& \replace{$\pi(\bar 1) q_{\bar 1, 2} + \pi(\bar 2) (q_{\bar 2, 2} + q_{\bar 2, 3}) = \pi(2) q_{2, 1}$}{$\pi_{\bar 1} q_{\bar 1, 2} + \pi_{\bar 2} (q_{\bar 2, 2} + q_{\bar 2, 3}) = \pi_2 q_{2, 1}$} \\
		\hline
		3 & $\{\bar 2, \bar 3\}$ and $3$
		& \replace{$\pi(\bar 2) q_{\bar 2, 3} + \pi(\bar 3) (q_{\bar 3, 3} + q_{\bar 3, 4}) = \pi(3) q_{3, 2}$}{$\pi_{\bar 2} q_{\bar 2, 3} + \pi_{\bar 3} (q_{\bar 3, 3} + q_{\bar 3, 4}) = \pi_3 q_{3, 2}$}
	\end{tabular}
	\caption{Cuts associated with
		the single-queue batch example of \Cref{fig:batch2}.}
	\label{tab:batch2-1}
\end{table}

Note that other subsets of $K_1$ and $K_2$ also form broad second-level cuts, such as $\{\bar{1}, 1\} \subset K_1, \{2\} \subset K_2$. This is the behavior predicted by \Cref{conj:second-level}, which states that broad second level cuts will only exist between subsets of connected components that also have narrow second-level cuts, but that the sources may differ.

We can recursively define third-level cuts based on the connected components of the second-level cut graph. However, as the second-level cut graph only adds a single hyperedge, the only difference between the first-level components and the second-level components is that $K_1$ and $K_2$ are combined into a single component. As result, there is only a single third-level cut, shown in \Cref{fig:batch2-cut3}. We can continue on to higher and higher levels of the cut graph, adding a single cut each time.

Corresponding to the fact that the $n$th cut graph is not fully connected for any finite level $n$, the Markov chain does not exhibit \add{a} (PS)$^n$ or S(PS)$^n$ \replace{product form}{product-form} \add{stationary distribution} for any finite level $n$. However, any two specific nodes $i, j \in V$ are connected in some (potentially large) level of the cut graph, and are in an S(PS)$^n$ product-form relationship for some correspondingly large $n$.

\paragraph*{Acknowledgements}

Thank you to Jean-Michel Fourneau for pointing out \cite{F87} in an informal discussion on queueing systems with product-form stationary distributions.
\add{The authors are also grateful to the two anonymous reviewers for their valuable feedback on an earlier version of the paper; in particular, the alternative proof of \Cref{lem:cut-equations} based on the strong law of large numbers for ergodic Markov chains and the suggestion of \Cref{ex:tree,ex:qbd}.}

\appendix

\section{Cliques in the cut graph} \label{app:clique}

As announced in \add{\Cref{ex:one-way-cycle} and }\Cref{sec:cut-graph},
\Cref{theo:clique} below gives a necessary and sufficient condition
for the existence of a clique in the cut graph
of a formal Markov chain.
This condition can be seen as an extension of \Cref{prop:jaf-cuts-1},
\replace{which corresponds to}{as an edge is} a clique of size~2.
\add{\Cref{fig:clique} shows a toy example
	that will be discussed after the theorem.}

\begin{theorem} \label{theo:clique}
	Consider a formal Markov chain $G = (V, E)$
	and a set $K \subseteq V$ of $n \ge 2$ nodes.
	Also let $V_i = A_i(G {\setminus} (K {\setminus} \{i\}))$ for each $i \in K$,
	and consider the directed graph $Q = (K, L)$ where
	$L = \{(i, j) \in K \times K: i \neq j~\text{and}~E \cap (V_i \times V_j) \neq \emptyset\}$.
	Then:
	\begin{enumerate}[(i)]
		\item \label{item:clique-1} $V = \bigcup_{i \in K} V_i$.
		\item \label{item:clique-2} The following statements are equivalent:
		\begin{enumerate}
			\item \label{item:clique-21}
			$K$ is a clique in $C_1(G)$.
			\item \label{item:clique-22}
			$(V_i)_{i \in K}$ is a partition of~$V$
			and $Q$ is a directed cycle.
		\end{enumerate}
		\item \label{item:clique-3}
		If the equivalent statements of \eqref{item:clique-2} are satisfied
		then, for each $i, j \in K$,
		the $i, j$-sourced cut is $(S, T)$ with
		$S = \bigcup_{k \in A_i(Q {\setminus} j)} V_k$
		and $T = \bigcup_{k \in A_j(Q {\setminus} i)} V_k$.
	\end{enumerate}
\end{theorem}

\Cref{theo:clique} is illustrated in \Cref{fig:clique} with a $9$-node formal Markov chain\add{~G. By definition, for each $i \in K$, $V_i$ is the set of ancestors of node~$i$ in the subgraph of~$G$ obtained by removing all nodes in $K {\setminus} \{i\}$.}
\replace{whose cut graph}{The cut graph of~$G$} contains the clique $K = \{1, 5, 6, 8, 9\}$.
Indeed, condition~\eqref{item:clique-22} of \Cref{theo:clique} is satisfied:
the sets $V_1$, $V_5$, $V_6$, $V_8$, and $V_9$ are disjoint,
and the quotient graph~$Q$
as defined in \Cref{theo:clique} is a cycle
visiting $V_1$, $V_5$, $V_6$, $V_8$, and $V_9$, in this order%
\footnote{
	Note that we identify each subset $V_i$ with its representative vertex $i \in K$. While formally $Q$ is defined with $K$ as its set of vertices,
	we equivalently think of it as having the family $(V_i)_{i \in K}$ as vertices.
}.
Intuitively, the nodes in~$K$ act as \emph{no-return points} in the graph~$G$
in the sense that, once the set $V_i$ has been exited (necessarily via~$i$) for some $i \in K$,
the only way of returning to this set is by traversing the graph~$G$
consistently with the cycle~$Q$.
Saying that the sets $V_1$, $V_5$, $V_6$, $V_8$, and $V_9$ are pairwise disjoint
(and therefore form a partition of~$V$)
is equivalent to saying that node~$i$ is the only exit point from set $V_i$, for each $i \in K$.
Focusing for instance on nodes~$5$ and $8$,
the $5, 8$-sourced cut is given by $(V_5 \cup V_1 \cup V_9, V_8 \cup V_6)$.

\begin{figure}[t]
	\centering
	\begin{tikzpicture}
		\foreach \a in {1, 2, ..., 6} {
			\node[state] (\a) at (180-\a*360/6: 2cm) {\phantom{\a}};
		}
		\node[state] (7) at ($(1)!.5!(2)+(0, .8cm)$) {\phantom{7}};
		\node[state] (8) at ($(1)!.5!(2)+(0, -.8cm)$) {\phantom{8}};
		\node[state] (9) at ($(3)!.5!(4)+(-1.1cm, .5cm)$) {\phantom{9}};
		
		\draw (6) edge[->] (1);
		\draw (6) edge[->] (8);
		\draw (1) edge[->, bend left] (7);
		\draw (1) edge[->] (8);
		\draw (7) edge[->, bend left] (1);
		\draw (7) edge[->] (8);
		\draw (7) edge[->, bend left] (2);
		\draw (8) edge[->] (2);
		\draw (2) edge[->, bend left] (7);
		\draw (2) edge[->] (3);
		\draw (3) edge[->] (4);
		\draw (3) edge[->] (9);
		\draw (9) edge[->, bend left] (4);
		\draw (4) edge[->, bend left] (9);
		\draw (4) edge[->] (5);
		\draw (5) edge[->] (6);
		
		\node at (6) {1};
		\node at (1) {2};
		\node at (7) {3};
		\node at (8) {4};
		\node at (2) {5};
		\node at (3) {6};
		\node at (9) {7};
		\node at (4) {8};
		\node at (5) {9};
		
		\fill[blue, opacity=.4, rounded corners]
		($(1.west)+(-.1cm, 0)$) -- ($(1.west)+(-.1cm, .2cm)$)
		-- ($(7.north)+(-.2cm, .1cm)$) -- ($(7.north)+(.2cm, .1cm)$)
		-- node[midway, right, above, opacity=1] {$V_5$} ($(2.east)+(.1cm, .2cm)$) -- ($(2.east)+(.1cm, -.2cm)$)
		-- ($(8.south)+(.2cm, -.1cm)$) -- ($(8.south)+(-.2cm, -.1cm)$)
		-- ($(1.west)+(-.1cm, -.2cm)$) -- ($(1.west)+(-.1cm, 0)$);
		
		\fill[yellow, semitransparent] (3) circle[radius=.5cm];
		\node[yellow, right] at ($(3.east)+(.2cm, 0)$) {$V_6$};
		
		\fill[red,  opacity=.4, rounded corners]
		($(4.south east)+(0, -.2cm)$) -- ($(4.south east)+(.1cm, -.2cm)$)
		-- node[midway, right, opacity=1] {$V_8$} ($(4.south east)+(.3cm, .2cm)$)
		-- ($(9.north east)+(.1cm, .2cm)$) -- ($(9.north west)+(-.1cm, .2cm)$) -- ($(9.north west)+(-.3cm, -.2cm)$)
		-- ($(4.south west)+(-.1cm, -.2cm)$) -- ($(4.south west)+(0, -.2cm)$);
		
		\fill[orange, semitransparent] (5) circle[radius=.5cm];
		\node[orange, left] at ($(5.west)+(-.2cm, 0)$) {$V_9$};
		
		\fill[green, semitransparent] (6) circle[radius=.5cm];
		\node[green, left] at ($(6.west)+(-.2cm, 0)$) {$V_1$};
	\end{tikzpicture}
	\caption{A formal Markov chain whose cut graph contains the clique $\{1, 5, 6, 8, 9\}$.}
	\label{fig:clique}
\end{figure}
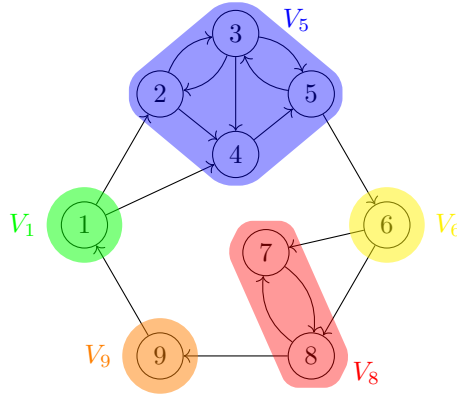

\begin{proof}[Proof of \Cref{theo:clique}]
	We prove each statement one after another.
	
	\medskip
	
	\noindent \eqref{item:clique-1}
	Let $k \in V$.
	Because $G$ is strongly connected,
	there exists a directed path $k_1, k_2, \ldots, k_n$,
	with $k_1 = k$ and $k_n \in K$.
	Then $k \in V_{k_p}$ with $p = \min\{q \in \{1, 2, \ldots, n\} | k_q \in K\}$.
	
	\medskip
	
	\noindent \eqref{item:clique-2}
	We prove each direction of the equivalence separately.
	
	\noindent
	First assume that \eqref{item:clique-21} is satisfied:
	$K$ is a clique in $C_1(G)$.
	By \Cref{prop:jaf-cuts-1},
	it means that $A_i(G {\setminus} j) \cap A_j(G {\setminus} i) = \emptyset$ for each $i, j \in K$.
	We now verify the two parts of \eqref{item:clique-22}:
	\begin{itemize}
		\item $(V_i)_{i \in K}$ is a partition of~$V$:
		For each $i, j \in K$, we have $V_i \cap V_j = \emptyset$ because
		$V_i \subseteq A_i(G {\setminus} j)$,
		$V_j \subseteq A_j(G {\setminus} i)$,
		and $A_i(G {\setminus} j) \cap A_j(G {\setminus} i) = \emptyset$.
		Therefore, $(V_i)_{i \in K}$ is a family of pairwise disjoint sets.
		Combining this with~\eqref{item:clique-1}
		implies that $(V_i)_{i \in K}$ is a partition of~$V$.
		\item $Q$ is a directed cycle:
		$Q$ is strongly connected
		because $G$ is and $(V_i)_{i \in V}$ covers~$V$.
		It remains to be proved that, for each $i \in K$,
		there is at most one~$j \in K {\setminus} \{i\}$
		such that $(i, j) \in L$.
		First observe that, for each $i, j \in K$,
		we have $E \cap (V_i \times V_j) \subseteq \{i\} \times V_j$
		because $V_i \subseteq A_i(G {\setminus} j)$,
		$V_j \subseteq A_j(G {\setminus} i)$, and
		\eqref{item:clique-21} implies that
		$(A_i(G {\setminus} j), A_j(G {\setminus} i))$
		is an $i, j$-sourced cut.
		Now assume for the sake of contradiction that
		there are $j, j' \in K$, with $j \neq j'$,
		such that $(i, j) \in L$ and $(i, j') \in L$.
		Combined with the previous observation,
		it follows there exist $k \in V_j$ and $k' \in V_{j'}$
		such that $(i, k) \in E$ and $(i, k') \in E$.
		Recalling the definitions of $V_j$ and $V_{j'}$,
		we conclude that $i \in A_j(G {\setminus} j') \cap A_{j'}(G {\setminus} j)$,
		By \Cref{prop:jaf-cuts-1},
		this contradicts our assumption that there is a $j, j'$-sourced cut.
	\end{itemize}
	
	\noindent
	Now assume that \eqref{item:clique-22} is satisfied, i.e.,
	$(V_i)_{i, \in K}$ is a partition of~$V$
	and $Q$ is a directed cycle.
	We proceed step-by-step.
	\begin{enumerate}[(A)]
		\item \label{intermediary-1}
		We first verify that,
		for each $(i, j) \in L$, we have
		$E \cap (V_i \times V_j) \subseteq \{i\} \times V_j$,
		i.e., edges in~$G$ from~$V_i$ to~$V_j$ necessarily have source node~$i$.
		Let $(i, j) \in L$.
		Assume for the sake of contradiction that
		there exist $k \in V_i {\setminus} \{i\}$
		and $\ell \in V_j$
		such that $(k, \ell) \in E$.
		Then there exists a path from~$k$ to~$j$
		(through $\ell$) that does not visit
		any node in~$K {\setminus} \{j\}$,
		meaning that $k \in V_j$,
		which is impossible because
		$k \in V_i$ and $V_i \cap V_j = \emptyset$.
	\end{enumerate}
	Now let $i, j \in K$,
	$S = \bigcup_{k \in A_i(Q {\setminus} j)} V_k$,
	and $T = \bigcup_{k \in A_j(Q {\setminus} i)} V_k$.
	Our end goal is to prove that $(S, T)$
	is an $i, j$-sourced cut.
	\begin{enumerate}[(A)]
		\setcounter{enumi}{1}
		\item \label{intermediary-3}
		We know that $(S, T)$ is a partition of~$V$ because
		$(A_i(Q {\setminus} j), A_j(Q {\setminus} i))$
		is a partition of~$K$
		(because~$Q$ is a directed cycle)
		and $(V_k)_{k \in k}$ is a partition of~$V$.
		\item \label{intermediary-2}
		Let us prove that,
		for each $k \in A_i(Q {\setminus} j)$,
		every path in~$G$
		from any node in~$V_k$ to node~$j$ visits node~$i$.
		Let $k \in A_i(Q {\setminus} j)$ and $\ell \in V_k$.
		Consider any path $\ell_1, \ell_2, \ldots, \ell_n$ in~$G$
		such that $\ell_1 = \ell$ and $\ell_n = j$.
		For each $p \in \{1, 2, \ldots, n\}$,
		let $k_p$ denote the unique node in~$K$
		such that $\ell_p \in V_{k_p}$;
		we have in particular $k_1 = k$ and $k_n = j$.
		By definition of~$Q$, for each $p \in \{1, 2, \ldots, n-1\}$,
		we have either $k_p = k_{p+1}$ or $(k_p, k_{p+1}) \in L$,
		i.e., the sequence of distinct nodes
		in $k_1, k_2, \ldots, k_n$ forms a path in~$Q$.
		Since $Q$ is a directed cycle
		and $k \in A_i(Q {\setminus} j)$,
		the only path in~$Q$ from~$k_1 = k$ to~$k_n = j$ visits~$i$.
		Hence, there is~$q \in \{1, 2, \ldots, n-1\}$
		such that $k_q = i$,
		and we can define
		$p = \max\{q \in \{1, 2, \ldots, n\}: k_q = i\}$.
		We know that $p \le n-1$ because $k_n = j \neq i$.
		We obtain $(k_p, k_{p+1}) \in L$,
		so that by~\ref{intermediary-1}
		we have necessarily $\ell_p = i$.
		\item \label{intermediary-4}
		Let us now prove that
		$A_j(G {\setminus} i) \subseteq T$,
		i.e., if $\ell \in S = V {\setminus} T$
		then $\ell \notin A_j(G {\setminus} i)$.
		Let $\ell \in S$.
		By definition of~$S$,
		there is $k \in A_i(Q {\setminus} j)$
		so that $\ell \in V_k$.
		By~\ref{intermediary-2}, every path in~$G$
		from node~$\ell$ to node~$j$ visits node~$i$
		(and such a path exists because $G$ is strongly connected).
		By definition of $A_j(G {\setminus} i)$,
		this means that $\ell \notin A_j(G {\setminus} i)$.
		\item Putting all the pieces together,
		we have that $A_j(G {\setminus} i) \subseteq T$ (by \ref{intermediary-4}),
		$A_i(G {\setminus} j) \subseteq S$ (by symmetry),
		$(S, T)$ is a partition of~$V$ (by \ref{intermediary-3}),
		and $A_i(G {\setminus} j) \cup A_j(G {\setminus} i) = V$
		(by \Cref{lem:jaf-cuts}\ref{prop:ancestors-1}).
		It follows that $A_i(G {\setminus} j) = S$ and $A_j(G {\setminus} i) = T$,
		and that $A_i(G {\setminus} j) \cap A_j(G {\setminus} i) = \emptyset$,
		i.e., nodes~$i$ and~$j$ are joint-ancestor free.
		By \Cref{prop:jaf-cuts-1},
		we conclude that
		$(A_i(G {\setminus} j), A_j(G {\setminus} i)) = (S, T)$
		is the $i, j$-sourced cut.
	\end{enumerate}
	
	\medskip
	
	\noindent \eqref{item:clique-3}
	This is a by-product of the proof of \eqref{item:clique-2}.
\end{proof}

\section{Second-level cuts: Progress towards \Cref{conj:second-level}}
\label{app:second-level}

In \Cref{sec:paths-cut-graph}, we prove intermediary results
that are then applied in \Cref{sec:expanding-case}
to prove \Cref{theo:induction-step-1}.

\subsection{Paths in the cut graph}
\label{sec:paths-cut-graph}

Let us first study the behavior of paths $i_1, i_2, \ldots, i_n$ in the cut graph $C_1(G)$,
where each pair of neighboring vertices in the path is joint-ancestor free.
We are interested in the joint-ancestor freeness of the two terminal nodes in the path,
$i_1$ and $i_n$.
We show that the endpoints $i_1$ and $i_n$ cannot share a joint ancestor in the subgraph graph $G {\setminus} \{i_2, i_3, \ldots i_{n-1}\}$.

First, for the purpose of induction, we prove this claim for length-3 paths:

\begin{lemma} \label{lem:path-3}
	Consider a directed graph $G = (V, E)$
	and three nodes $i_1, i_2, i_3 \in V$
	such that $i_1$ and $i_2$ are joint-ancestor free
	(i.e., $A_{i_1}(G {\setminus} i_2)
	\cap A_{i_2}(G {\setminus} i_1) = \emptyset$)
	and $i_2$ and $i_3$ are joint-ancestor free
	(i.e., $A_{i_2}(G {\setminus} i_3)
	\cap A_{i_3}(G {\setminus} i_2) = \emptyset$).
	Then $A_{i_1}(G {\setminus} \{i_2, i_3\})
	\cap A_{i_3}(G {\setminus} \{i_1, i_2\}) = \emptyset$.
\end{lemma}

\begin{proof}
	Assume for the sake of contradiction that
	$A_{i_1}(G {\setminus} \{i_2, i_3\})
	\cap A_{i_3}(G {\setminus} \{i_1, i_2\}) \neq \emptyset$,
	and let $k \in A_{i_1}(G {\setminus} \{i_2, i_3\})
	\cap A_{i_3}(G {\setminus} \{i_1, i_2\})$.
	We will prove that either
	$k \in A_{i_2}(G {\setminus} i_1)$
	or $k \in A_{i_2}(G {\setminus} i_3)$,
	which will contradict our assumption that
	$i_1$ and $i_2$ are joint-ancestor free
	and $i_1$ and $i_3$ are joint-ancestor free.
	Since $G$ is strongly connected,
	there is a path $p = p(k \to i_2)$.
	We now make a case disjunction.
	
	First, suppose that $p(k \to i_2)$ either does not visite node~$i_1$ or does not visit node~$i_3$.
	If $p(k \to i_2)$ does not visit node $i_1$, then $k \in A_{i_2}(G {\setminus} i_1)$.
	Since $k \in A_{i_1}(G {\setminus} \{i_2, i_3\}) \subseteq A_{i_1}(G {\setminus} i_2)$, this contradicts our assumption that nodes $i_1$ and $i_2$ are joint-ancestor free.
	If $p(k \to i_2)$ does not visit node $i_3$, then $k \in A_{i_2}(G {\setminus} i_3)$, which leads to a similar contradiction regarding nodes~$i_2$ and~$i_3$.
	
	On the other hand, suppose that $p(k \to i_2)$ visits both $i_1$ and $i_3$.
	Let us consider~$p'$, the last portion of~$p$ beginning at the last visit to either $i_1$ or $i_3$.
	Without loss of generality, suppose $p'$ begins at $i_1$
	and reaches $i_2$ without visiting $i_3$.
	Since we assumed that $k \in A_{i_1}(G {\setminus} \{i_2, i_3\})$,
	concatenating $p(k \to i_1 {\setminus} \{i_2, i_3\})$ and $p'$
	gives us a path from $k$ to $i_2$ without visiting $i_3$.
	So $k \in A_{i_2}(G {\setminus} i_3)$,
	which as explained before leads to a contradiction.
	
	In every case, we have a contradiction to either our assumption that $i_1$ and $i_2$ are joint-ancestor free, or that $i_2$ and $i_3$ are joint-ancestor free.
	Thus, our initial assumption much be wrong: There is no node $k$ in $A_{i_1}(G {\setminus} \{i_2, i_3\})
	\cap A_{i_3}(G {\setminus} \{i_1, i_2\})$, as desired. \qedhere
\end{proof}

Next, we build inductively on this result to handle paths of arbitrary lengths.

\begin{lemma} \label{lem:path-n}
	Consider a directed graph $G = (V, E)$
	and a sequence of $n \ge 3$ distinct nodes
	$i_1, i_2, \ldots, i_n$
	such that $i_p$ and $i_{p+1}$
	are joint-ancestor free
	for each $p \in \{1, 2, \ldots, n-1\}$.
	Then $A_{i_1}(G {\setminus} \{i_2, i_3, \ldots, i_n\})
	\cap A_{i_n}(G {\setminus} \{i_1, i_2, \ldots, i_{n-1}\})
	= \emptyset$.
\end{lemma}

\begin{proof}
	We make a proof by induction over~$n$,
	using \Cref{lem:path-3} as a base case for $n = 3$.
	
	Let $n \ge 4$ and assume the induction assumption is true for each $p \in \{3, 4, \ldots, n-1\}$.
	Assume for the sake of contradiction that
	the conclusion does not hold,
	i.e., there exist $k \in V$ and two paths
	$p_{(a)} = p(k \to i_1 {\setminus} \{i_2, i_3, \ldots, i_n\})$
	and $p_{(b)} = p(k \to i_n {\setminus} \{i_1, i_2, \dots, i_{n-1}\})$.
	
	Since $G$ is strongly connected,
	we know that at least one of the following is true:
	there is a path $p_{(1)} = p(i_1 \to i_{n-1} {\setminus} i_n)$,
	or there is a path $p_{(2)} = p(i_n \to i_{n-1} {\setminus} i_1)$.
	To see why, consider a path $x$ from $i_1$ to $i_{n-1}$. If the path $x$ avoids $i_n$, $x$ is $p_{(1)}$, and we have the first case. If the path visits $i_n$, the tail of the path $x$ starting at its visit to $i_n$ is $p_{(2)}$, satisfying the second case.
	
	Let us consider each case in turn:
	\begin{itemize}
		\item Suppose there is a path $p_{(1)} = p(i_1 \to i_{n-1} {\setminus} i_n)$.
		In this case, concatenating
		$p_{(a)}$
		and $p_{(1)}$
		gives us a path
		$p_{(3)} = p(k \to i_{n-1} {\setminus} i_n)$.
		The existence of paths $p_{(3)}$ and~$p_{(b)}$
		implies that that
		$k \in A_{i_{n-1}}(G {\setminus} i_n) \cap A_{i_n}(G {\setminus} i_{n-1})$,
		which contradicts our assumption that
		$i_{n-1}$ and $i_n$ are joint-ancestor free.
		\item On the other hand, suppose that there is a path $p_{(2)} = p(i_n \to i_{n-1} {\setminus} i_1)$.
		Let $q \in \{2, 3, \ldots, n-1\}$
		so that $i_q$ is the first node in $p_{(2)}$ that belongs to $\{i_2, i_3, \ldots, i_{n-1}\}$.
		This provides us with a path
		$p_{(4)} = p(i_n \to i_q {\setminus} \{i_1, \ldots, i_{q-1}, i_{q+1}, \ldots, i_{n-1}\})$
		Concatenating $p_{(b)}$
		and $p_{(4)}$
		gives us a path
		$p_{(5)} = p(k \to i_q {\setminus} \{i_1, \ldots, i_{q-1}, i_{q+1}, \ldots, i_{n-1}\})$.
		The existence of paths $p_{(5)}$ and $p_{(a)}$ implies that
		$k \in A_{i_1}(G {\setminus} \{i_2, i_3, \ldots, i_{q-1}\})
		\cap A_{i_q}(G {\setminus} \{i_1, \ldots, i_{q-1}\})$.
		If $q = 2$, this contradicts our assumption
		that $i_1$ and $i_2$ are joint-ancestor free.
		If $q \ge 3$, this contradicts the induction assumption.\qedhere
	\end{itemize}
\end{proof}

\subsection{Special case: Expanding $I$ when $|J|=1$}
\label{sec:expanding-case}

Now building on \Cref{lem:path-n},
we are ready to prove \Cref{theo:induction-step-1},
which we see as a stepping stone to prove \Cref{conj:second-level},
that any joint-ancestor free subsets $I \subseteq K_1$ and $J \subseteq K_2$
can be expanded to joint-ancestor freeness of the entire connected components of the cut graph, $K_1$ and $K_2$.
In other words, that any broad second-level cut gives rise to a narrow second-level cut.
Here, we only focus on the case of expanding $I$ when $J$ is a single node $\{j\}$.

\begin{theorem} \label{theo:induction-step-1}
	Consider a directed graph $G = (V, E)$.
	Let $K_1$ and $K_2$ denote two connected components of $C_1(G)$.
	Assume that there is a nonempty strict subset $I$ of $K_1$
	and a vertex $j \in K_2$
	such that $I$ and $j$ are joint-ancestor free (in G).
	Then there exists $i \in K_1 {\setminus} I$
	such that $I \cup \{i\}$ and $j$
	are also joint-ancestor free.
\end{theorem}

\begin{proof}
	First, because $I$ and $j$ are joint-ancestor free
	and because $G$ is strongly connected,
	there is $\ell \in A_I(G {\setminus} j)$
	such that $(j, \ell) \in E$:
	In particular, there must be a path from $j$ to some node in $I$,
	and we may take $\ell$ to be the second node on that path, after $j$.
	
	Next, because $\ell \in A_I(G {\setminus} j)$,
	there is a path $p(\ell \to i' {\setminus} j)$
	for some $i' \in I$.
	In particular, by ending when the path first enters $I$,
	there must exist a path
	$p_{(1)}(\ell \to i' {\setminus} (\{j\} \cup I {\setminus} \{i'\})$
	for some $i' \in I$.
	
	Now, we will switch to viewing paths in the cut graph $C_1(G)$.
	In particular, we will think about $I$ and $K_1$,
	the connected component of $C_1(G)$ that $I$ lies within.
	Because $K_1$ is a connected component of $C_1(G)$,
	there is a path in the cut graph going from $i'$ to an arbitrary node $i \in K_1 {\setminus} I$.
	In particular, there is a path in the cut graph that
	stays within $I$ until it visits $i$ as the first node in the path outside of $I$ and in $K_1$.
	In other words, this is a cut graph path in $I \cup \{i\}$.

	Now, assume for the sake of contradiction that
	$I \cup \{i\}$ and $j$
	are not joint-ancestor free,
	i.e., there exists
	$k \in A_{I \cup \{i\}}(G {\setminus} j)
	\cap A_j(G {\setminus} (I \cup \{i\}))$.
	Because $I$ and $j$ are joint-ancestor free,
	we necessarily have
	$k \in A_i(G {\setminus} j) \cap A_j(G {\setminus} (I \cup \{i\}))$.
	Hence, there exists a path (in $G$)
	$p(k \to i {\setminus} j)$
	and a path $p(k \to j {\setminus} (I \cup \{i\})$.

	Now, note that the path $p(k \to i {\setminus} j)$ does not visit $I$,
	so it is also a path $p(k \to i {\setminus} I)$.
	To see why, note that if this path did visit $I$,
	then taking the portion from $k$
	to the first visit to $I$
	would give a path $p(k \to I {\setminus} j)$.
	But because we also have a path
	$p(k \to j {\setminus} (I \cup \{i\}))$,
	it follows that
	$k \in A_I(G {\setminus} j) \cap A_j(G {\setminus} I)$,
	which contradicts our assumption
	that $I$ and $j$ are joint-ancestor free.
	
	Now, let us return to the path $p_{(1)}(\ell \to i' {\setminus} (\{j\} \cup I {\setminus} \{i'\}))$.
	We claim that the path $p_{(1)}(\ell \to i' {\setminus} (\{j\} \cup I {\setminus} \{i'\}))$ does not visit node $i$,
	so in particular it is a path
	$p_{(1)}(\ell \to i' {\setminus} (I \cup \{i\} {\setminus} \{i'\})$.
	To see why, note that if $p_{(1)}$ did visit $i$ prior to visiting $i'$,
	then by taking the portion of the path just after visiting $i$,
	there is a path $p(i \to I {\setminus} j)$.
	Concatenating this path with the $p(k \to i {\setminus} j)$ path mentioned previously, we now have a path $p(k \to I {\setminus} j)$.
	Thus, $k$ is a joint ancestor of $I$ and $j$, contradicting our assumption.
	
	Now, we're ready to put it all together.
	Concatenating
	the path $p(k \to j {\setminus} (I \cup \{i\})$,
	the edge $(j, \ell)$,
	and then to the path
	$p_{(1)}(\ell \to i' {\setminus} (I \cup \{i\} {\setminus} \{i'\})$
	yields a path
	$p(k \to i' {\setminus} (I \cup \{i\} {\setminus} \{i'\})$.
	Therefore, we have a path
	$p(k \to i {\setminus} I)$
	and a path $p(k \to i' {\setminus} (I \cup \{i\} {\setminus} \{i'\})$.
	By \Cref{lem:path-n},
	this contradicts the fact that
	there is a path between~$i$ and~$i'$
	through~$I$ in the cut graph $C_1(G)$.
	
	Thus, our assumption was false, and $I \cup \{i\}$ and $j$ are joint-ancestor free, as desired.
\end{proof}

\bibliography{paper}

\begin{thebibliography}{10}

\bibitem{B73}
Jeffrey~P. Buzen.
\newblock Computational algorithms for closed queueing networks with
  exponential servers.
\newblock {\em Communications of the {ACM}}, 16(9):527--531, September 1973.

\bibitem{R80}
M.~Reiser and S.~S. Lavenberg.
\newblock Mean-value analysis of closed multichain queuing networks.
\newblock {\em Journal of the {ACM}}, 27(2):313--322, April 1980.

\bibitem{K21}
L.~R. van Kreveld, O.~J. Boxma, J.~L. Dorsman, and M.~R.~H. Mandjes.
\newblock Scaling limits for closed product-form queueing networks.
\newblock {\em Performance Evaluation}, 151:102220, November 2021.

\bibitem{C24-1}
Ellen Cardinaels, Sem Borst, and Johan S.~H. van Leeuwaarden.
\newblock Heavy-traffic universality of redundancy systems with assignment
  constraints.
\newblock {\em Operations Research}, 72(4):1539--1555, July 2024.

\bibitem{C24-2}
Ellen Cardinaels, Sem Borst, and Johan~S.H. van Leeuwaarden.
\newblock Multi-dimensional state space collapse in non-complete resource
  pooling scenarios.
\newblock {\em Proceedings of the ACM on Measurement and Analysis of Computing
  Systems}, 8(2):24:1--24:52, May 2024.

\bibitem{A24}
Jonatha Anselmi, Bruno Gaujal, and Louis-Sébastien Rebuffi.
\newblock Learning optimal admission control in partially observable queueing
  networks.
\newblock {\em Queueing Systems}, 108(1):31--79, October 2024.

\bibitem{C24}
Céline Comte, Matthieu Jonckheere, Jaron Sanders, and Albert Senen-Cerda.
\newblock Score-aware policy-gradient and performance guarantees using local
  lyapunov stability.
\newblock {\em Journal of Machine Learning Research}, 26(132):1--74, 2025.

\bibitem{J57}
James~R. Jackson.
\newblock Networks of waiting lines.
\newblock {\em Operations Research}, 5(4):518--521, 1957.

\bibitem{BCMP75}
Forest Baskett, K.~Mani Chandy, Richard~R. Muntz, and Fernando~G. Palacios.
\newblock Open, closed, and mixed networks of queues with different classes of
  customers.
\newblock {\em Journal of the ACM}, 22(2):248–260, apr 1975.

\bibitem{serfozo}
Richard Serfozo.
\newblock {\em Introduction to stochastic networks}.
\newblock Springer, 1999.

\bibitem{BKK95}
S.~A. Berezner, C.~F. Kriel, and A.~E. Krzesinski.
\newblock Quasi-reversible multiclass queues with order independent departure
  rates.
\newblock {\em Queueing Systems}, 19(4):345--359, December 1995.

\bibitem{K11}
A.~E. Krzesinski.
\newblock Order independent queues.
\newblock In Richard~J. Boucherie and Nico M.~van Dijk, editors, {\em Queueing
  networks: A fundamental approach}, number 154 in International Series in
  Operations Research \& Management Science, pages 85--120. Springer {US},
  2011.

\bibitem{A82}
Enrique~Daniel Andjel.
\newblock Invariant measures for the zero range process.
\newblock {\em The Annals of Probability}, 10(3), August 1982.

\bibitem{kelly}
F.~P. Kelly.
\newblock {\em Reversibility and stochastic networks}.
\newblock Cambridge University Press, 2011.

\bibitem{GHS23}
Isaac Grosof, Mor Harchol-Balter, and Alan Scheller-Wolf.
\newblock New stability results for multiserver-job models via product-form
  saturated systems.
\newblock {\em {ACM} {SIGMETRICS} Performance Evaluation Review}, 51(2):6--8,
  October 2023.

\bibitem{BK96}
S.~A. Berezner and A.~E. Krzesinski.
\newblock Order independent loss queues.
\newblock {\em Queueing Systems}, 23(1):331--335, March 1996.

\bibitem{CD21}
Céline Comte and Jan-Pieter Dorsman.
\newblock Pass-and-swap queues.
\newblock {\em Queueing Systems}, 98(3):275--331, August 2021.

\bibitem{GR20}
Kristen Gardner and Rhonda Righter.
\newblock Product forms for {FCFS} queueing models with arbitrary server-job
  compatibilities: an overview.
\newblock {\em Queueing Systems}, 96(1):3--51, October 2020.

\bibitem{ABDV22}
Urtzi Ayesta, Tejas Bodas, Jan-Pieter L.~Dorsman, and Ina~Maria Verloop.
\newblock A token-based central queue with order-independent service rates.
\newblock {\em Operations Research}, 70(1):545--561, 2022.

\bibitem{RM17}
Alexander Rumyantsev and Evsey Morozov.
\newblock Stability criterion of a multiserver model with simultaneous service.
\newblock {\em Annals of Operations Research}, 252(1):29--39, May 2017.

\bibitem{F87}
Brion~N. Feinberg and Samuel~S. Chiu.
\newblock A method to calculate steady-state distributions of large {Markov}
  chains by aggregating states.
\newblock {\em Operations Research}, 35(2):282--290, April 1987.

\bibitem{GHS20}
Isaac Grosof, Mor Harchol-Balter, and Alan Scheller-Wolf.
\newblock Stability for two-class multiserver-job systems.
\newblock {\em arXiv preprint arXiv:2010.00631}, 2020.

\bibitem{H93}
Peter~G. Harrison and Naresh~M. Patel.
\newblock {\em Performance modelling of communication networks and computer
  architectures}.
\newblock Wokingham, England ; Reading, Mass. : Addison-Wesley, 1993.

\bibitem{L99}
G.~Latouche and V.~Ramaswami.
\newblock {\em Introduction to Matrix Analytic Methods in Stochastic Modeling}.
\newblock {ASA}-{SIAM} Series on Statistics and Applied Mathematics. Society
  for Industrial and Applied Mathematics, jan 1999.

\bibitem{P11}
Juan~F. Pérez and Benny Van~Houdt.
\newblock Quasi-birth-and-death processes with restricted transitions and its
  applications.
\newblock 68(2):126--141, February 2011.

\bibitem{A09}
Manindra Agrawal and Ramprasad SaptharishiI.
\newblock Classifying polynomials and identity testing.
\newblock Indian Academy of Sciences, 2009.

\bibitem{K15}
Swastik Kopparty, Shubhangi Saraf, and Amir Shpilka.
\newblock Equivalence of polynomial identity testing and polynomial
  factorization.
\newblock {\em Computational Complexity}, 24(2):295--331, June 2015.

\bibitem{D92}
Susanna Donatelli and Matteo Sereno.
\newblock On the product form solution for stochastic {Petri} nets.
\newblock In K.~Jensen, editor, {\em Application and Theory of Petri Nets
  1992}, pages 154--172. Springer, 1992.

\bibitem{B12}
Simonetta Balsamo, Peter~G. Harrison, and Andrea Marin.
\newblock Methodological construction of product-form stochastic {Petri} nets
  for performance evaluation.
\newblock {\em Journal of Systems and Software}, 85(7):1520--1539, July 2012.

\bibitem{H87}
Moshe Haviv.
\newblock Aggregation/disaggregation methods for computing the stationary
  distribution of a {Markov} chain.
\newblock {\em {SIAM} Journal on Numerical Analysis}, 24(4):952--966, August.

\bibitem{SY10}
Amir Shpilka and Amir Yehudayoff.
\newblock Arithmetic circuits: A survey of recent results and open questions.
\newblock {\em Foundations and Trends® in Theoretical Computer Science},
  5(3–4):207--388, 2010.

\bibitem{B20}
Pierre Brémaud.
\newblock {\em Markov Chains: Gibbs Fields, Monte Carlo Simulation and Queues},
  volume~31 of {\em Texts in Applied Mathematics}.
\newblock Springer International Publishing.

\bibitem{BF95}
François Baccelli and Serguei Foss.
\newblock On the saturation rule for the stability of queues.
\newblock {\em Journal of Applied Probability}, 32(2):494–507, 1995.

\bibitem{FK04}
Serguei Foss and Takis Konstantopoulos.
\newblock An overview of some stochastic stability methods.
\newblock {\em Journal of the Operations Research Society of Japan},
  47(4):275--303, 2004.

\bibitem{GHHS23}
Isaac Grosof, Yige Hong, Mor Harchol-Balter, and Alan Scheller-Wolf.
\newblock The {RESET} and {MARC} techniques, with application to
  multiserver-job analysis.
\newblock {\em Performance Evaluation}, 162:102378, 2023.

\end{thebibliography}

\end{document}